\documentclass[11pt]{article}
\usepackage{a4wide}
\usepackage{amsmath, amsthm, xcolor}
\usepackage{amssymb}
\usepackage{graphicx,soul}
\usepackage[normalem]{ulem}
\usepackage{relsize}

\newcommand{\hidden}[1]{}

\usepackage{color}
\usepackage[mathscr]{euscript}
\usepackage{tikz}

\newcommand{\R}{{\Bbb R}}

\newcommand{\Z}{{\Bbb Z}}
\newcommand{\N}{{\Bbb N}}

\newcommand{\D}{{\Bbb D}}

\newcommand{\bE}{\mathbb{E}}
\newcommand{\T}{{\Bbb T}}

\newcommand{\slT}{{\beta(T)}}

\newcommand{\dist}{{\rm dist}}
\newcommand{\s}{{\small{$\Sigma$}}}
\newcommand{\Bad}{{\bf Bad}}

\newcommand{\dd}{\mathrm{d}}

\newtheorem{theorem}{Theorem}
\newtheorem{proposition}{Proposition}

\newtheorem{corollary}{Corollary}
\newtheorem{lemma}{Lemma}

\newtheorem{thmF}{Theorem FMP$^+$ \!\!\!\!}

\newtheorem{definition}{Definition}
\theoremstyle{remark}
\newtheorem{remark}{{Remark}}

\def\ep{ \epsilon }

\newcommand{\Bni}{B_{i,n}}

\newcommand{\vx}{\mathbf{x}}

\newcommand{\va}{\mathbf{a}}
\newcommand{\vj}{\mathbf{j}}

\newcommand{\cM}{\mathcal{M}}
\newcommand{\cC}{\mathcal{C}}

\newcommand{\cR}{\mathcal{R}}

\newcommand{\cP}{\mathcal{P}}
\newcommand{\cH}{\mathcal{H}}
\newcommand{\UP}{\mathcal{U}(\Psi)}
\newcommand{\cB}{\mathcal{B}}

\newcommand{\slv}[1]{{\color{purple}  {#1}}}

\begin{document}

\title{The Shrinking Target Problem for Matrix Transformations \\ of Tori: revisiting the standard problem}

%\title{Appendix to the Shrinking Target Problem for Matrix Transformations of Tori: revisted}

 \author{ Bing Li \\ {\small\sc (SCUT) } \and Lingmin Liao \\ {\small\sc (Wuhan) }  \and
Sanju Velani \\ {\small\sc (York) }
 \and ~ Evgeniy Zorin \\ {\small\sc (York)}}

\bigskip

\date{\emph{``You've always had the power my dear, \\ you just had to learn it yourself''}}

 \maketitle

\begin{abstract}
	Let $T$ be a $d\times d$ matrix with real coefficients.
	Then $T$ determines a self-map of the $d$-dimensional torus
	$\T^d=\R^d/\Z^d$.   Let $ \{E_n \}_{n \in \mathbb{N}} $ be a sequence of
	subsets  of $\T^d$ and  let $W(T,\{E_n \})$ be the set of points $\vx
	\in \T^d$ such that
	$T^n(\vx)\in E_n $ for infinitely many $n\in\N$.  For a large class of
	subsets (namely, those satisfying the so called bounded property
	$ ({\boldsymbol{\rm B}}) $ which includes balls, rectangles, and
	hyperboloids) we show that the $d$-dimensional Lebesgue measure of the
	shrinking target set $W(T,\{E_n \})$  is  zero (resp. one) if a natural
	volume sum converges (resp. diverges).  In fact,  we prove a
	quantitative form of this zero-one criteria that describes the
	asymptotic behaviour of the counting function $R(x,N):= \# \big\{ 1\le n
	\le  N :    T^{n}(x) \in E_n \} $.  The counting result makes use of a
	general quantitative statement that holds for a large class
	measure-preserving dynamical systems (namely, those satisfying the so
	called summable-mixing property).   We next turn our attention to the
	Hausdorff dimension of $W(T,\{E_n \})$.  In the case the subsets $E_n$
	are balls, rectangles or hyperboloids we obtain precise formulae for the
	dimension. These shapes correspond, respectively,  to the  simultaneous,
	weighted and multiplicative theories of classical  Diophantine
	approximation.  The dimension results for balls generalises those
	obtained in \cite{hvMat} for integer matrices to real matrices. In the final section, we discuss various problems that stem from the results proved in the paper.
\end{abstract}

\bigskip

\section{Introduction}

Let $(X,\mathcal{B},\mu,T)$ be a measure-preserving dynamical system.  Recall, that by definition $\mu$ is a probability measure. Now let $\{E_n \}_{n \in \mathbb{N}} $ be a sequence of subsets  in $\mathcal{B}$ and let
\begin{eqnarray*}
W\big(T,\{E_n\}\big) & := & \limsup_{n\to\infty}T^{-n}(E_n)   \\[1ex] & = &  \{x\in   X: T^n(x)\in E_n\ \hbox{ for infinitely many }n\in \mathbb{N}\} \, .
\end{eqnarray*}
For obvious reasons the sets $E_n$ can be thought of as targets that the orbit under $T$ of points in $X$ have to hit. The interesting situation is usually, when working within a metric space,  the diameters of $E_n$ tend to zero as $n$ increases. It is thus natural to refer to $ W\big(T,\{E_n\}\big)$ as the corresponding shrinking target set associated with the given dynamical system and target sets. Since $T$ is measure-preserving $ \mu (T^{-n}(E_n))  =  \mu (E_n)$, and  a straightforward consequence of the (convergent) Borel-Cantelli Lemma is that
\begin{equation} \label{appconv}
\mu\big(W(T,\{E_n\})\big)=0   \qquad {\rm if \ } \qquad \sum_{n=1}^\infty \mu(E_n)  \, < \, \infty  \, .
\end{equation}

\noindent  Now two natural questions arise.   Both fall under the umbrella of the ``shrinking target problem''  formulated in \cite{hv1}.
\begin{itemize}
  \item[] \vspace*{-1ex}
  \begin{itemize} \item[~ \textbf{(P1)}] What is the $\mu$-measure of $W(T,\{E_n\})$ if the measure sum in \eqref{appconv} diverges? \\[-2ex]
  \item[~ \textbf{(P2)}] What is the Hausdorff dimension of $W(T,\{E_n\})$  if the measure sum converges and so $\mu\big(W(T,\{E_n\})\big)=0$?
      \end{itemize}
      \vspace*{-1ex}
\end{itemize}

\noindent  To be precise, the target sets $E_n$ in the original formulation in  \cite{hv1} are restricted to balls $B_n$.  The more general setup naturally incorporates a larger class of problems. For example, within the context of simultaneous  Diophantine approximation,  it enables us to address problems associated with the weighted (the target sets are rectangular) and multiplicative (the target sets are hyperbola) theories -- see Remark~\ref{rk-2} in \S\ref{secMTT} below.

In this paper  we revisit the  shrinking target problem investigated in \cite{hvMat} in which $T$ is a matrix transformation of the $d$-dimensional torus $X=\T^d:=\R^d/\Z^d$.  There are several reasons for doing this.  Firstly, for integer  matrix transformations a solution to (P1)  was announced in \cite{hvMat}; namely,  under some regularity condition on the rate at which the diameters of the balls $B_n$ tend to zero, we have that
 \begin{equation} \label{pastappdiv}
m_d(W(T,\{B_n\}))=1   \qquad {\rm if \ } \qquad \sum_{n=1}^\infty m_d(B_n)  \, = \, \infty  \,
\end{equation}
where $m_d$ is $d$-dimensional Lebesgue measure.  However, the intended paper establishing this divergent analogue  of  \eqref{appconv}  was never completed\footnote{The author SV would like to take this opportunity to apologise for making an announcement and then not delivering the goods!} and to the best of our knowledge such a result has not to date appeared in print elsewhere.    In this paper not only do we rectify the situation but  we consider the set up in which $T$ is a  real (rather than just integer) matrix transformation and the `target' sets are general sets rather than just balls.  Furthermore, our results are significantly stronger than statements such as \eqref{pastappdiv}.  In a nutshell, our solution to (P1) consists of full measure statements that are quantitative in nature. Next,   turning our attention to (P2), Theorem~2 in \cite{hvMat} provides a precise formula for the Hausdorff dimension of $W(T,\{B_n\})$  when  $T$ is an integer matrix transformation diagonalizable over the rationals. In this paper we investigate the more general situation in which $T$ is real and by making use of technology that was not available at the time of \cite{hvMat},  we show (for instance) that the aforementioned formula for the Hausdorff dimension holds for a large class of real diagonal matrix transformations.

At the heart of our solution to (P1) for matrix transformations of tori,  is a result that holds for a large class of  measure-preserving dynamical systems.   We start with describing this broader result and then move onto formally stating our theorems for matrix transformations.

\subsection{A quantitative full measure result for \s-mixing dynamical systems  \label{genmix}}

Given a measure-preserving dynamical system  $(X,\mathcal{B},\mu,T)$ and a sequence $\{E_n \}_{n \in \mathbb{N}} $ of subsets  in $\mathcal{B}$,  we will show that if $\mu$ is exponentially mixing and the measure sum  in \eqref{appconv} diverges then the associated $\limsup $ set $W\big(T,\{E_n\}\big)$  is of full measure.   However, it turns out that a lot more is true.  We can establish a quantitative full measure statement and at the same time work with the  potentially  weaker notion of \s-mixing.

 \begin{definition} \label{defmixing}
 Let $(X,\mathcal{B},\mu,T)$ be a measure-preserving dynamical system and  $\mathcal{C}$ be a collection of measurable subsets of $X$.  For $ n \in \N$, let
\begin{equation}\label{phimixing}
\phi(n):=\sup\left\{ \; \left|\frac{\mu(E\cap T^{-n}F)}{\mu(F)}-\mu(E)\right| \, : \,  E\in\mathcal{C}, F\in\mathcal{C}\right\}.
\end{equation}
We say that $\mu$ is  \s-mixing  (short for summable-mixing) with respect to $(T,\mathcal{C})$ if the series $\sum_{n=1}^\infty\phi(n)$ converges.
\end{definition}

Recall, that the above of mixing is stronger than that of {\em $\phi$-mixing} which simply requires that $ \phi(n) \to 0$ as $n \to \infty$.   Also recall,  that $\mu$ is \emph{exponentially mixing} with respect to $(T,\mathcal{C})$ if there exists a constant $0<\gamma<1$ such that
\begin{equation}\label{mixing}
\mu(E\cap T^{-n}(F))=\mu(E)\mu(F)+O(\gamma^n)\mu(F) ,
\end{equation}
for any $n \ge 1$ and $E, F\in\mathcal{C}$  -- the  implied  constant in the big O does not depend on the sets $E$ and $F$.  In other words, and not surprisingly,   exponentially mixing  and  \s-mixing coincide whenever $\phi(n)$ converges to zero exponentially fast.   It is worth mentioning that in the standard definition,  condition (\ref{mixing}) is required to hold for any  $F \in \mathcal{B}$ rather than  just in $\mathcal{C}$.   We refer the reader to the survey paper \cite{B2005} for further details including ``other'' variants of the notion of exponentially mixing.  Also, see \S\ref{missed} below.

	As above,  let $\{E_n \}_{n \in \mathbb{N}} $ be a sequence of measurable subsets  of $X$.  Then, given $ N \in \N$ and $x \in X$,   consider the counting function
\begin{equation} \label{countdef}
R(x,N)=R(x,N; T, \{E_n \}) := \# \big\{ 1\le n \le   N :    T^{n}(x) \in E_n \}.
\end{equation}
As alluded to in the definition, we will often simply write $R(x,N)$ for $R(x,N; T, \{E_n \})$ since the other dependencies will be clear from the context and are usually fixed.  It is easily seen that the convergent  statement \eqref{appconv} is equivalent to saying that if the  measure sum converges, then  $\lim_{N \to \infty} R(x,N)  $ is finite  for  $\mu$--almost all $x \in X$.  The following  result implies that for a large class of dynamical systems,  if  the measure sum  diverges  then  $\mu$--almost all $x \in X$ `hit' the target sets $E_n$  the `expected' number of times.

\begin{theorem}\label{gencountthm}
Let $(X,\mathcal{B},\mu,T)$ be a measure-preserving dynamical system  and $\mathcal{C}$ be a collection of subsets of $X$. Suppose that $\mu$ is \s-mixing with respect to $(T,\mathcal{C})$ and let  $\{E_n \}_{n \in \mathbb{N}} $ be a sequence of subsets  in $\mathcal{C}$.   Then, for any given $\varepsilon>0$, we have that
\begin{equation} \label{def_Psi}R(x,N)=\Phi(N)+O\left(\Phi^{1/2}(N) \ (\log\Phi(N))^{3/2+\varepsilon}\right)
\end{equation}
\noindent for $\mu$-almost all $x\in X$, where
$$\Phi(N):=\sum_{n=1}^N \mu(E_n) \,  .
$$
\end{theorem}

A simple consequence of Theorem~\ref{gencountthm} is that  $\lim_{N \to \infty} R(x,N) = \infty  $ for  $\mu$--almost all $x \in X$ if  the measure sum diverges and so together with \eqref{appconv} we obtain the following zero-full measure criterion.

\begin{corollary} \label{gencountcor}
	Let $(X,\mathcal{B},\mu,T)$ be a measure-preserving dynamical system  and $\mathcal{C}$ be a collection of subsets of $X$. Suppose that $\mu$ is \s-mixing with respect to $(T,\mathcal{C})$ and let  $\{E_n \}_{n \in \mathbb{N}} $ be a sequence of subsets  in $\mathcal{C}$.    Then
\begin{eqnarray}\label{appdiv}
		\mu\big(W(T,\{E_n\})\big)=
		\begin{cases}
			0 &\text{if}\ \  \sum_{n=1}^\infty \mu\big(E_n\big)<\infty\\[2ex]
			1 &\text{if}\ \ \sum_{n=1}^\infty \mu\big(E_n\big)=\infty.
		\end{cases}
	\end{eqnarray}
\end{corollary}

\bigskip

Before moving onto considering  the specific situation in which $T$ is a matrix transformation of the torus we discuss previous related works.

\subsubsection{Connection to other works \label{missed}}

We will focus on two previous works that are related to the framework presented above; i.e. the notion of  \s-mixing and its consequences.
In an interesting paper \cite{FMP},  Fern\'{a}ndez, Meli\'{a}n $\&$ Pestana introduced the notion of a transformation $T$ being \emph{uniformly mixing at a point $x_0 \in X $}. Their notion coincides with our  Definition~\ref{defmixing} if we restrict the collection $\mathcal{C}$ to balls $B$ centered at $x_0$.  The upshot  \cite[Theorem~1]{FMP} is that given a decreasing sequence of balls $B_n := B(x_0, r_n)$, if $T$ is  uniformly mixing at $x_0 $ and $ \sum_{n=1}^{\infty} \mu(B_n) = \infty $ then
\begin{equation} \label{jkl}
\lim_{N \to \infty} \frac{R(x,N)}{\Phi(N)}   = \lim_{N \to \infty} \frac{  \# \big\{ 1\le n \le   N :    T^{n}(x) \in B_n \big\} }{  \sum_{n=1}^{N} \mu(B_n)  }   = 1 \, .
\end{equation}
Clearly our Theorem~\ref{gencountthm} not only implies this asymptotic statement but it also provides a reasonably sharp estimate for the error term.  As a consequence, the various applications of \eqref{jkl} considered in  \cite{FMP} can be strengthened accordingly.   Indeed,  their main motivating application to inner functions  \cite[Theorem~2]{FMP} can be improved to the following statement.

\begin{thmF} Let $f: \D \to \D$ be an inner function with $f(0)=0$, but not a rotation. Let $\xi_0$ be a point in $\partial\D$ and let $\{r_n\}$ be a decreasing sequence of positive numbers. If $\sum_{n =1}^{\infty} r_n = \infty $, then for any given $\varepsilon>0$, we have that
$$ \# \big\{ 1\le n \le   N :  d\big((f^*)^n(\xi), \xi_0    \big)  < r_n   \big\}  =\Phi(N)+O\left(\Phi^{1/2}(N) \ (\log\Phi(N))^{3/2+\varepsilon}\right)
$$
\noindent for $\mu$-almost all $x\in X$, where
$\Phi(N):=\sum_{n=1}^N r_n
$,  $ f^*(\xi) = \lim_{r \to 1_-} f(r\xi) $ and $d$ is the angular distance in $\partial\D$.
\end{thmF}

  In the later stages of preparing  this manuscript, we discovered that our Theorem~\ref{gencountthm}  overlaps  with a result of  Philipp \cite[Theorem~3]{P1967} dating back to 1967. Indeed, in his theorem the condition imposed on the  sequence $\{E_n \}_{n \in \mathbb{N}} $ of measurable  sets is in effect equivalent to our notion of  \s-mixing with $\mathcal{C}= \mathcal{B} $.   It appears  that \cite[Theorem~3]{P1967} has been either entirely overlooked, or at least not fully exploited in previous works.  For the sake of completeness we have decided to include the proof of Theorem~\ref{gencountthm} in \S\ref{bcsv}.  Moreover, our proof is pretty short and unlike Philipp's approach it exploits a rather general tool (Lemma~\ref{ebc} in \S\ref{bcsv}) for establishing  sharp counting statements.  To the best of our knowledge, a slightly weaker version of the tool, which suffices to establish Theorem~\ref{gencountthm}, first appears in Sprind{\v z}uk's book \cite[Lemma~10]{Sprindzuk}  which was some ten years after Philipp's paper. Furthermore, we  have decided to include a self contained proof of the Corollary~\ref{gencountcor} since it is rather nifty and some readers may only be interested in the zero-full measure criterion rather than its stronger quantitative form.

%We now consider the specific situation in which $T$ is a matrix transformation of the torus.

\subsection{Quantitative full measure results for matrix transformations \label{secMTT}}

Let $T$ be a $d\times d$ non-singular matrix with real coefficients.
Then, $T$ determines a self-map of  the $d$-dimensional torus
 $X=\T^d:=\R^d/\Z^d$; namely, it sends $\vx\in\T^d$ to $T\vx$ modulo one. In what follows, $T$ will denote both the matrix and the transformation.  It should be obvious from the context what is meant.  Furthermore,  for $n \in \N$,  by  $T^n$  we will always mean the $n$-th iteration of the transformation $T$ rather than the matrix multiplied $n$ times.   With reference to the general setup of \S\ref{genmix},  we now describe a broad collection $\cC$  of  `target' sets contained in $\T^d $ so that for any sequence $\{E_n \}_{n \in \mathbb{N}} $ of subsets in $\mathcal{C}$  we are able to address the  shrinking target `measure'  problem (P1) for the associated $\limsup$ set $ W\big(T,\{E_n\}\big) $. In order to do this, we require  the notion of the  Minkowski content of a set in $\R^d$.  We  start by recalling this basic  notion  from geometric measure theory.

 \medskip

Let $0\le s \le d$ be two positive integers and let  $A$ be a subset of $\R^d$.  Let $m_d$ denote the $d$-dimensional Lebesgue measure and $\alpha(d)$  denote the volume of the $d$-dimensional unit Euclidean ball $\{ \vx \in \R^d : |\vx| < 1 \}$.   By convention, we define  $\alpha(0) :=1$.  For $0< \epsilon < \infty$,  we  let $A(\ep)$ denote the $\ep$-neighbourhood of $A$; that is
$$
A(\ep) := \{\vx  \in \R^d: \dist(\vx, A)<\epsilon \}   \, .
$$
Then,  following the classical text of Federer \cite[Section 3.2.37]{Fed},   the $s$-dimensional upper and lower  Minkowski content of $A$ are defined, respectively as
\[
M^{*s}(A):=\limsup_{\epsilon\to 0^+} \frac{ m_d (A(\ep)) }{\alpha(d-s)  \, \epsilon^{d-s}}   \qquad
{\rm and  }  \qquad
M^{s}_{*}(A):=\liminf_{\epsilon\to 0^+} \frac{ m_d (A(\ep)) }{\alpha(d-s)\epsilon^{d-s}}  \, .
\]
If these upper and lower  Minkowski contents are equal, then their common value is called the $s$-dimensional  Minkowski content of $A$ and is denoted by $M^{s}(A)$.  In general, the set functions $M^{*s} $  and $M_{*}^{s}  $ are not measures.  However, for nice sets it turns out that both equal a constant multiple of  the Lebesgue measure $m_s$.   In particular,    a result of Federer \cite[Theorem 3.2.39]{Fed} states that
if $A$ is a closed $s$-rectifiable subset of $\mathbb{R}^d$ (i.e. the image of a bounded set from $\mathbb{R}^s$ under a Lipschitz function), then the $s$-dimensional Minkowski content of $A$ exists, and is equal to the $s$-dimensional Hausdorff measure of $A$.  Recall that for integer $s$ the latter is a constant multiple of  $s$-dimensional Lebesgue measure.  Also, for the sake of completeness it is worth mentioning that the  Minkowski content is intimately related to the Minkowski dimension which, nowadays is more commonly referred to as the box dimension. When considering this fractal dimension, $s$ need not be an integer and we put $\alpha(d-s)=1 $ in the above definitions of  upper and lower  Minkowski contents.    For further details see \cite[Section 3.1]{F}, \cite[Sections 3.2.37-44]{Fed}, \cite[Chapter 5]{MAT} and references within.

\medskip

The following proposition  identifies  the  collection $\cC$  of  `target' sets alluded to above as subsets $E$ of $\T^d$ for which the boundary $\partial E$ has  bounded $(d-1)$-dimensional upper Minkowski content.  It makes use of the work initiated by Keller \cite{K1979, KPhD}  on the existence and properties of absolutely continuous invariant measures for piecewise expanding maps, and subsequently developed by  the likes of G\'ora $\&$ Boyarsky \cite{GB1989}, Buzzi \cite{B1997, B2000, B2001},   Buzzi $\&$ Maume-Deschamps \cite{BM2002}, Saussol \cite{S2000} and  Tsujii \cite{T2000, T2001}.

\begin{proposition} \label{mainohyes}
Let $T$ be a real,  non-singular  matrix transformation of the torus
 $\T^d$.
Suppose that all eigenvalues of $T$ are of modulus  strictly larger than $1$. Then
\begin{itemize}
  \item[(i)]  there exists an absolutely continuous (with respect to Lebesgue measure $m_d$) invariant probability measure (acim)  $\mu$,
  \item[(ii)] {the support $A \subseteq \T^d$ of any acim $\mu$ can be decomposed into finitely many disjoint measurable sets $A_1, \dots, A_s$} such that for each $1 \le i \le s$ the restriction $\mu|_{A_i}$ of $\mu$ to $A_i$ is ergodic and is equivalent to the restriction $m_d|_{A_i}$ of  Lebesgue measure $m_d$ to $A_i$,
  \item[(iii)] each ergodic component $A_i$ in (ii) can in turn be decomposed into
  finitely many disjoint measurable sets $A_{i1}, \dots, A_{ip_i}$ such that  for each $1 \le j \le p_i$ the restriction   $\mu|_{A_{ij}}$ is mixing with respect to $T^{p_i}$,
    \item[(iv)] on each mixing component $A_{ij}$ in (iii), the restriction $\mu|_{A_{ij}}$ is exponentially mixing with respect to $(T^{p_i}, \mathcal{C})  $ for any collection  $\mathcal{C}$  of subsets $E$ of $A_{ij}$ satisfying the bounded property
$$ ({\boldsymbol{\rm B}})   : \quad \sup_{E\in \mathcal{C}}   \; M^{*(d-1)}(\partial E)  \, <  \, \infty  \, .$$
\end{itemize}
\end{proposition}

\bigskip

\begin{remark}  By definition, the restriction $\mu|_{A}$ of  a probability measure $\mu$ to  a set $A$ with $\mu(A) > 0$  is  normalized so that it too is a probability measure. In other words, for an arbitrary measurable set $E$
$$\mu|_{A} (E) :=  {\tiny \frac{1}{\mu(A)}} \, \mu(E) \, . $$
 For each $1 \le i \le s$,  the  sets  $A_{ij}$  ($1 \le j \le p_i$) appearing in part (iii) are referred to as the mixing components of $A_i$ $(=\bigcup_{j=1}^{p_i}A_{ij})$  and the positive integers $p_i$  are the period of the mixing components.  These mixing components satisfy the property that
$$
T(A_{ij}) = A_{ij+1} \quad (1 \le j \le p_i-1) \qquad {\rm and } \qquad    T(A_{ip_i}) = A_{i1}  \, .
$$
Also, for the sake of clarity, completeness and convenience,  recall that if $\mu$ and $\nu$ are two measures on the same measurable space,  then $\mu$ is \emph{absolutely continuous} with respect to $\nu$  (written $\mu \ll \nu$)  if $\mu(E) = 0 $ for every measurable set $ E $ for which $ \nu(E) = 0 $.  Moreover, the  measures $\mu$ and $\nu$ are  \emph{equivalent} if $\mu \ll \nu$ and $\nu \ll \mu$ and  are said to be \emph{strongly equivalent} { or \emph{comparable}}  if there exists a constant $C\ge 1$ such that
for every measurable set $ E $
\[
{ C^{-1}} \nu(E) \le   \mu(E) \le   C \nu(E).
\]
Given a  measure-preserving dynamical system $(X, \mathcal{B}, \mu, T)$, the  invariant measure $\mu$ is  ergodic if for every set $E\in\mathcal{B}$  with $T^{-1}E=E$  we have either $\mu(E)=0$ or $\mu(E)=1$.  Moreover, $\mu$ is said to be  mixing with respect to $T$  (often referred to strong-mixing) if for every  $E, F\in\mathcal{B}$
\[
\lim_{n\to\infty} \mu(E\cap T^{-n}F)=\mu(E)\mu(F).
\]
Clearly, exponentially mixing tells us that the implied error term in the above limit  decays exponentially.  Also, if we put $F=E$ we immediately see that mixing implies ergodic.
\end{remark}

\medskip

%We will postpone the proof of proposition  to the next  section.

The following constitutes our most general measure theoretic result for the shrinking target problem for matrix transformations of tori. As we shall see the ``divergent'' part, which is the hard part,  is essentially an immediate consequence of combining Theorem~\ref{gencountthm} and Proposition~\ref{mainohyes}.

\begin{theorem}\label{measure-matrix}
Let $T$ be a real, non-singular matrix transformation of the torus $\mathbb{T}^d$. Suppose that all eigenvalues of $T$ are of modulus  strictly larger than $1$ and let $\mathcal{C}$ be any  collection of subsets $E$ of $\mathbb{T}^d$ satisfying the bounded property
$ ({\boldsymbol{\rm B}}) $.  Furthermore,  let $\mu$ be an acim and suppose it  has  support $\T^d$ and is mixing with respect to $T$.  Then
for  any sequence $ \{E_n \}_{n \in \mathbb{N}} $ of subsets  in $\mathcal{C}$ and $\varepsilon>0$, we have that
\begin{equation}  \label{formula-MT-count}R(\vx,N)=\Phi(N)+O\left(\Phi^{1/2}(N) \ (\log\Phi(N))^{3/2+\varepsilon}\right) %\, ,   \qquad  where  \quad
%%\Psi(N):=\textstyle{\sum_{n=1}^N\psi(q_n)}\varsigma
%\Phi(N):=\sum_{n=1}^N \mu(E_n) \, .
 \end{equation}
  for $\mu$-almost all (equivalently  $m_d$-almost all) $\vx\in \mathbb{T}^d$, where $\Phi(N):=\sum_{n=1}^N \mu(E_n)$.
In particular,
	\begin{eqnarray}\label{formula-measure-matrix}
		m_d\big(W(T,\{E_n\})\big)=\mu\big(W(T,\{E_n\})\big)=
		\begin{cases}
			0 &\text{if}\ \  \sum_{n=1}^\infty \mu\big(E_n\big)<\infty\\[2ex]
			1 &\text{if}\ \ \sum_{n=1}^\infty \mu\big(E_n\big)=\infty.
		\end{cases}
	\end{eqnarray}
%Furthermore,
%$$
%m_d\big(W(T,\{E_n\})\big)=\mu\big(W(T,\{E_n\})\big)
%$$
\end{theorem}

\medskip

We note that the existence of the acim measure $\mu$ in Theorem~\ref{measure-matrix}   is guaranteed by part (i) of Proposition~\ref{mainohyes} and that the assumptions imposed on it, namely that the support of $\mu$ is the whole space $\T^d$ and that $\mu$ is mixing with respect to $T$, are often satisfied.  Indeed, this is the situation when the eigenvalues of $T$ are large in modulus or the coefficients of $T$ are integers.   Regarding the former we have the following precise statement. We will come to the integer situation shortly (see Theorem~\ref{corInteger} below).

%
% The assumptions imposed on the measure $\mu$ in Theorem~\ref{measure-matrix}, namely that the support of $\mu$ is the whole space $\T^d$ and that $\mu$ is mixing with respect to $T$, are  often satisfied.  Indeed, this is the situation when the eigenvalues of $T$ are large in modulus or the coefficients of $T$ are integers.   Regarding the former we have the following precise statement. We will return to the integer situation shortly.

\begin{theorem}   \label{cor-bigger-eigen}
	Let $T$ be a real, non-singular matrix transformation of the torus $\mathbb{T}^d$. Suppose that all eigenvalues of {$T$} are of modulus strictly larger than $1+\sqrt{d}$. Let $\mathcal{C}$ be any  collection of subsets of $\mathbb{T}^d$ satisfying the bounded property
$ ({\boldsymbol{\rm B}}) $.  Then there is a unique acim $\mu$, such that for  any sequence $ \{E_n \}_{n \in \mathbb{N}} $ in $\mathcal{C}$ and $\varepsilon>0$, the counting formula \eqref{formula-MT-count} holds for $\mu$-almost all (equivalently  $m_d$-almost all) $\vx\in \mathbb{T}^d$, where $\Phi(N):=\sum_{n=1}^N \mu(E_n)$.
In particular, the zero-full measure  criteria  \eqref{formula-measure-matrix}  holds.
\end{theorem}

\medskip

\begin{remark}  \label{unique}
 The fact that the acim $\mu$  appearing in Theorem~\ref{cor-bigger-eigen} is unique is a trivial consequence of the fact that any acim satisfying the hypotheses of Theorem~\ref{measure-matrix} has to be unique. Indeed, to see that this is the case, suppose there exist two such measures.  Then by part (ii) of  Proposition~\ref{mainohyes},  both are equivalent to $m_d$.
By assumption, both are mixing with respect to $T$ and hence ergodic. It thus follows (see \cite[Theorem 6.10]{W-GTM79}) that the two measures are equal.
\end{remark}

\begin{remark}  \label{force}
By using the full force of Proposition~\ref{mainohyes}, the assumptions on  $\mu$ in Theorem~\ref{measure-matrix} can be completely dropped if we restrict our attention to the shrinking target set $W(T^p,\{E_n\}) \cap A $.  Here $A \subseteq \T^d$ is the support of the acim $\mu$ (guaranteed by part (i) of  Proposition~\ref{mainohyes}) and $p:=   p_1 p_2  \ldots p_s$ where the integers $p_i$ are the periods of the mixing components associated with  part (iii) of Proposition~\ref{mainohyes}.  Establishing   Theorem~\ref{metricresult} below  is an illustration of precisely this remark in action.    In short, the point of making the assumptions on $\mu$ in Theorem~\ref{measure-matrix} is to obtain a simple statement for the size of $W(T,\{E_n\})$  in terms of the probability measure $m_d$ supported on $\T^d$.
%The assumption is to concentrate ourselves to a simple case of Proposition \ref{mainohyes}, i.e., the case with only one mixing component of period $1$. In general, by Proposition \ref{mainohyes}, we can obtain the same result if we restrict ourselves to the study of the shrinking target sets with respect to an iteration of $T$ on the support of $\mu$.
\end{remark}

\medskip

\begin{remark}\label{rk-2} We consider two special families  of target sets  that correspond to  ``natural'' setups within the classical theory of Diophantine approximation.
Let $\psi: \R^+ \to \R^+$ be a real positive function and fix some point $\va:= (a_1, \ldots,  a_d)  \in \T^d$. For $n \in \N$, let
$$
B_n = B(\va, \psi(n)):= \Big\{ \vx  \in \mathbb{T}^d :   \max_{1\le i \le d} \|x_i-a_i\| \le \psi(n)   \Big\}
$$
and
\begin{equation} \label{iona}
H_n = H(\va, \psi(n))  :=   \Big\{ \vx \in \mathbb{T}^d :  \prod_{1\le i\le d} \|x_i-a_i\|  \le \psi(n)   \Big\}  \, ,
\end{equation}
where $\| \, . \, \|$ denotes the distance to the nearest integer.
Clearly, $B_n$ is  a  ball with respect to the maximum norm and $H_n$ is a hyperboloid  --   both are centred at the fixed point $\va$.
In turn, let
$$W(T,\psi,\va):=\{\vx\in\mathbb{T}^d: T^n(\vx)\in B(\va,\psi(n))\ \ \text{for infinitely many}\ n\in\mathbb{N}\},$$
and
$$W^{\times}(T,\psi,\va):=\{\vx\in\mathbb{T}^d: T^n(\vx)\in H(\va,\psi(n))  \ \ \text{for infinitely many}\ n\in\mathbb{N}\},$$
denote the corresponding shrinking target sets.  The former is intimately related to sets studied within the classical simultaneous theory of Diophantine approximation and the latter to the multiplicative theory. To see this explicitly, suppose that $T$ is an integer, diagonal  matrix. In fact, suppose that
$$
T = {\rm diag} \, (t_1, \ldots, t_d  )    \quad  {\rm with} \quad  t_i   \ge 2 \,
$$
and for convenience suppose $\va$ is the origin $\textbf{0}$. Then, on using the fact that $T$ is integer,  it is easily seen that for any given $\vx:= (x_1, \ldots,  x_d)  \in [0,1)^d$ we have
$$
T^n(\vx)\in B(\textbf{0},\psi(n)) \quad \Longleftrightarrow \quad \max_{1\le i \le d} \|t_i^nx_i\| \le \psi(n)
$$
and
$$
T^n(\vx)\in H(\textbf{0},\psi(n)) \quad \Longleftrightarrow \quad \prod_{1\le i \le d} \|t_i^nx_i\| \le \psi(n) \, .
$$
  It is evident that  both  the families of target sets $\{B_n\}_{n\ge   1}$ and $\{H_n\}_{n\ge   1}$ satisfy the bounded property
$ ({\boldsymbol{\rm B}}) $.   Thus, at the very least, our theorems incorporate both the simultaneous and multiplicative aspects of the classical theory of Diophantine approximation in which the denominators of the rational approximates are restricted to lacunary sequences. For the explicit statements see Corollaries~\ref{Cor1}~$\&$~\ref{Cor2} below.   In fact, our bounded property
$ ({\boldsymbol{\rm B}}) $ condition is far more general than the so called property $ ({\boldsymbol{\rm P}}) $ condition (see \S\ref{galSV}) imposed  by Gallagher in his elegant and influential paper \cite{GallP}.   We reiterate that our results hold for any family of target sets $\{E_n\}_{n\ge   1}$ whose  boundaries are  rectifiable and are of uniformly bounded $(d-1)$-dimensional Lebesgue measure.
\end{remark}

\medskip
%
%
%We now investigate natural  situations in which the measure $\mu$ associated with Theorem~\ref{measure-matrix} is strongly equivalent to Lebesgue measure $m_d$ on $\T^d$ so that we can replace $\sum_{n=1}^\infty\mu(E_n)$ by $\sum_{n=1}^\infty m_d(E_n)$ in the righthand side of (\ref{formula-measure-matrix}).

We now investigate natural  situations in which the  measure $\mu$ associated with Theorem~\ref{measure-matrix} is strongly equivalent to Lebesgue measure $m_d$ on $\T^d$.  For such situations  we can replace $\mu$ by $ m_d$ in the finite sum $\Phi(N)$ and the righthand side of (\ref{formula-measure-matrix}) and thus obtain  statements entirely in terms of Lebesgue measure.
%Although the statements can  be viewed as corollaries to Theorem~\ref{measure-matrix}, they are really direct consequences of Propositions~\ref{gencountthm}~$\&$~\ref{mainohyes}.   The work lies in showing  that the measure under consideration satisfies the hypotheses of the propositions. Indeed, this is precisely in line with establishing the general theorem above.
To start with, let us stick with  real, non-singular matrices and suppose that $T$ is diagonal with all eigenvalues (or equivalently diagonal entries)  $\beta_1,\beta_2,\dots,\beta_d$ of modulus strictly larger than $1$.   Now with this in mind, let $\beta\in \mathbb{R}$ such that $|\beta|>1$ and  let $\mu_\beta$ be corresponding Parry measure for positive $\beta$ or the Yrrap measure for negative $\beta$ -- see $\S\ref{secmetricresult}$ for background and further details.   Also, let  $K(\beta)$ denote  the support of $\mu_\beta$.  Then (see Proposition~\ref{Prop:support} below),
\begin{equation}\label{supp}
K(\beta)= [0,1]  \qquad {\rm if  }  \qquad  \beta \in (-\infty, -g] \, \cup  \,  (1, +\infty) \, , \end{equation}
and  $K(\beta)$ is a finite union of closed intervals contained in  $[0,1]$ if $\beta \in (-g, -1)$.   Here and throughout, $$g:=(\sqrt5 +1)/2$$ is the golden ratio.    Now returning to the transformation $T$ of the torus $\T^d$,  we consider the product measure $ \nu$ of the corresponding  one-dimensional Parry-Yrrap measures $ \mu_{\beta_i}$; that is
\begin{equation}\label{prodmeas}
\nu:=\mu_{\beta_1}\times\mu_{\beta_2}\times\cdots\times\mu_{\beta_d}   \, .
\end{equation}
Then by definition,  the support of $\nu$ is
\[
K:=\prod_{i=1}^d K(\beta_i),
\]
and in view of \eqref{supp} we have that $K=\T^d$ if all $\beta_i$ are in $(-\infty, -g] \, \cup  \,  (1, +\infty)$.   On exploiting the properties of the Parry-Yrrap  measures $\mu_{\beta_i}$  and using the full force of Proposition~\ref{mainohyes} (see Remark~\ref{force}) we are able to show  that  $\nu$ is  exponentially mixing with respect to $(T, \mathcal{C})  $ for any collection  $\mathcal{C}$  of subsets $E$ of $K$ satisfying the bounded property $ ({\boldsymbol{\rm B}}) $.  The details of this are given  in  \S\ref{secmetricresult}   and is at the heart of establishing the following statement for real, diagonal matrix transformations.

%
%\begin{theorem} \label{metricresult}
%	Let $T$ be a real, non-singular matrix transformation of the torus $\mathbb{T}^d$. Suppose that $T$ is diagonal with  all eigenvalues in $(-\infty, -g]\cup (1, +\infty)$.  Let $\mathcal{C}$ be any  collection of subsets $E$ of $\mathbb{T}^d$ satisfying the bounded property
%$ ({\boldsymbol{\rm B}}) $.  Then for  any sequence $ \{E_n \}_{n \in \mathbb{N}} $ in $\mathcal{C}$ we have that
% 	\begin{eqnarray*}
%		m_d\big(W(T,\{E_n\})\big)=
%		\begin{cases}
%			0 &\text{if}\ \  \sum_{n=1}^\infty m_d\big(E_n\big)<\infty\\[2ex]
%			1 &\text{if}\ \ \sum_{n=1}^\infty m_d\big(E_n\big)=\infty.
%		\end{cases}
%	\end{eqnarray*}
%\end{theorem}

%
%\begin{theorem} \label{metricresult}
%	Let $T$ be a real, non-singular matrix transformation of the torus $\mathbb{T}^d$. Suppose that $T$ is diagonal and all  eigenvalues are of modulus strictly larger than $1$, and let $\mathcal{C}$ be any  collection of subsets $E$ of $K$ satisfying the bounded property
%$ ({\boldsymbol{\rm B}}) $.  Then for  any sequence $ \{E_n \}_{n \in \mathbb{N}} $ of subsets  in $\mathcal{C}$,  we have that
%	\begin{eqnarray*}
%		m_d|_K\big(W(T,\{E_n\})\big)=\mu\big(W(T,\{E_n\})\big)=
%		\begin{cases}
%			0 &\text{if}\ \  \sum_{n=1}^\infty \mu\big(E_n\big)<\infty\\[2ex]
%			1 &\text{if}\ \ \sum_{n=1}^\infty \mu\big(E_n\big)=\infty.
%		\end{cases}
%	\end{eqnarray*}
%	\end{theorem}

\begin{theorem} \label{metricresult}
	Let $T$ be a real, non-singular matrix transformation of the torus $\mathbb{T}^d$. Suppose that $T$ is diagonal and all  eigenvalues $\beta_1,\beta_2,\dots,\beta_d$ are of modulus strictly larger than $1$.  Let $\nu$ be the product measure given by \eqref{prodmeas}  with support $K \subseteq \mathbb{T}^d $.  Let $\mathcal{C}$ be any  collection of subsets $E$ of $K$ satisfying the bounded property
$ ({\boldsymbol{\rm B}}) $.  Then for  any sequence $ \{E_n \}_{n \in \mathbb{N}} $ in $\mathcal{C}$ and $\varepsilon>0$,  the counting formula \eqref{formula-MT-count} holds for $\nu$-almost all (equivalently  $m_d|_K$-almost all) $\vx\in \mathbb{T}^d$, where $\Phi(N):=\sum_{n=1}^N \nu(E_n)$.
In particular,
	\begin{eqnarray*}		m_d|_K\big(W(T,\{E_n\})\big)=\nu\big(W(T,\{E_n\})\big)=
		\begin{cases}
			0 &\text{if}\ \  \sum_{n=1}^\infty \nu\big(E_n\big)<\infty\\[2ex]
			1 &\text{if}\ \ \sum_{n=1}^\infty \nu\big(E_n\big)=\infty.
		\end{cases}
	\end{eqnarray*}

\noindent Furthermore,  if all the eigenvalues of $T$ are in $(-\infty, -g]\cup (1, +\infty)$ then $K= \mathbb{T}^d$ and we can replace $\nu $ by $m_d$ in the above; i.e. the counting formula \eqref{formula-MT-count} holds for $m_d$-almost all $\vx\in \mathbb{T}^d$, where $\Phi(N):=\sum_{n=1}^N m_d(E_n)$ and in particular
\begin{eqnarray*}
		m_d\big(W(T,\{E_n\})\big)=
		\begin{cases}
			0 &\text{if}\ \  \sum_{n=1}^\infty m_d\big(E_n\big)<\infty\\[2ex]
			1 &\text{if}\ \ \sum_{n=1}^\infty m_d\big(E_n\big)=\infty.
		\end{cases}
	\end{eqnarray*}

\end{theorem}

In the case the collection  $\mathcal{C}$  of subsets of $K$ is restricted to rectangles   with sides parallel to the axes (they clearly satisfy the bounded property
$ ({\boldsymbol{\rm B}}) $) we can avoid using Proposition~\ref{mainohyes} and give a self-contained and reasonably elementary proof of the above theorem (see \S\ref{specialthm4}). In particular, it is more than enough to establish the following corollary for balls (cubes); i.e., when we take $E_n=B_n$ (see Remark \ref{rk-2}) in the above theorem.

%\hi{REMOVE
%It can then be shown (see Lemma~\ref{productmeasure} in \S\ref{secmetricresult} below) that the product measure $ \mu$
%%is an acim which is equivalent to Lebesgue measure $m_d$ and
% is exponentially mixing with respect to
%any  collection $\mathcal{C}$ of rectangles contained in $K$ \hi{(RECTANGLES DON'T SHOW UP IN STATEMENTS BELOW)}. The following statement is a consequence of Theorem \ref{measure-matrix}.  A direct proof for the case that $\mathcal{C}$ is a collection of rectangles contained in $K$ can also be provided by using only Theorem~\ref{gencountthm}.}

\begin{corollary}\label{Cor1}
Let $T$ be a real, non-singular matrix transformation of the torus $\mathbb{T}^d$. Suppose that $T$ is diagonal with  all eigenvalues in $(-\infty, -g]\cup (1, +\infty)$. Let $\psi: \mathbb{R}^+\to \mathbb{R}^+$ be real positive function and $\va\in \T^d$. Then
for  any  $\varepsilon>0$, we have that
$$ \# \big\{ 1\le n \le   N :    T^n(\vx)\in B(\va,\psi(n)) \big\}
=
\Phi(N)+O\left(\Phi^{1/2}(N) \ (\log\Phi(N))^{3/2+\varepsilon}\right) %\, ,   \qquad  where  \quad
%%\Psi(N):=\textstyle{\sum_{n=1}^N\psi(q_n)}\varsigma
%\Phi(N):=\sum_{n=1}^N \mu(E_n) \, .
 $$
  for $m_d$-almost all  $\vx\in \mathbb{T}^d$, where $\Phi(N):=\sum_{n=1}^N (2 \psi(n))^d$.
In particular,
		\begin{eqnarray*}
		m_d\big(W(T,\psi,\va)\big)=
		\begin{cases}
			0 &\text{if}\ \sum_{n=1}^\infty \psi(n)^d<\infty\\[2ex]
			1 &\text{if}\ \sum_{n=1}^\infty \psi(n)^d=\infty.
		\end{cases}
	\end{eqnarray*}

\end{corollary}

\noindent In fact, if we assume that $\psi(n) \to 0 $ as $n \to \infty$,  we are able to appropriately extend Corollary~\ref{Cor1} to the situation in which the   eigenvalues are in $(-\infty, -g] \, \cup  \,   [-1, +\infty)$.  In other words,  we can incorporate the interval [-1,1] into the allowed range of the eigenvalues.    This is the subject of  \S\ref{metricresultext}  below.
%It is also observed (see Remark~\ref{LS} at the end of  \S\ref{secmetricresult}) that it is natural to exclude the interval  $(-g,-1)$ from the allowable range.

In another direction, if $T$ is an integer matrix transformation we are able to use  a nifty ``reduction'' argument to relax the condition that $T$ is diagonal in Theorem \ref{metricresult} to  $T$ is diagonalizable over $\Z$.   This reduction argument is the subject of \S\ref{redarg} below.   In fact, for integer matrices far more is true.
\begin{theorem}   \label{corInteger}
	Let $T$ be an integer,  non-singular  matrix transformation of the torus $\mathbb{T}^d$.  Suppose that all eigenvalues of $T$ are of modulus strictly larger than $1$ and  and let $\mathcal{C}$ be any  collection of subsets $E$ of $\mathbb{T}^d$ satisfying the bounded property
$ ({\boldsymbol{\rm B}}) $.  Then, for  any sequence $ \{E_n \}_{n \in \mathbb{N}} $ of subsets  in $\mathcal{C}$ and $\varepsilon>0$,
the counting formula \eqref{formula-MT-count} holds for $m_d$-almost all $\vx\in \mathbb{T}^d$, where $\Phi(N):=\sum_{n=1}^N m_d(E_n)$. In particular,
\begin{eqnarray*}
		m_d\big(W(T,\{E_n\})\big)=
		\begin{cases}
			0 &\text{if}\ \  \sum_{n=1}^\infty m_d\big(E_n\big)<\infty\\[2ex]
			1 &\text{if}\ \ \sum_{n=1}^\infty m_d\big(E_n\big)=\infty.
		\end{cases}
	\end{eqnarray*}
\end{theorem}

\medskip

To end with, we illustrate natural  ``number theoretic'' consequences of our results.  Let $ t_1, \ldots, t_d~\ge  ~2 $ be integers and let $T = {\rm diag} \, (t_1, \ldots, t_d  )$.  Then with reference to Remark~\ref{rk-2}, it follows that
$$W(T,\psi,\va)=\{\vx\in [0,1)^d: \max_{1\le i \le d} \|t_i^nx_i-a_i\| \le \psi(n)  \ \ \text{for infinitely many}\ n\in\mathbb{N}\}$$
and
$$W^{\times}(T,\psi,\va)=\{\vx\in [0,1)^d: \prod_{1\le i \le d} \|t_i^nx_i - a_i\| \le \psi(n)  \ \ \text{for infinitely many}\ n\in\mathbb{N}\} \, . $$
Thus,  Theorem~\ref{corInteger} implies the following statement for multiplicative Diophantine approximation.    In fact, since $T$ is diagonal, it is also covered by the ``furthermore part'' of Theorem~\ref{metricresult}.
 %\red{To end with, we describe a ``number theoretic'' consequence of Theorem~\ref{corInteger}.  Let $ t_1, \ldots, t_d~\ge  ~2 $ be integers and let $T = {\rm diag} \, (t_1, \ldots, t_d  )$.  Then with reference to Remark~\ref{rk-2}, it follows that
%$$W^{\times}(T,\psi,\va)=\{\vx\in [0,1]^d: \prod_{1\le i \le d} \|t_i^nx_i - a_i\| \le \psi(n)  \ \ \text{for infinitely many}\ n\in\mathbb{N}\}$$
%and so Theorem~\ref{corInteger} implies the following statement for multiplicative Diophantine approximation.  In fact, since $T$ is diagonal, it is also covered by the ``furthermore part'' of Theorem~\ref{metricresult}.  }

\begin{corollary}\label{Cor2}
Let $ t_1, \ldots, t_d \ge   2 $ be integers and let $T = {\rm diag} \, (t_1, \ldots, t_d  )$.  Let $\psi: \mathbb{R}^+\to \mathbb{R}^+$ be real positive function such that $\psi (x) < 2^{-d}$ and $\va =(a_1, \ldots, a_d  ) \in \T^d$. Then
$$
 \# \Big\{ 1\le n \le   N :  \prod_{1\le i \le d} \|t_i^nx_i - a_i\| \le \psi(n)  \Big\} =
\Phi(N)+O\left(\Phi^{1/2}(N) \ (\log\Phi(N))^{3/2+\varepsilon}\right) %\, ,   \qquad  where  \quad
%%\Psi(N):=\textstyle{\sum_{n=1}^N\psi(q_n)}\varsigma
%\Phi(N):=\sum_{n=1}^N \mu(E_n) \, .
 $$
  for $m_d$-almost all  $\vx= (x_1, \ldots, x_d  )\in \mathbb{T}^d$, where
  $$
 \Phi(N)=\sum_{n=1}^N 2^d \psi(n) \left(\sum_{s=0}^{d-1}\frac{1}{s!}\left(\log\frac{1}{2^{d}
 \psi(n)}\right)^{s}\right)     $$
  In particular,
\begin{eqnarray*}
		m_d\big(W^{\times}(T,\psi,\va)\big)=
		\begin{cases}
			0 &\text{if}\ \sum_{n=1}^\infty \psi(n) \,  \big(\log  \frac{1}{\psi(n)}\big)^{d-1} <\infty\\[3ex]
			1 &\text{if}\ \sum_{n=1}^\infty \psi(n) \,  \big(\log  \frac{1}{\psi(n)}\big)^{d-1}=\infty.
		\end{cases}
	\end{eqnarray*}
\end{corollary}

\noindent

\noindent The analogous statement for the simultaneous set $W(T,\psi,\va)$  is clearly covered by Corollary~\ref{Cor1} above.
The condition that $\psi (x) < 2^{-d}$  is only required for the counting statement.

Corollary~\ref{Cor2} is probably most familiar to number theorists  within the context of when  $t_1= \ldots =t_d$.  This corresponds to approximating arbitrary points $\vx\in [0,1]^d$ by  ``shifted'' rational points $((p_1+a_1)/q, \ldots, (p_d+a_d)/q)$ with  denominators $q$ restricted to an integer lacunary sequence.   In this setup, the zero-full measure criterion within the corollary  can just as easily be deduced from the elegant work of Gallagher \cite{GallP}  mentioned in Remark~\ref{rk-2}.  Also, under the same setup and the assumption that $\psi$ is non-increasing, the corresponding  quantitative version (with a slightly worse error term) can be deduced from \cite[Theorem~4.6]{Harman}.

\medskip

\begin{remark}  \label{Baowei}  For $n \in \N$, let $H_n= H(\va, \psi(n))$ be the hyperboloid region given by \eqref{iona}.   Clearly, Corollary~\ref{Cor2}   follows directly from   Theorem~\ref{corInteger} on letting  $E_n = H_n$ and on showing that \begin{equation}  \label{iona2}
m_d(H_n) =   2^d \psi(n) \left(\sum_{s=0}^{d-1}\frac{1}{s!}
\left(\log\frac{1}{2^{d}\psi(n)}\right)^{s}\right)   \, . \end{equation}
For the sake of completeness we will provide  the details of this measure calculation in \S\ref{volcal}.
\end{remark}

%\begin{remark}  \label{Baowei}  \sv{REMOVE}
%\red{ The  zero-full measure criterion within Corollary~\ref{Cor2} (i.e., the `in particular' part) can be directly obtained from the main ``Gallagher type'' result in \cite[Theorem 1.3]{BWpre}.
%}
%\end{remark}

\subsection{Dimension results for matrix transformations}

We address the shrinking target `dimension'  problem (P2) in the case $T$ is a self-map of  the $d$-dimensional torus  $\T^d$ and the target sets  are a sequence $\{B_n \}_{n \in \mathbb{N}} $ of balls as in the original formulation of the problem.   The following two theorems constitute our main dimension results. It turns out that these statements for balls can be exploited   to determine the dimension of shrinking targets sets in the case the targets are  a sequence $\{H_n \}_{n \in \mathbb{N}} $  of hyperboloids.  Throughout, given a real positive function $\psi: \R^+ \to \R^+$  we let $\lambda =\lambda(\psi)$ denote its lower order  at infinity; that is $$\lambda=\lambda(\psi):=\liminf_{n\to\infty}\frac{-\log\psi(n)}{n}.$$

\begin{theorem}\label{dimresult}
	Let $T$ be a real, non-singular matrix transformation of the torus $\mathbb{T}^d$. Suppose that $T$ is diagonal with all  eigenvalues $\beta_1,\beta_2,\dots,\beta_d$  strictly larger than $1$.  Assume that $1<\beta_1\le \beta_2\le\cdots\le \beta_d$. Let $\psi: \mathbb{R}^+\to \mathbb{R}^+$ be a real positive  function and $\va \in \mathbb{T}^d$. Then
	$$\dim_{\rm H} W(T,\psi, \va)=\min_{1\le   i\le   d}\theta_i(\lambda),$$
	where

	$$
	\theta_i(\lambda):=\frac{i\log \beta_i-\sum\limits_{k:\beta_k>\beta_ie^{\lambda}}(\log \beta_j-\log \beta_i-\lambda) +\sum\limits_{k>i}\log \beta_j}{\lambda+\log \beta_i}  \, .
	$$
	
\end{theorem}

\bigskip

\begin{remark}  \label{rectlove}  We will in fact deduce the above theorem from a more general statement concerning  rectangular target sets -- see  Theorem~\ref{rectangledimresult} in \S\ref{pfthm6}.  \end{remark}

In the case $d=1$, the above result corresponds to the main result in  \cite{SW2013}.   It turns out that while we are currently unable to prove in full generality the analogue of Theorem \ref{dimresult}  that incorporates negative eigenvalues, we can do so in the one dimensional case.   Thus, the following statement  for $\beta<-1$ is  new and extends the work of Shen $\&$ Wang \cite{SW2013} from positive to arbitrary  $\beta$-transformations  $T_{\beta}$.

\begin{theorem}\label{dimensionalone}
  Let $\beta$ be a real number with $|\beta|>1$ and $K(\beta)$ be the support of the associated Parry-Yrrap measure. Let $\psi: \mathbb{R}^+\to \mathbb{R}^+$ be a real positive decreasing function and  $a \in K(\beta)$.
  Then
	$$\dim_{\rm H} W(T_\beta,\psi, a)=\frac{\log|\beta|}{\lambda+\log|\beta|}  \, . $$
\end{theorem}

\medskip \noindent  The proof of Theorem~\ref{dimensionalone} makes use of a general approximation technique for any  piecewise linear map of the unit interval  with constant slope.  The associated result (Proposition \ref{Markov} in \S\ref{dimensionalone})  may prove to be useful for other problems.

\bigskip

 We now mention two consequences of our main dimension theorems. The first is that if $T$ is  an integer matrix transformation, then  in Theorem~\ref{dimresult} we can replace the condition that $T$ is diagonal by $T$ is diagonalizable over $\Z$.

  \begin{theorem}\label{dimresultinteger}
	Let $T$ be an integer, non-singular matrix transformation of the torus $\mathbb{T}^d$. Suppose that $T$ is  diagonalizable over $\Z$ with all  eigenvalues $\beta_1,\beta_2,\dots,\beta_d$  strictly larger than $1$.  Assume that $1<\beta_1\le \beta_2\le\cdots\le \beta_d$. Let $\psi: \mathbb{R}^+\to \mathbb{R}^+$ be a real positive function and $\va \in \mathbb{T}^d$. Then
	$$\dim_{\rm H} W(T,\psi, \va)=\min_{1\le   i\le   d}\theta_i(\lambda)  \, . $$
\end{theorem}

The theorem follows from Theorem~\ref{dimresult} by using a ``reduction'' argument -- see \S\ref{iwill}.  The second is that the above theorems for balls enables us to establish  the dimension of the  multiplicative set  $W^\times (T,\psi, \va)$. In fact, we only require the $d=1$ statement and so we are able to utilize the more general Theorem~\ref{dimensionalone}.

\begin{theorem}\label{multiplicative}
Let $T$ be a real, non-singular matrix transformation of the torus $\mathbb{T}^d$. Suppose that $T$ is diagonal and  all  eigenvalues $\beta_1,\beta_2,\dots,\beta_d$  are of modulus strictly larger than $1$.  Assume that $1<|\beta_1|\le   |\beta_2|\le   \cdots \le   |\beta_d|$. Let $\psi: \mathbb{R}^+\to \mathbb{R}^+$ be a real positive decreasing function and $\va \in \mathbb{T}^d$ with $a_d\in K(\beta_d)$. Then
$$\dim_{\rm H} W^\times (T,\psi, \va)=d-1+\frac{\log|\beta_d|}{\lambda+\log|\beta_d|}  \, . $$
\end{theorem}

\noindent  This consequence of Theorem~\ref{dimensionalone} was pointed out to us by  Baowei Wang.  We thank him for sharing his insight and indeed for providing the details of the proof which forms the appendix. We stress that  making use of Theorem~\ref{dimensionalone},  rather than the previously known $d=1$ case of Theorem~\ref{dimresult} due to Shen $\&$ Wang \cite{SW2013} is crucial. The latter  requires  that all the eigenvalues are positive and strictly larger than one and would thus yield a weaker version of Theorem~\ref{multiplicative}.
%\noindent \red{ In the case that all the eigenvalues are positive and strictly larger than one, the above theorem follows from the main ``Jarn\'{\i}k type'' result in \cite[Theorem 1.9]{BWpre}.   The proof of  Theorem~\ref{multiplicative} follows essentially the same ``standard''  argument as that appearing in  \cite{BWpre} and will not be reproduced in this paper.   The difference is that we have Theorem~\ref{dimensionalone} at our disposal whereas in \cite{BWpre} the authors  did not and so they appeal instead to the known $d=1$ case of Theorem~\ref{dimresult} due to Shen $\&$ Wang \cite{SW2013}.}

\section{Establishing Theorem~\ref{gencountthm} and Corollary~\ref{gencountcor} \label{bcsv}}

 The following  statement~\cite[Lemma~1.5]{Harman} represents an important  tool in the theory of metric Diophantine approximation for establishing counting statements.  It has its bases in the familiar variance method of probability theory and can be viewed as the quantitative form of the (diveregnce)  Borel-Cantelli Lemma \cite[Lemma~2.2]{durham}.

\begin{lemma} \label{ebc}
Let $(X,\cal{B},\mu)$ be a probability space, let $(f_n(x))_{n \in \N}$ be a sequence of non-negative $\mu$-measurable functions defined on $X$, and $(f_n)_{n \in \N },\ (\phi_n)_{n  \in \N}$ be sequences of real numbers  such that
$$ 0\le   f_n \le   \phi_n \hspace{7mm} (n=1,2,\ldots).  $$

\noindent %Write
%$$
%\Phi(N)= \sum\limits_{n=1}^{N}\phi_n
%$$
%and suppose that $\Phi(N) \to \infty$  as   $ N \to \infty$.
Suppose that for arbitrary  $a,b \in \N$ with $a <  b$, we have
\begin{equation} \label{ebc_condition1}
\int_{X} \left(\sum_{n=a}^{b} \big( f_n(x) -  f_n \big) \right)^2\mathrm{d}\mu(x)\, \le  \,  C\!\sum_{n=a}^{b}\phi_n
\end{equation}

\noindent for an absolute constant $C>0$. Then, for any given $\varepsilon>0$,  we have
\begin{equation} \label{ebc_conclusion}
\sum_{n=1}^N f_n(x)\, =\, \sum_{n=1}^{N}f_n\, +\, O\left(\Phi(N)^{1/2}\log^{\frac{3}{2}+\varepsilon}\Phi(N)+\max_{1\le   k\le   N}f_k\right)
\end{equation}
\noindent for $\mu$-almost all $x\in X$, where $
\Phi(N):= \sum\limits_{n=1}^{N}\phi_n
$. %(\textbf{Harman has $\Phi(N) \to \infty$ as   $ N \to \infty$}).
\end{lemma}

\medskip

Note that in statistical terms, if the sequence $f_n$ is the mean of $f_n(x)$; i.e.
$$
f_n = \int_{X} f_n(x) \mathrm{d}\mu(x) \, ,
$$
then the l.h.s.~of~\eqref{ebc_condition1} is simply the variance ${\rm Var}(Z_{a,b}) $ of the random variable  $$Z_{a,b}=Z_{a,b}(x):=\sum_{n=a}^{b}f_n(x) \, . $$ In particular,
$$
 {\rm Var}(Z_{a,b})  = \bE(Z^2_{a,b})  -  \bE(Z_{a,b})^2
$$
where
$$
 \bE(Z_{a,b}) = \int_X Z_{a,b}(x) \mathrm{d}\mu(x)  \, .
$$

\noindent The following extremely useful classical inequality that bounds the probability that a  random variable is small, in terms of its expectation and second moment,   is a well know  consequence of the Cauchy-Schwarz inequality.

\begin{lemma}[Paley-Zygmund Inequality] \label{PZ}
Let $(X,\cal{B},\mu)$ be a probability space and $Z$ be  a non-negative random variable. Then for any  $0 < \lambda < 1 $, we have that
$$ \mu\Big(\{ x \in X :Z(x)>\lambda\mathbb{E}(Z)\} \Big)  \, \ge   \,  (1-\lambda)^2   \ \frac{\mathbb{E}(Z)^2}{\mathbb{E}(Z^2)} \, .   $$
\end{lemma}

We shall see that a straightforward application of Lemma~\ref{PZ} leads to a direct proof of Corollary~\ref{gencountcor}.  Deducing  Theorem~\ref{gencountthm} from Lemma~\ref{ebc} is also pretty straightforward.

\bigskip

%\subsection{Proof of Theorem~\ref{gencountthm}}

\begin{proof}[\textbf{Proof of Theorem~\ref{gencountthm}}] Given a sequence $\{E_n \}_{n \in \mathbb{N}} $ of subsets in $\mathcal{C}$, we consider Lemma~\ref{ebc} with

\begin{equation} \label{Harman_choice_parameters}
 f_n(x):= {\chi}_{{T^{-n}(E_{n})}} (x)  =  {\Large\chi}_{_{E_{n}}} (T^n(x)) \quad {\rm and } \quad
 f_n:= \phi_n :=  \mu (E_n) \, ,
\end{equation}
where  $ \chi_{_{E}} $   is the characteristic function  of the set  $E \subset X $.   Then, clearly for any $ x \in X$ and $N \in \N$,  we have that the
$$
{\rm l.h.s. \ of \ } \eqref{ebc_conclusion}  \ = \ \# \big\{ 1\le n \le   N :    T^{n}(x) \in E_n \} :=  \ R(x, N) \,
$$
and so  \eqref{ebc_conclusion}  and  \eqref{def_Psi} coincide.  Thus, to complete the proof of Theorem~\ref{gencountthm} we need to verify that  \eqref{ebc_condition1} is satisfied.  Note that by definition, $f_n$ is the mean of $f_n(x) $ and so

\begin{equation}  \label{ohyest}
{\rm l.h.s. \ of \ } \eqref{ebc_condition1}  \ = \ {\rm Var}(Z_{a,b}) \ = \ \bE(Z_{a,b}^2)-\bE(Z_{a,b})^2  \,
\end{equation}
where
\begin{itemize}
  \item[$\bullet$] ${\rm Var}(Z_{a,b}) $ is variance of the random variable
\begin{equation} \label{proof_combined_first_moment_ub}
Z_{a,b}=Z_{a,b}(x)=\sum\limits_{n=a}^b\chi_{E_n}(T^n(x))=\sum\limits_{n=a}^b\chi_{T^{-n}E_n}(x),
\end{equation}
  \item[$\bullet$]   the expectation
\begin{equation} \label{exp_ub}
\mathbb{E}(Z_{a,b})=\sum\limits_{n=a}^b\mu(E_n) \, ,
\end{equation}
  \item[$\bullet$] and  the second moment
\begin{equation} \label{proof_combined_second_moment_ub}
\begin{aligned}
\mathbb{E}(Z_{a,b}^2)&= \sum_{a\le m, n\le b}\mu\left(T^{-m}(E_m)\cap T^{-n}(E_n)\right) \\[2ex] & =   \sum_{a\le n\le b}\mu(E_n)+2\sum_{a\le m< n\le b}\mu\left(E_m\cap T^{-(n-m)}(E_n)\right).
\end{aligned}
\end{equation}
\end{itemize}

\noindent By making use of the  \s-mixing property \eqref{phimixing}, it follows that
\begin{eqnarray*}
\sum_{a\le m< n\le b}\mu\left(E_m\cap T^{-(n-m)}(E_n)\right) & \le   & \sum_{a\le m< n\le b} \!\! \mu(E_m)\mu(E_n) \ +\sum_{a\le m< n\le b} \!\!\! \phi(n-m) \ \mu(E_n) \nonumber \\[2ex]
& \le   & \sum_{a\le m< n\le b} \!\! \mu(E_m)\mu(E_n) \ + \sum_{a\le   n \le   b }\left(\sum_{a\le   m<n}\!\!\! \phi(n-m) \right)\mu(E_n)  \nonumber \\[2ex]
& \le   & \sum_{a\le m< n\le b} \!\! \mu(E_m)\mu(E_n) \ +  \  \kappa \sum_{a\le   n \le   b }\mu(E_n)  \nonumber
\end{eqnarray*}
where  $ \kappa := \sum_{n=1}^\infty\phi(n) < \infty$.  This together with \eqref{proof_combined_second_moment_ub} implies  that
\begin{eqnarray}\label{countcor_proof_double_count}
\mathbb{E}(Z_{a,b}^2)
& \le   &  (2\kappa +1) \sum_{a\le   n \le   b }\mu(E_n) \ +  \ 2\, \sum_{a\le m< n\le b} \!\! \mu(E_m)\mu(E_n) \nonumber  \\[2ex]
& \le   &  (2\kappa +1) \sum_{a\le   n \le   b }\mu(E_n) \ +  \  \left(\sum_{a\le n\le b} \!\! \mu(E_n) \right)^2   \, .
\end{eqnarray}
The upshot of \eqref{ohyest}, \eqref{exp_ub} and \eqref{countcor_proof_double_count} is that
$$
{\rm Var}(Z_{a,b}) \ \le   \ (2\kappa +1) \sum_{a\le   n \le   b }\mu(E_n)  \, .
$$
This verifies \eqref{ebc_condition1}  with $C= 2\kappa +1$ and thereby completes the proof of Theorem~\ref{gencountthm}.
 \end{proof}

\bigskip

\begin{proof}[\textbf{Proof of Corollary~\ref{gencountcor}}] In view of \eqref{appconv},  we assume that the sum in \eqref{appdiv} diverges.  With the same notation as in the proof of Theorem~\ref{gencountthm}, we start with the  observation that for any $\lambda>0$
\begin{equation}\label{really}
  \mu\Big(W(T,\{E_n\})\Big)  \ \ge \  \mu\Big(\limsup_{b\to\infty}(Z_{1,b}>\lambda\mathbb{E}(Z_{1,b}))\Big)\ \ge \  \limsup_{b\to\infty}\mu\Big(Z_{1,b}>\lambda\mathbb{E}(Z_{1,b})\Big) \, . \end{equation}
To estimate the measure on the r.h.s. we use the Paley-Zygmund inequality (Lemma~\ref{PZ}) and the estimates \eqref{exp_ub} and \eqref{countcor_proof_double_count}.  With this in mind,  for any  $0 < \lambda < 1 $, it follows that
\begin{eqnarray*} \mu\Big(Z_{1,b}>\lambda\mathbb{E}(Z_{1,b})\Big)  & \ge   & (1-\lambda)^2   \ \  \frac{\mathbb{E}(Z_{1,b})^2}{\mathbb{E}(Z_{1,b}^2)} \\[2ex] & \ge & (1-\lambda)^2 \  \frac{\Big(\sum\limits_{1\le n\le b}\mu(E_n)\Big)^2}{\Big(\sum\limits_{1\le n\le b}\mu(E_n)\Big)^2+(2\kappa +1)\sum\limits_{1\le n\le b}\mu(E_n)}  \, .
\end{eqnarray*}

\noindent By the divergent sum hypothesis, on letting $b\to\infty$ and $\lambda\to 0$,
	we obtain that $$ \limsup_{b\to\infty}\mu\Big(Z_{1,b}>\lambda\mathbb{E}(Z_{1,b})\Big)=1 \, . $$
which together with  \eqref{really} completes the proof of the corollary.
\end{proof}

\medskip

\section{Establishing  measure results for matrix transformations }
\subsection{Proof of Proposition \ref{mainohyes}} \label{Sec:1.3}

%Saussol (\cite[Theorem 5.1]{S2000}) proved that for a piecewise expanding map $T:X \longrightarrow X$, where $X\subset \mathbb{R}^d$ is a compact set, there exists an absolutely continuous invariant probability measure $\mu$. Further, {\color{blue} the support of $\mu$ can be decomposed into finitely many disjoint measurable sets, on each of such set $\mu$ is mixing with respect to $T^p$}.

Let $X\subset \mathbb{R}^d$ be a compact set.   The proof of Proposition~\ref{mainohyes} makes essential use of the work of Saussol \cite{S2000} for general piecewise expanding maps $T:X \longrightarrow X$ on $X$.  Clearly, our particular case  in which $X=[0,1]^d$ and the real, non-singular matrix $T$ (with the modulus of all eigenvalues  strictly larger than one) that sends $\vx$ to $T\vx \ {\rm mod} \ 1$  defines a piecewise expanding map on $X$.  In  what follows will state and apply Saussol's results to our setup.  So with this in mind,
the first three parts, apart  for the equivalence of the restricted measures $\mu|_{A_i}$ and $m_d|_{A_i}$ in part (ii), follow from \cite[Theorem 5.1]{S2000}.  To prove the equivalence of the restricted measures we note that by \cite[Proposition 5.1]{S2000}, the Randon-Nykodym derivative $f$ of $\mu|_{A_i}$ with respect to $m_d|_{A_i}$ is $m_d$-almost surely strictly positive. Hence,  for any measurable subset $E$, if $m_d|_{A_i}(E)>0$, then
\[
\mu|_{A_i}(E)=\int_E f(\vx) m_d|_{A_i}(d\vx)>0
\]
and so $m_d|_{A_i} \ll \mu|_{A_i}$. The other direction follows directly from part (i).

It remains to prove part (iv) of the proposition.  Without loss of generality, we will assume that $A_{ij}$ is $A$ and thus $\mu|_{A_{ij}} = \mu$  and also  $\mu$ is mixing with respect to $T$.    The key to proving part (iv) is to use the fact  that the acim $\mu$  satisfies the property of exponential decay of correlations.
%It is worth pointing out that the  existence of the acim $\mu$ (i.e., the first part of the proposition)  can be obtained as a direct consequence of combining Theorems 1 $\&$ 3 of Buzzi \cite{B1997}.  The work of Buzzi is specifically for expanding affine maps of the torus unlike Saussol \cite{S2000} which is for general  piecewise expanding maps.
%Let $N\ge 1$ be an integer. We will work in the Euclidean space $\mathbb{R}^N$. We denote by $B_\epsilon(x)$ the ball with center $x$ and radius $\epsilon$.
With this mind, for a set $E\subset\mathbb{R}^d$,
%we write
%$$B_\epsilon(E):=\{y\in\mathbb{R}^d: \sup_{x\in E} |x-y|\le \epsilon\},$$ and
define the {\it oscillation} of $\varphi \in L^1(m_d)$ over $E$ as
$${\rm osc}(\varphi,E):= \underset{ E}{ \textrm{ess-sup} } (\varphi)- \underset{ \ E}{ \textrm{ess-inf} } (\varphi).$$
For given real numbers $0<\alpha \le 1$ and $0<\epsilon_0<1$, define the following {\it $\alpha$-seminorm}
$$|\varphi|_\alpha:=\underset{ 0<\epsilon\le\epsilon_0}{ \sup \  } \epsilon^{ -\alpha} \int_{ \mathbb{R}^d}{\rm osc}(\varphi,B(x, \epsilon))dx.$$
%We observe (see \cite{Saussol1}) that $X\ni x \mapsto {\rm osc}(\varphi,B_\epsilon(x))$ is a measurable function and
%$$\text{supp} (\text{osc}(\varphi,B_\epsilon(x)))\subset B_\epsilon(\text{supp } \varphi).$$
Let $V_{\alpha}$ be the space of $L^1(m_d)-$functions such that $|\varphi|_\alpha<\infty$ endowed with the norm
$$\|\varphi\|_{\alpha}:=\|\varphi\|_{L^1(m_d)}+|\varphi|_\alpha.$$
Then
$(V_{\alpha},\|\cdot\|_{\alpha})$ is a Banach space which does not depend on the choice of $\epsilon_0 $  and
 $V_{\alpha} \subset L^\infty (m_d)$  (see \cite[Section 3]{S2000}).
Since the acim $\mu$ is mixing with respect to $T$, it follows from \cite[Theorem 6.1]{S2000} that there exist constants $C>0$ and $0\le   \gamma<1$ such that for all $\psi\in V_{\alpha}$, for all $\phi\in L^1(\mu)$ and for all $n\in\mathbb{N}$,  we have
\begin{align}\label{decay}
\left|\int_{[0,1]^d}  \!\!\!\!\! \psi.\phi\circ T^n\, d\mu \ - \  \int_{[0,1]^d} \!\!\!\!\!\!\! \psi \ d\mu\int_{[0,1]^d}\!\!\!\!\!\!\!\phi \ d \mu\right| \ \le \ C \ \|\psi\|_{\alpha}  \ \|\phi\|_1  \ \gamma^{n}.
\end{align}

\noindent Now, let $\mathcal{C}$ be any  collection of subsets of $A$ satisfying the bounded property
$ ({\boldsymbol{\rm B}}) $.  Take  $\psi=\chi_E$ and $\phi=\chi_F$ with $E , F \in\mathcal{C}$   and assume for the moment that
\begin{align}\label{exp-mixing-E}
\sup_{E\in \mathcal{C}}\|\chi_E\|_{\alpha}  <\infty.
\end{align}
Note that \eqref{exp-mixing-E} implies that  $\psi=\chi_E\in V_\alpha$  and thus together with (\ref{decay}), we obtain  that % Then the exponentially mixing is satisfied, if we can show
\begin{align*}
\left|\mu(E\cap T^{-n}F)-\mu(E)\mu(F)\right|\le C \cdot \|\chi_E\|_{\alpha} \cdot \mu(F) \cdot \gamma^{n} \le    C \cdot \Big(\sup_{E\in \mathcal{C}} \|\chi_E\|_{\alpha}\Big) \cdot \gamma^n \cdot \mu(F).
\end{align*}
Hence the exponentially mixing property (\ref{mixing}) is satisfied for the collection $\mathcal{C}$.  This proves part (iv) modulo \eqref{exp-mixing-E}.

%In fact, let $E=B(y,r)$ be a ball centered at $y$ with radius $r$. Then for $\epsilon\ge r$, we have
%\[
%{\rm osc}(\chi_{B(y,r)}, B(x, \epsilon)) \le \chi_{B(y, r+\epsilon)}(x),
%\]
%and for $\epsilon<r$, we have
%\[
%{\rm osc}(\chi_{B(y,r)}, B(x, \epsilon)) \le \chi_{B(y, r+\epsilon) \setminus B(y, r-\epsilon)}(x).
%\]
%Thus for $\epsilon\ge r$,
%\[
%\epsilon^{-\alpha} \int_{\mathbb{R}^d} {\rm osc}(\chi_{B(y,r)},B(x, \epsilon))dx \le C_0 \epsilon^{-\alpha}\cdot (r+\epsilon)^d \le C_0 2^d \epsilon^{d-\alpha} \le C_0 2^d \epsilon_0^{d-\alpha},
%\]
%and for $\epsilon<r$,
%\begin{align*}
%\epsilon^{-\alpha} \int_{\mathbb{R}^d} {\rm osc}(\chi_{B(y,r)}, B(x, \epsilon))dx &\le C_0 \epsilon^{-\alpha} \big((r+\epsilon)^d-(r-\epsilon)^d \big)\\
%&\le 2 C_0 \epsilon^{-\alpha} \sum_{k=1}^{d} r^{d-k}\epsilon^k {d \choose k} \\
%& \le 2 C_0 \epsilon^{-\alpha} \cdot \epsilon 2^d= 2^{d+1} C_0 \epsilon^{1-\alpha} \le 2^{d+1} C_0 \epsilon_0^{1-\alpha}.
%\end{align*}
%Hence $|\chi_E|_\alpha < \infty$ and $\chi_E\in V_\alpha$.

We now prove  \eqref{exp-mixing-E}. To start with, observe  that  for the characteristic  function $\chi_E$, the oscillation can only be non-zero on the boundary $\partial E$ of $E$.  It can be verified that for any $\epsilon>0$,
\[
{\rm osc}(\chi_E, B(x, \epsilon)) \, \le  \,  \chi_{(\partial E)(\epsilon)}(x),
\]
where $(\partial E)(\epsilon)$    is the $\epsilon $-neighborhood of $\partial E$.
Thus,
\[
\epsilon^{-\alpha} \int_{\mathbb{R}^d} {\rm osc}(\chi_{E},B(x, \epsilon))dx \le \epsilon^{-\alpha}\cdot m_d((\partial E)(\epsilon)).
\]
By the bounded property $ ({\boldsymbol{\rm B}}) $ imposed on $\mathcal{C}$, there exists a constant $C_1$ such that the $(d-1)$-dimensional upper Minkowski content of $ \partial E$
\[
M^{*(d-1)}(\partial E) \le   C_1,
\]
for all $E\in \mathcal{C}$.
Hence, by the definition of $M^{*(d-1)}$, there exists a constant $C_0$ such that for all $E\in \mathcal{C}$
\[
m_d((\partial E)(\epsilon)) \, < \, C_0 \epsilon.
\]
Consequently, for all $E\in \mathcal{C}$,
\[
|\chi_E|_\alpha=\underset{ 0<\epsilon\le\epsilon_0}{ \sup \  } \epsilon^{ -\alpha} \int_{ \mathbb{R}^d}{\rm osc}(\chi_E,B(x, \epsilon))dx \, \le \, C_0\epsilon_0^{1-\alpha}.
\]
On the other hand, for $E\subset [0,1]^d$, we have that
\[
\|\chi_E\|_{L^1(m_d)}= m_d(E)\le   1.
\]
Thus, for all $E\in \mathcal{C}$, it follows that
\[
\| \chi_E\|_\alpha \le   1+ C_0\epsilon_0^{1-\alpha}.
\]
%$\chi_E \in V_\alpha$ and have uniformly bounded $\| \cdot \|_\alpha$ norm.
 Therefore, (\ref{exp-mixing-E}) holds and this  completes the proof of the proposition.
\hfill  $\square$

\bigskip

\begin{remark}
For the sake of completeness, we mention that in the case that $T$ is an integer,  non-singular,  matrix transformation of the torus $\mathbb{T}^d$  with all eigenvalues in  modulus strictly larger than one,   Fan \cite{F1999} proved the exponential decay of correlation formula (\ref{decay}). Also in the case  $d=1$, the first three parts of Proposition~\ref{mainohyes} coincide with the Main {Result} of Wagner~\cite{W1979}
\end{remark}

\subsection{Proof of Theorems \ref{measure-matrix},  \ref{cor-bigger-eigen} and \ref{corInteger}  \label{section3}  }

\begin{proof}[Proof of Theorem~\ref{measure-matrix}]
%The convergent part of the theorem simply  makes use of  the fact that the measure $\mu$ under consideration is invariant under the action of $T$  and so \eqref{appconv} is applicable.  For the ``divergent'' part,
In view of parts (i) to (iii) of  Proposition~\ref{mainohyes}, if the acim $\mu$ has $\T^d$ as its support and is mixing with respect to $T$, then $\mu$ has one mixing component (namely the whole space $\mathbb{T}^d$) of period one. %By the uniqueness of ergodic decomposition, $\mathbb{T}^d$ is the unique mixing (ergodic) component of $\mu$ \red{MAKE CLEARER -- HOW IS Prop~\ref{mainohyes} USED??}.
Furthermore, by part (iv) of Proposition \ref{mainohyes}, on this unique mixing component, $\mu$ is exponentially mixing with respect to $(T,\mathcal{C})$ for any  collection $\mathcal{C}$ of subsets of $\mathbb{T}^d$ satisfying the bounded property $ ({\boldsymbol{\rm B}}) $. The desired  counting part \eqref{formula-MT-count} of the theorem and the  zero-full measure criteria \eqref{formula-measure-matrix}
with respect to the measure $\mu$ now  immediately follow on  applying Theorem~\ref{gencountthm} and Corollary~\ref{gencountcor} respectively.
 To complete the proof of  Theorem~\ref{measure-matrix}, it remains to prove that the measures  $\mu$  and  $m_d$ share the same zero and full  measure sets (note that we have already shown above that $\mu\big(W(T,\{E_n\})\big)$ is either 0 or 1).  This follows directly from part (ii) of  Proposition \ref{mainohyes} since it implies that $\mu$ is equivalent to $m_d$ and it is easily seen that
equivalent measures share the same zero and full  measure sets.
\end{proof}

\medskip

The following lemma will be used in establishing Theorems~\ref{cor-bigger-eigen} and \ref{corInteger}.

\medskip

\begin{lemma}\label{lem:uni-acim}
Let $T$ be a real, non-singular matrix transformation of the torus $\mathbb{T}^d$. Suppose that (i) all eigenvalues of {$T$} are of modulus strictly larger than $1+\sqrt{d}$ or  that (ii) $T$ is integer and all eigenvalues of $T$ are of modulus strictly larger than $1$. %Let $T$ be a matrix in Theorem~\ref{cor-bigger-eigen}, or in Theorem ~\ref{corInteger}.
Then there is a unique acim $\mu$. Furthermore, such an acim $\mu$ has $\T^d$ as its support and is of maximal entropy.
\end{lemma}

\begin{proof} We are given that all eigenvalues of $T$ are of modulus strictly larger than $1$.   Thus, by part (i) of Proposition \ref{mainohyes}, there exists an acim $\mu$.   By \cite[Proposition 1]{B1997}, if all eigenvalues of $T$ have modulus strictly larger than $1+\sqrt{d}$ or if $T$ is integers, then the dynamical system $(\mathbb{T}^d, T)$ is topological transitive. % {\color{blue} topological mixing and hence topological transitive}.
Hence, by  \cite[Theorem 3 and its Corollary]{B1997} the acim $\mu$ has the whole space $\T^d$ as its support and is the unique maximal entropy measure.  %By Theorem  \ref{measure-matrix}, we need only to show that
\end{proof}

We now prove Theorems~\ref{cor-bigger-eigen} and \ref{corInteger}.
\begin{proof}[Proof of Theorem~\ref{cor-bigger-eigen}]

By \cite[Lemma 5]{B1997}, if the eigenvalues of $T$ are all of modulus strictly larger than $1+\sqrt{d}$, then $T$ is locally eventually onto. Thus the unique maximal entropy acim $\mu$ coming from  Lemma~\ref{lem:uni-acim} which has $\T^d$ as its support, is exact (see for example \cite[Theorem 5.2.12]{PU2010}) and hence mixing with respect to $T$ (see \cite[Proposition 12.2]{PY1998}).   In other words, this measure $\mu$ satisfies the hypotheses of Theorem~\ref{measure-matrix}.  On  applying Theorem~\ref{measure-matrix}, we obtain Theorem~\ref{cor-bigger-eigen}.
\end{proof}

%Since for the integer coefficients case, the Lebesgue measure is invariant, mixing and of maximal entropy (see for example \cite[Section 1.1 and Theorems 1.11, 1.28, $\&$ 8.15]{W-GTM79}), we can conclude that the Lebesgue measure $m_d$ is the unique maximal entropy mixing acim. By Theorem \ref{measure-matrix}, $m_d$ is of exponentially mixing with respect to
%any  collection $\mathcal{C}$ of subsets $E$ of $\mathbb{T}^d$ satisfying the bounded property
%$ ({\boldsymbol{\rm B}}) $ and Theorem~\ref{corInteger} follows.

\begin{proof}[Proof of Theorem~\ref{corInteger}]

Observe that if $T$ is integer, then $T$ is an endomorphism of the torus $\T^d$ (see for example, \cite[Theorem 0.15]{W-GTM79}). By  \cite[Corollary 1.10.1 and Theorem 1.28]{W-GTM79}, the Lebesgue measure $m_d$  is mixing with respect to $T$. %$T$ has no roots of unity as eigenvalues if and only if $T$ is mixing with respect to the Lebesgue measure $m_d$.
Furthermore, by  \cite[Theorem 8.15]{W-GTM79}, $m_d$ is of maximal entropy.
Thus $m_d$ is nothing but the unique maximal entropy acim $\mu$ in Lemma~\ref{lem:uni-acim}. %$m_d$ is of exponentially mixing with respect to any  collection $\mathcal{C}$ of subsets $E$ of $\mathbb{T}^d$ satisfying the bounded property $ ({\boldsymbol{\rm B}}) $
Hence, Theorem~\ref{corInteger} follows on applying  Theorem~\ref{measure-matrix} with $\mu = m_d$.
\end{proof}

\subsubsection{Proof of Corollary~\ref{Cor2} \label{volcal}}

 As noted in Remark~\ref{Baowei}, the corollary follows directly from  Theorem~\ref{corInteger} on showing that the $d$-dimensional Lebesgue measure $m_d$ of the hyperboloid region $ H(\va, \psi(n))$ satisfies \eqref{iona2}.  It is easily versified that $ m_d (H(\va, \psi(n))) $ is independent of the `shift' $\va \in \T^d$. So with this in mind, it suffices to prove the following statement.

\begin{lemma}
Given  $d\in\N$ and $\delta>0$, let
\begin{equation} \label{ie_hyperbola}
H_d(\delta):=\big\{(x_1,\dots,x_d)\in [0, 1)^d: \|x_1\|\dots\|x_d\| <\delta \big\}  \, .
\end{equation}
Then
\begin{equation} \label{result_Hk}
m_d(H_d(\delta))=\begin{cases}
1 &\text{ if }   \  \delta \ge   2^{-d}\\[2ex]
2^{d}\delta\left(\sum_{t=0}^{d-1}\frac{1}{t!}\left(\log\frac{1}{2^{d}\delta}\right)^{t}\right) &\text{ if }   \  \delta < 2^{-d}  \, .
\end{cases}
\end{equation}
\end{lemma}
\begin{proof}
To simplify computations, first note  that the measure  of $H_d(\delta)$ is equal to $2^d$ times the measure of $ [0,1/2]^d  \cap H_d(\delta) $.  Furthermore,  it is technically simpler to work with points restricted to  $[0,1/2]^d$ since the inequality under consideration  is equivalent to
\begin{equation} \label{ie_hyperbola_half}
%\frac{1}
x_1\dots x_d\le  \delta.
\end{equation}
So, from now on we will focus on computing the measure of the set
\begin{equation} \label{def_Vk}
V_d(\delta):=\big\{ (x_1,\dots,x_d)\in [0,1/2]^d  \, : \,  {x_1\dots x_d}\le  \delta\big\}    \,
\end{equation}
and recall that
\begin{equation} \label{hypersv}
2^d   \, m_d \big( V_d(\delta)\big)  \ =   \ m_d \big( H_d(\delta)  \big) \, .
\end{equation}

\bigskip

%Suppose, that  $\delta  \ge  2^{-d} $. Then it is a easily versified that $V_d(\delta) = [0,1/2]^d$  and this together with \eqref{hypersv}  implies  \eqref{result_Hk}.   So without loss of generality assume that $\delta  < 2^{-d} $.

\noindent \emph{Case \!(a):  if  $\delta  \ge  2^{-d} $.}  Then it is a easily versified that $V_d(\delta) =
[0,1/2]^d$  and this together with \eqref{hypersv}  implies  \eqref{result_Hk}.   So, without loss of generality  we can assume that $\delta  < 2^{-d} $.

\medskip

\noindent \emph{Case \!(b): if   $\delta  < 2^{-d} $.}   In view of \eqref{hypersv},  to  establish \eqref{result_Hk} we need to show that for any $d\in\N$ and $ 0< \delta   <  2^{-d} $
\begin{equation} \label{computation_Vk}
m_d\big( V_d(\delta) \big)  \, = \, \delta \ \sum_{t=0}^{d-1}\frac{1}{t!}
\left(\log\frac{1}{2^{d}\delta}\right)^{t}  \, .
\end{equation}
This we now do  by induction on $d$.  For $d=1$, we have that
\begin{equation} \label{computation_V1}
m_1\big(V_1(\delta) \big) =  m_1\big(    \left\{x\in[0,1/2]  \, : \, x\le\delta\right\}   \big)  \, = \, \delta \,
\end{equation}
and this coincides with \eqref{computation_Vk}.   Now let $d \ge 2$ and observe that we can rewrite~\eqref{ie_hyperbola_half} as
\begin{equation} \label{ie_hyperbola_changed}
x_1\dots x_{d-1}\le  \delta/x_d.
\end{equation}
Note that since  $ (x_1,\dots,x_d)\in [0,1/2]^d$ ,  the left hand side of~\eqref{ie_hyperbola_changed} is not bigger than $(1/2)^{d-1}$. Hence, it follows that  for any $ 0< x_d \le 2^{d-1}\delta$,  inequality \eqref{ie_hyperbola_changed} is satisfied   for all $0\le   x_1,\dots,x_{d-1}\le   1/2$.  The $m_d$-measure of the set of such points $ (x_1,\dots,x_d)$   is thus equal to $2^{-d+1}  \times  2^{d-1}\delta = \delta $.  On the other hand, for
 any  fixed value of $x_d \in (2^{d-1}\delta, 1/2]$, the  $m_{d-1}$-measure of the set of points $(x_1\dots x_{d-1})\in [0,1/2]^{d-1}$ satisfying~\eqref{ie_hyperbola_changed} is by definition equal to $m_{d-1}\big(V_{d-1}(\delta/x_d)\big)$. The upshot is that for any $d \ge 2$ and $ 0< \delta   <  2^{-d} $
\begin{equation} \label{Vk_value_one}
m_d\big(V_d(\delta) \big) \, = \, \delta \, + \, \int_{2^{d-1}\delta}^{1/2}  m_{d-1}\big(   V_{k-1}(\delta/x_d)  \big) \; \dd x_d.
\end{equation}
Now assume that \eqref{computation_Vk} holds with $d-1$ in place of $d$. Then, it follows via  \eqref{Vk_value_one} that

\begin{eqnarray*}
m_d\big(V_d(\delta) \big) &= &
\delta \ + \ \int_{2^{d-1}\delta}^{1/2}   \ \frac{\delta}{x_d}
\left(\sum_{t=0}^{d-2}\frac{1}{t!}
\left(\log\frac{x_d}{2^{d-1}\delta}\right)^{t}\right) \, \dd x_d
\\[2ex]
&= & \delta \ + \  \delta \, \int_{1}^{\frac{1}{2^{d}\delta} }   \
\left( \sum_{t=0}^{d-2} \frac{1}{t!}
 \frac{(\log y)^t}{y} \right)  \, \dd y   \\[2ex]
&= & \delta \ + \  \delta \ \sum_{t=0}^{d-2} \frac{1}{t!}   \ \int_{1}^{\frac{1}{2^{d}\delta} }   \
 \frac{(\log y)^t}{y} \ \dd y \\[2ex]
& = & \delta \ + \ \delta
\ \sum_{t=0}^{d-2}\frac{1}{(t+1)!}
\left(\log\frac{1}{2^{d}\delta}\right)^{t+1}  \\[2ex]
& = & \delta \ + \ \delta
\ \sum_{t=1}^{d-1}\frac{1}{t!}
\left(\log\frac{1}{2^{d}\delta}\right)^{t}  \, = \,  \delta
\ \sum_{t=0}^{d-1}\frac{1}{t!}
\left(\log\frac{1}{2^{d}\delta}\right)^{t}  \, .
\end{eqnarray*}
This completes the induction step and so establishes  \eqref{computation_Vk} for any $d  \ge 1 $.

\end{proof}

\medskip

\subsection{Proof of Theorem~\ref{metricresult}   \label{secmetricresult}}

We start by summarising  various basic facts concerning  $\beta$-transformations that will be  utilized in proving   Theorem~\ref{metricresult}. So, with this in mind,
let $\beta$ be a real number such that  $|\beta|>1$ and let  $T_{\beta}: [0,1)\to [0,1)$ be the associated $\beta$-transformation  given by $$T_{\beta}(x)=\beta x \ (\text{mod}\ 1).$$
 For obvious reasons, when $\beta<-1$ the corresponding transformation is refereed to as the negative $\beta$-transformation.
 %\sv{We remark that in the literature, the (negative) $\beta$-transformations are defined on the interval $[0,1)$, and here we identify $[0,1)$ as $\T$.}

For $ \beta > 1$, R\'{e}nyi \cite[Theorem 1]{R57} proved that there exists a unique  $T_\beta$-invariant measure $\mu_\beta$ (the so called Parry measure) that is strongly equivalent to  (one-dimensional) Lebesgue measure $m_1$ on the unit interval. Clearly, this implies that $\mu_\beta$ is  absolutely continuous with respect to Lebesgue measure. For the negative $\beta$-transformation, Ito and Sadahiro \cite{IS2009} proved that there is a unique $T_{\beta}$-invariant measure $\mu_{\beta}$ (the so called Yrrap measure) which is absolutely continuous with respect to Lebesgue measure $m_1$. The following proposition implies  that $\mu_{\beta}$ is in fact strongly equivalent to $m_1$ when $\beta \le   -g$.  Note that in view of  \cite[Theorem 16]{IS2009},
% Let $\beta>1$ be a real number. The $\beta$-transformation $T_{\beta}: [0,1]\to [0,1]$ is given by $$T_{\beta}(x)=\beta x - \lfloor \beta x\rfloor.$$ %$$T_{\beta}(x)=\beta x \ (\text{mod}\ 1).$$
% Similarly, the negative $\beta$-transformation
%  $T_{-\beta}: [0,1]\to [0,1]$ is defined as $$T_{-\beta}(x)=-\beta x+  \lfloor \beta x\rfloor+1.$$ %$$T_{-\beta}(x)=-\beta x \ (\text{mod}\ 1).$$
% R\'{e}nyi \cite{R57} proved that there exists a unique $T_\beta$-invariant measure $\mu_\beta$ (the so called Parry measure) strongly equivalent to the Lebesgue measure. For negative $\beta$-transformation, Ito and Sadahiro \cite{IS2009} proved that there is a unique $T_{-\beta}$-invariant measure $\mu_{-\beta}$ (called Yrrap measure) which is absolutely continuous with respect to the Lebesgue measure.
for the negative   $\beta$-transformation the corresponding density function of $\mu_{\beta}$  with respect to $m_1$ is given by
 $$h_{\beta}(x):=\frac{1}{F(\beta)}\sum_{n\ge 0\atop T_{\beta}^n1\ge x}\frac{1}{\beta^n},$$
 where
 $$F(\beta):=\int_0^1\sum_{n\ge 0\atop T_{\beta}^n1\ge x}\frac{1}{\beta^n}dx$$
 is the normalising function.
\begin{proposition}\label{negativebetameasure}
	Let $\beta \le -g$. Then the Yrrap measure  $\mu_{\beta}$ is strongly equivalent to the  Lebesgue measure $m_1$ on the unit interval. More precisely, there exists a constant $C(\beta)>0$ such that
	$$C(\beta)^{-1}\le h_{\beta}(x)\le C(\beta)    \qquad \forall 	\quad  x\in (0,1) \, . $$
\end{proposition}
\begin{proof}   For  $\beta=-g$,    it is easily verified  that
	$$h_{\beta}(x)=\begin{cases}
		\frac{1}{3+\beta}\ \ &\text{if}\quad  0<x\le 2+\beta\\[2ex]
		\frac{-\beta}{3+\beta}\ \ &\text{if}\quad   2+\beta<x< 1.
	\end{cases}
$$	
Hence,  we can choose $C(\beta)=\frac{1}{3+\beta}$.
Without loss of generality, assume $\beta<-g$ and note that
$$ 1+\sum_{n=0}^\infty\frac{1}{\beta^{2n+1}} \ = \ \frac{\beta^2+\beta-1}{\beta^2-1}   \  \le \   F(\beta)  \  \le  \   \sum_{n=0}^\infty \frac{1}{\beta^{2n}}  \ = \ \frac{\beta^2}{\beta^2-1} \, . $$
It then immediately  follows that
$$h_{\beta}(x)  \ \le \  \frac{1}{F(\beta)}\sum_{n=0}^\infty \frac{1}{\beta^{2n}} \  =\ \frac{\beta^2}{\beta^2+\beta-1}$$
and
$$h_{\beta}(x) \ \ge \  \frac{1}{F(\beta)}\left(1+\sum_{n=0}^\infty\frac{1}{\beta^{2n+1}}\right) \ = \ \frac{\beta^2   + \beta-1}{\beta^2}.$$
Hence, we can choose  $
C(\beta)= \frac{\beta^2}{\beta^2+\beta-1}  \,$  .
\end{proof}

The following result identifies  the nature of the support of the $T_\beta$-invariant measure $\mu_\beta$.
%In what follows, we let $\beta$ be a real number with $|\beta|>1$ and denote by $T_\beta$ the corresponding $\beta$- (or negative $\beta$-) transformation.  Let $\mu_\beta$ be the associated Parry (or Yrrap) measure and let $K(\beta)$ denote its support.

\begin{proposition}\label{Prop:support}
	Let $\beta$ be a real number with $|\beta|>1$, $\mu_\beta$ be the associated Parry-Yrrap measure and let $K(\beta)$ denote the support of $\mu_\beta$. Then
\begin{equation*}
K(\beta)=[0,1] \qquad {\rm if  }  \quad  \beta \in (-\infty, -g] \, \cup  \,  (1, +\infty) \, , \end{equation*}
and  $K(\beta)$ is a finite union of closed intervals contained in $[0,1]$ if $\beta \in (-g, -1)$.  Furthermore,  $\mu_\beta$ is mixing with respect to $T_{\beta}$  and is equivalent to  the measure $m|_{K(\beta)}$; i.e. the one-dimensional Lebesgue measure restricted to $K(\beta)$.
\end{proposition}

From this point onwards, given $\beta$ with $|\beta|>1$, $\mu_\beta$ will always denote the associated Parry-Yrrap measure and  $K(\beta)$ will denote the support of $\mu_\beta$.

%
%
%\begin{remark}
%From this point onwards, given $\beta$ with $|\beta|>1$, $\mu_\beta$ will always denote the associated Parry-Yrrap measure and  $K(\beta)$ will denote the support of $\mu_\beta$.
%\end{remark}

\begin{proof}
If $\beta\in (-\infty, -g] \, \cup  \,  (1, +\infty)$, the result immediately follows from the fact that  $\mu_{\beta}$ is strongly equivalent to the  Lebesgue measure $m_1$ on $[0,1]$  --  this is
Proposition \ref{negativebetameasure} for $\beta \le -g$ and as already mentioned established by   R\'{e}nyi \cite[Theorem 1]{R57} for $\beta > 1$.   In general,  Keller \cite{K1978} proved for any $|\beta|>1$ the support of $\mu_\beta$ is a finite union of closed intervals.   The precise description of the closed intervals for the non-trivial case when   $\beta \in (-g, -1)$ was given by Liao $\&$ Steiner \cite[Theorem 2.1]{LS12}.

 For the furthermore part, it follows via  Rokhlin \cite[Section 4.5]{R1961} for $\beta>1$ and  Liao $\&$ Steiner \cite[Corollary 2.3]{LS12} for  $\beta<-1$, that the  $T_\beta$-invariant measure $\mu_\beta$ is exact and hence mixing  with respect to $T_\beta$ (\cite[Proposition 12.2]{PY1998}).  In turn, by the
 %By Liao and Steiner \cite[Corollary 2.4.]{LS12} (see also Faller \cite[Theorem 6.23]{Fa08})
the Main Result in  \cite{W1979} it follows that  $\mu_\beta$ is equivalent to $m_1$ restricted to  $K(\beta)$.
\end{proof}

Of course, as indicated in above proof of the proposition, for $\beta \in (-\infty, -g] \, \cup  \,  (1, +\infty)$ we have that $\mu_{\beta}$ is strongly equivalent to the  Lebesgue measure $m_1$ on $[0,1]$ rather than simply equivalent.
The next statement is a  straight forward consequence of Proposition~\ref{Prop:support} and basic  properties of product measures.
\begin{lemma}\label{productmeasure}
Let $T$ be a real, non-singular matrix transformation of the torus $\mathbb{T}^d$. Suppose that $T$ is diagonal and all  eigenvalues $\beta_1,\beta_2,\dots,\beta_d$ are of modulus strictly larger than $1$.  Then the product measure $$\nu:=\mu_{\beta_1}\times\mu_{\beta_2}\times\cdots\times\mu_{\beta_d}$$ has support $
	K:=\prod_{i=1}^d K(\beta_i)
$ and is a $T$-invariant mixing measure that is equivalent to  $m_d|_K$; i.e. the $d$-dimensional  Lebesgue measure restricted to $K$.
\end{lemma}

\begin{proof} By definition, the product measure $\nu$ has support $K=\prod_{i=1}^d K(\beta_i)$ and in view of  Proposition~\ref{Prop:support} it is $T$-invariant and
equivalent to the measure $m_d|_K$.  Furthermore, since each $\mu_{\beta_i}$ is mixing with respect to $T_{\beta_i}$,  on following the proof of \cite[Theorem 1.24]{W-GTM79} it is easily verified that $\nu$ is mixing with respect to $T$.
\end{proof}

 We now show that Theorem~\ref{metricresult} is an easy consequence of  Lemma~\ref{productmeasure} together with  Proposition~\ref{mainohyes} and Theorem~\ref{gencountthm}. In the next section we will provide a  self-contained and essentially elementary proof of Theorem~\ref{metricresult}  in the case the collection  $\mathcal{C}$  of subsets of $K$ is restricted to rectangles  with sides parallel to the axes. This is an important class of ``target sets'' that clearly satisfy the bounded  property
$ ({\boldsymbol{\rm B}}) $ and the proof will avoid appealing to  Proposition \ref{mainohyes}.

%For general collection $\mathcal{C}$ of subsets $E$ of $K$ satisfying the bounded property $ ({\boldsymbol{\rm B}}) $, we apply  Propositions \ref{gencountthm} and \ref{mainohyes} and Theorem \ref{measure-matrix}.   %  to give an alternative proof of Theorem \ref{metricresult}.

\begin{proof}[Proof of Theorem \ref{metricresult} using Proposition \ref{mainohyes}]
Lemma \ref{productmeasure} implies  that  product measure $\nu$ is a $T$-invariant mixing measure equivalent to $m_d|_K$. Hence by Proposition~\ref{mainohyes}, $\nu$ is exponentially mixing with respect to $(T, \mathcal{C})$ where $\mathcal{C}$ is any collection  of subsets $E$ of $K$ satisfying the bounded property $ ({\boldsymbol{\rm B}}) $. Then the main counting part of Theorem \ref{metricresult} immediately follows from Theorem~\ref{gencountthm}.  For the ``furthermore'' part we first recall (as we have done so several times)  that  by R\'{e}nyi \cite[Theorem 2]{R57} and Proposition~\ref{negativebetameasure}, for any $\beta\in (-\infty, -g]\cup (1, +\infty)$  the Parry-Yrrap measure $\mu_\beta$ is strongly equivalent to the  Lebesgue measure $m_1$ on  $[0,1]$.
 It thus  follows that the product measure $\nu$ is strongly equivalent to $m_d$  restricted  to $K=\T^d$. The upshot of this is that we  can replace $\nu$ by $m_d$ in the first part of the theorem and thereby completes the proof of Theorem~\ref{metricresult}.
%On the other hand, remark that the balls $R_n= R(\va, \psi(n))$ with respect to the maximum norm are cubes: $B(a_1,\psi(n))\times \cdots \times B(a_d, \psi(n))$. Thus by the strong equivalence  between $\mu$ and $m_d$, we have
%\[
%\sum_{n=1}^\infty \mu(R_n) \asymp \sum_{n=1}^\infty m_d(R_n)=2\sum_{n=1}^\infty \psi(n).
%\]
%Therefore, Theorem \ref{metricresult} (ii) follows immediately from Theorem \ref{metricresult} (i).
%The seconde part of the theorem follows from the fact that when all $\beta_i$ are in $(-\infty, -g]\cup (1, +\infty)$, $\mu$ is strongly equivalent to $m_d$. Further, $\mu$ has support $\T^d$ and is mixing. Then we can finish the proof by applying Theorem \ref{measure-matrix}.
% thus an acim with the whole space as its support. On the other hand, since for each $1\le   i \le   d$, $\mu_i$ is mixing with respect to $T_{\beta_i}$, we have $\mu$ is mixing with respect to $T$. Then we can apply Theorem \ref{measure-matrix} to conclude the result for the measure $\mu$.
\end{proof}

\subsubsection{Theorem~\ref{metricresult} for rectangles: a self contained and direct  proof \label{specialthm4} }

In the proof of Theorem~\ref{metricresult} given above, we make use  of  Proposition~\ref{mainohyes} to deduce that the product  measure $\nu$ is exponentially mixing with respect to $(T, \mathcal{C})$ where $\mathcal{C}$ is any collection  of subsets $E$ of $K$ satisfying the bounded property $ ({\boldsymbol{\rm B}}) $.   The following result enables us to bypass the proposition in the case $\mathcal{C}$   is restricted to rectangles  with sides parallel to the axes.

\begin{lemma}\label{productmeasure-exp-m}
Let $T$ be a real, non-singular matrix transformation of the torus $\mathbb{T}^d$. Suppose that $T$ is diagonal and all  eigenvalues $\beta_1,\beta_2,\dots,\beta_d$ are of modulus strictly larger than $1$.  Let $\nu:=\mu_{\beta_1}\times\mu_{\beta_2}\times\cdots\times\mu_{\beta_d}$ be the product measure and   $
	K:=\prod_{i=1}^d K(\beta_i)$ be its support.     Then $\nu$ is exponentially mixing with respect to $(T, \mathcal{R})$ for any collection  $\mathcal{R}$ of rectangles of $K$ with sides parallel to the axes.
\end{lemma}

It is easily  seen that by appealing to Lemma~\ref{productmeasure-exp-m} instead of Proposition~\ref{mainohyes} in the ``Proof of Theorem \ref{metricresult} using  Proposition \ref{mainohyes}'' given in the previous section,   we obtain the special case of  Theorem \ref{metricresult} in which $\mathcal{C}$ is any collection  $\mathcal{R}$ of rectangles of $K$ with sides parallel to the axes.  In other words, it enables us to provide a self-contained and direct proof of Theorem~\ref{metricresult} for rectangular target sets.

\begin{proof}[Proof of Lemma~\ref{productmeasure-exp-m}]
First, we assert that for any $\beta \in \mathbb{R}$ with $|\beta|>1$,  $\mu_\beta$ is exponentially mixing with respect to $(T_\beta, \mathcal{C})$ where $\mathcal{C}$ is any collection of intervals of $K(\beta)$.  When $\beta>1$, this is the classical result of Gel'fond \cite[Formula (12)]{G1959} and Philipp \cite[Lemma 7]{P1967}. When $\beta<-1$, the assertion  follows from a general result of Baladi \cite[Theorem 3.4]{Baladi} for piecewise monotone expanding interval maps.

%To be precise, this is proved only for positive
 %$\beta$ but  the argument in all likelihood can be adapted for negative  $\beta$.
  % In any case  Proposition \ref{mainohyes} in \S\ref{secMTT} implies $\mu_\beta$ is exponentially mixing.

%	Denote by $\mu_{\beta_i}$ such measures for $\beta_i$ ($i=1,\cdots,d$) in the lemma. Let $0<\gamma_i<1$ be the corresponding exponentially mixing constant.
%%	({\color{red} do you still have exponentially mixing property for negative beta case for $\beta\in (-\infty, -g]$}? YES!)
%	Then the product measure $\mu:=\mu_{\beta_1}\times\mu_{\beta_2}\times\cdots\times\mu_{\beta_d}$ is $T$-invariant.
	
We  now verify that the product  $\nu$ satisfies the desired exponentially mixing property. With this in mind,  let  $E=B(z_1,r_1)\times \cdots \times B(z_d, r_d)$ and $F=B(z'_1,r'_1)\times \cdots \times B(z'_d, r'_d)$  be any two rectangles in $\mathcal{R}$.
	Then
	\begin{eqnarray*}
		\nu(E\cap T^{-n}F)
=\int\chi_{E}(x_1,\dots,x_d) \, \chi_{F}(T_{\beta_1}^nx_1,\dots,T_{\beta_d}^nx_d) \ d\mu_{\beta_1}(x_1)\cdots \, d\mu_{\beta_d}(x_d)  \, ,
			\end{eqnarray*}
		and by the property of the product measure the right hand side   equals %noting that we choose the norm of maximal value,
\begin{eqnarray*}		
\int\chi_{B(z_1,r_1)}(x_1)\chi_{B(z_1',r_1')}(T_{\beta_1}^nx_1)d\mu_{\beta_1}(x_1)\cdots\int\chi_{B(z_d,r_d)}(x_d)\chi_{B(z_d',r_d')}(T_{\beta_d}^nx_d)d\mu_{\beta_d}(x_d).
\end{eqnarray*}
It then follows by the exponentially mixing property of $T_{\beta_i}$,  that
\begin{eqnarray*}		
		\nu(E\cap T^{-n}F)=\prod_{i=1}^d  \Big( \, \mu_{\beta_i}(B(z_i,r_i)) \, \mu_{\beta_i}(B(z_i',r_i'))+O(\gamma_i^{n}) \, \mu_{\beta_i}(B(z_i',r_i')) \, \Big)
\end{eqnarray*}	
where $0<\gamma_i<1$.    This together with the fact that $$\nu(E)=\mu_{\beta_1}(B(z_1,r))\cdots\mu_{\beta_d}(B(z_d,r)) \, ,  \quad \nu(F)=\mu_{\beta_1}(B(z_1',r'))\cdots\mu_{\beta_d}(B(z_d',r'))$$ and $\mu_{\beta_i}(B(z_i,r))\le 1$  $(1\le i\le d)$ implies that
	\begin{eqnarray*}	
		\nu(E\cap T^{-n}F)=\nu(E)\nu(F)+O(\gamma^n)\mu(F)    \quad {\rm where}  \quad \gamma=\max\{\gamma_1,\dots,\gamma_d\}  \, .
	\end{eqnarray*}
In other words, \eqref{mixing} holds for rectangles in $\mathcal{R}$ and we are done.
 % and then also is true for any measurable set $F$ since each term in \eqref{mixing} includes $F$ and $F$ can be generated by balls.
\end{proof}

 %We first give a direct proof for the case where $\mathcal{C}$ is a collection of rectangles in $K$ by using Theorem~\ref{gencountthm}.

%
%\medskip
%
%\begin{remark} \hi{I DON'T THINK WE NEED THIS REMARK}
%We have the same result of Theorem \ref{metricresult} (ii) for the shrinking target problem of Euclidean balls:
%$$
%B_n := B(\va, \psi(n))= \Big\{ \vx  \in \mathbb{T}^d :   \Big({\sum_{i=1}^d |a_i-x_i|^2}\Big)^{1/2} \le \psi(n)   \Big\}.
%$$
%To prove this, one needs only notice that
%$$
%R(\va, \psi(n)/\sqrt{d}) \subset B(\va, \psi(n))\subset R(\va, \psi(n)).
%$$
%\end{remark}

%\begin{remark}  \label{LS}
%A consequence of Proposition \ref{negativebetameasure} is that the support of the Yrrap measure  $\mu_\beta$ is the full  interval $[0,1)$ when $\beta \in(-\infty, -g]$. However, by Rokhlin \cite{R1961}, and  Liao and Steiner \cite[Theorem 2.1]{LS12}, when $\beta \in(-g, -1)$, the support of $\mu_\beta$  is a finite union of disjoint closed intervals which is not full.  Thus it is natural to preclude this interval when it comes to establishing  measure theoretic statements with respect to Lebesgue measure; such as, Theorem~\ref{metricresult}.
%\end{remark}

\subsubsection{Extending Corollary \ref{Cor1} to incorporate eigenvalues in $[-1,1]$ } \label{metricresultext}
 We show that on assuming $\psi(n)\to 0$ as $n\to\infty$, we can naturally extend Corollary \ref{Cor1} to the situation that all the  eigenvalues of $T$  are in $(-\infty, -g] \, \cup  \,   [-1, +\infty)$; that is to say,  we can incorporate the interval  $[-1,1]$.  %In the same way, Theorem~\ref{metricresult} (i) can be evidently extended to all diagonal matrices, we omit the details.
 %Note that in view of Remark~\ref{LS}, this is best possible.
In short, assuming that $T$  has eigenvalues in $[-1,1]$, we do this in most cases  by reformulating  the shrinking target set  $W(T,\psi,\va)$ in terms of related ``lower dimensional'   sets $W(T_*,\psi,\va_*)$  for which   Corollary~\ref{Cor1} in its current form  is applicable.  In other words, all the eigenvalues of the related transformation $T_*$  are  in $(-\infty, -g] \, \cup  \,   (1, +\infty)$.
%Let $T$ be a real, non-singular, diagonal  matrix transformation of the torus $\mathbb{T}^d$ with  eigenvalues $\beta_1,\beta_2,\dots,\beta_d$  in $(-\infty, -g]\cup [-1, +\infty)$.  Without loss of generality, assume that there is at least one eigenvalue, say $\beta_1$,  in  $[-1,1]$.
Let $T$ be a real, non-singular, diagonal  matrix transformation of the torus $\mathbb{T}^d$ with  eigenvalues $\beta_1,\beta_2,\dots,\beta_d$  in $(-\infty, -g]\cup [-1, +\infty)$.  Without loss of generality, assume that there is at least one eigenvalue in  $[-1,1]$.  In fact, let us assume that there is only one such eigenvalue, say $\beta_1$.  It should be self-evident how to deal with the situation in which that are  multiple eigenvalues in  $[-1,1]$.
   We consider the  three separate  situations  depending on whether $ |\beta_1| < 1$, $\beta_1 =1 $ or  $\beta_1 =-1 $.  Note that $T=\text{diag}(\beta_1,\cdots,\beta_d)$ and so for any $\vx=(x_1,\dots,x_d)\in \T^d$
$$T^n(\vx)=(T_{\beta_1}^n(x_1),T_{\beta_2}^n(x_2),\cdots,T_{\beta_d}^n(x_d)).$$

%\noindent \textbf{(i)} One of eigenvalues is strictly less than $1$ in modulus. Without loss of generality, we assume $|\beta_1|<1$.
\noindent \textbf{(i)} \emph{We assume $|\beta_1|<1$.}  We distinguish between three subcases:
\begin{itemize}
  \item Case 1: $\beta_1=0$. Then it is easily verified that
  \begin{eqnarray*}
	W(T,\psi,\va)=
	\begin{cases}
		\emptyset &\text{if}\ a_1\neq 0\\[1ex]
		\mathbb{T}\times W(T_*,\psi, \va_*) &\text{if}\ a_1=0,
	\end{cases}
\end{eqnarray*}
where
\begin{equation} \label{upp} T_*:=\text{diag}(\beta_{2},\cdots,\beta_d)  \quad {\rm   \ and   \   } \quad  \va_*:=(a_{2},\cdots,a_d)\in\mathbb{T}^{d-1}    \, .  \end{equation}
\item Case 2: $0<\beta_1<1$. Then $T_{\beta_1}^n(x_1)=\beta_1^nx_1\to 0$ as $n\to\infty$ for any $x_1\in\mathbb{T}$. Note that zero is the unique fixed point of $T_{\beta_1}$. Thus,
\[
\big\{x_1\in [0,1) : T_{\beta_1}^n(x_1)< \psi(n)\  \hbox{ for infinitely many }n\in \mathbb{N} \big\}  \, =  \, I_* \ \text{ or } \ \overline{I_*},
\]
where $I_*:=[0, \min\{1,\tau\})$, $\overline{I_*}$ is the closure of the set $I_*$ and \begin{equation} \label{up}
\tau:=
\limsup\limits_{n\to\infty}  \, \psi(n)  \, |\beta_1|^{-n}
.\end{equation}
Indeed, the set under consideration  is $I_*$ if $\psi(n) \, |\beta_1|^{-n}\le    \tau   $ for infinitely many $  n\in \mathbb{N}$ and $\overline{I_*}$ otherwise.  Hence, with $ T_*$  and $ \va_*$ as in \eqref{upp}, it follows that

\begin{eqnarray*}
	W(T,\psi,\va)=
	\begin{cases}
		\emptyset &\text{if}\ \ a_1\neq 0\\[1ex]
		\{0\}\times W(T_*,\psi, \va_*) &\text{if}\ \  a_1=0\ \text{and}\ \tau=0,\\[1ex]
I_*\times W(T_*,\psi, \va_*)  \text{ \  \ or \  \ } \overline{I_*}\times W(T_*,\psi, \va_*) &\text{if}\ \  a_1=0\ \text{and}\ \tau>0.
	\end{cases}
\end{eqnarray*}
\item Case 3: $-1<\beta_1<0$. Let $F$ be the countable set of all preimages of zero; that is,
    $$
    F:=\big\{x_1\in\mathbb{T}: T_{\beta_1}^n(x_1)=0\ \text{for some}\ n\ge   0\big\}.
    $$
      If $x_1\notin F$, then $T_{\beta_1}x_1=\beta_1x_1+1$ and $$T_{\beta_1}^nx_1=\beta_1^nx_1+\beta_1^{n-1}+\beta_1^{n-2}+\cdots+\beta_1+1\to \frac{1}{1-\beta_1}$$
    as $n\to\infty$.  Note that zero and $\frac{1}{1-\beta_1}$ are the fixed points of $T_{\beta_1}$. Thus,
  \[
 \big\{x_1\in [0,1)\setminus F : T_{\beta_1}^n(x_1)< \psi(n)\ \ \text{for infinitely many}\ n\in\mathbb{N}\big\} = J_*\setminus F  \ \text{ or } \  \overline{J_*}\setminus F ,
\]
where $J_*:=\left( (\frac{1}{1-\beta_1}-\tau, \frac{1}{1-\beta_1}+\tau)\cap [0, 1)\right)$, $\overline{J_*}$ is the closure of the set $J_*$ and $\tau$ is given by \eqref{up}.  Indeed, the set under consideration  is $J_*\setminus F$   if
$\psi(n) \, |\beta_1|^{-n} \le    \tau $ for infinitely many $ n\in \mathbb{N}$ and $\overline{J_*}\setminus F$ otherwise.
Hence, with $ T_*$  and $ \va_*$ as in \eqref{upp}, it follows that
\begin{eqnarray*}
	W(T,\psi,\va)=
	\begin{cases}
F\times W(T_*,\psi, \va_*) &\text{if}\ \ a_1=0\\[1ex]
		\emptyset &\text{if}\  \ a_1\neq \frac{1}{1-\beta_1}, a_1\neq 0\\[1ex]
		\{\frac{1}{1-\beta_1}\}\times W(T_*,\psi, \va_*) &\text{if} \ \ a_1=\frac{1}{1-\beta_1}\ \text{and}\ \tau=0,\\[1ex]
J_*\times W(T_*,\psi, \va_*)  \text{ \ \  or  \  \  } \overline{J_*} \times W(T_*,\psi, \va_*) &\text{if} \  \ a_1=\frac{1}{1-\beta_1}\ \text{and}\ \tau>0.
	\end{cases}
\end{eqnarray*}
\end{itemize}
The upshot is that in order to determine the size of $W(T,\psi, \va)$ or the behaviour of the associated counting function we need to investigate the shrinking target set $W(T_*,\psi, \va_*) \subseteq \T^{d-1}$.   Recall, that for ease of discussion we are assuming that $\beta_1$ is the only eigenvalue of $T$ in  $[-1,1]$.  Thus, all the eigenvalues of the related transformation $T_*$  are  in $(-\infty, -g] \, \cup  \,   (1, +\infty)$ and so Corollary~\ref{Cor1} is applicable with $T$ replaced by $ T_*$, $\va$ replaced by $\va_* $ and $d$ replaced by $d-1$.

\bigskip

\noindent \textbf{(ii)} \emph{We assume $\beta_1=1$.}
For any $\vx=(x_1,\dots,x_d)\in \mathbb{T}^d$ and $ \va=(a_1,\cdots,a_d)\in \mathbb{T}^d$, the condition $T^n(\vx)\in B(\va,\psi(n))$
%$$|T^n\vx-\va|=\max\Big\{|x_1-y_1|,|T_{\beta_2}^n(x_2)-y_2|,\cdots,|T_{\beta_d}^n(x_d)-y_d|\Big\}<\psi(n)$$
implies that $$\|x_1-a_1\|<\psi(n).$$  Now,  since we are assuming that $\psi(n)\to 0$ and $n \to \infty$, it follows that for any $\vx\in W(T,\psi, \va)$ we must  have that  $x_1=a_1$. Hence, with $ T_*$  and $ \va_*$ as in \eqref{upp}, it follows that
$$W(T,\psi,\va)=\{a_1\}\times W(T_*,\psi, \va_*)  \, . $$
As  in (i), the upshot is that we need to investigate the shrinking target set $W(T_*,\psi, \va_*) \subseteq \T^{d-1}$ and that  for this setup  Corollary~\ref{Cor1} is applicable with $T$ replaced by $ T_*$, $\va$ replaced by $\va_* $ and $d$ replaced by $d-1$.

%\noindent \textbf{(ii)} \emph{We assume $\beta_1=1$.}
%Then, for any $\va=(a_1,\cdots,a_d)\in \mathbb{T}^d$, the condition $T^n(\vx)\in B(\va,\psi(n))$
%%$$|T^n\vx-\va|=\max\Big\{|x_1-y_1|,|T_{\beta_2}^n(x_2)-y_2|,\cdots,|T_{\beta_d}^n(x_d)-y_d|\Big\}<\psi(n)$$
%implies $$|x_1-a_1|<\psi(n).$$ Since $\psi(n)\to 0$, for any $x\in W(T,\psi, \va)$, we have $x_1=a_1$. Let  $k$ be the multiple number of eigenvalue $1$. Without loss of generality, assume $\beta_1=\beta_2=\cdots=\beta_k=1$, and $\beta_{k+1},\cdots,\beta_d\in (-\infty, -g]\cup (1, +\infty)$. Then for any $\va=(a_1,\cdots,a_d)\in\mathbb{T}^d$,
%$$W(T,\psi,\va)=\{a_1\}\times\cdots\times \{a_k\}\times W(T_*,\psi, \va_*) $$
%where $T_*=\text{diag}(\beta_{k+1},\cdots,\beta_d)$ and $\va_*=(a_{k+1},\cdots,a_d)\in\mathbb{T}^{d-k}$. By Corollary \ref{Cor1}, we have
%\begin{eqnarray*}
%	m_{d-k}(W(T_*,\psi,\va_*))=
%	\begin{cases}
%		0 &\text{if}\ \sum_{n=1}^\infty \psi(n)^{d-k}<\infty\\
%		1 &\text{if}\ \sum_{n=1}^\infty \psi(n)^{d-k}=\infty.
%	\end{cases}
%\end{eqnarray*}
%Thus, we have reduced the study of $W(T,\psi,\va)$ to the study of a shrinking target set $W(T_*,\psi,\va_*)$ in a lower dimensional space whose Lebesgue measure is determined by Corollary~\ref{Cor1}.

\bigskip

\noindent \textbf{(iii)} \emph{We assume $\beta_1=-1$.}
 Then, for $x_1\neq 0$
\begin{eqnarray*}
T_{\beta_1}^n(x_1)=
	\begin{cases}
		x_1 &\text{if}\ \  n \ \text{is even}\\[1ex]
		-x_1+1 &\text{if}\   \  n\ \text{is odd}.
	\end{cases}
\end{eqnarray*}
If $a_1=0$, then the same reasoning as in (ii) shows that $ W(T,\psi, \va)=\{0\}\times W(T_*,\psi, \va_*)$  and so we can apply  Corollary~\ref{Cor1} with $T$ replaced by $ T_*$, $\va$ replaced by $\va_* $ and $d$ replaced by $d-1$. If $a_1\neq 0$, it follows that for any  $\vx\in W(T,\psi, \va)$  we must have that  $x_1=a_1$ or $x_1=-a_1+1$. Thus,  with $ T_*$  and $ \va_*$ as in \eqref{upp},  we have that
 $$W(T,\psi, \va)=\{a_1\}\times W'(T_*,\psi, \va_*) \ \bigcup \  \{1-a_1\}\times W''(T_*,\psi,\va_*),$$
where
 $$ W'(T_*,\psi, \va_*):=\{\vx\in\mathbb{T}^{d-1}: T_*^n(\vx)\in B(\va_*,\psi(n))\ \ \text{for infinitely many even}\ n\in\mathbb{N}\},$$
 $$ W''(T_*,\psi, \va_*):=\{\vx\in\mathbb{T}^{d-1}: T_*^n(\vx)\in B(\va_*,\psi(n))\ \ \text{for infinitely many odd}\ n\in\mathbb{N}\}.$$

 Now observe that $W'(T_*,\psi, \va_*)$ and $W''(T_*,\psi, \va_*)$ are shrinking target sets with respect to the   transformation $T_*\circ T_*$ of the torus $\mathbb{T}^{d-1}$.  Indeed,
$$
 W'(T_*,\psi, \va_*)=\{\vx\in\mathbb{T}^{d-1}: (T_*\circ T_*)^{n}(x)\in B(\va_*,\psi(2n))\ \ \text{for infinitely many}\ n\in\mathbb{N}\},
$$
and
$$
 W''(T_*,\psi, \va_*)=\{\vx\in\mathbb{T}^{d-1}: (T_*\circ T_*)^{n}(x)\in T_*^{-1}B(\va_*,\psi(2n+1))\ \ \text{for infinitely many}\ n\in\mathbb{N}\}.
$$
Now in essence,   the above procedure removes the presence of the problematic  eigenvalue  $\beta_1=-1$ but still none of the results we have established for matrix transformations are applicable to the above shrinking targets sets. The reason for this is simple.  The composition map  $T_*\circ T_*$  is not a matrix transformation of the torus $\mathbb{T}^{d-1}$.  The upshot is that we need to appeal to Theorem~\ref{gencountthm} and its corollary directly.  In view of the argument set out in \S\ref{specialthm4} that provides a self-contained proof of Theorem~\ref{metricresult} for rectangular target sets, it is easily seen that the desired counting and  measure statements for $W'(T_*,\psi, \va_*)$ and   $W''(T_*,\psi, \va_*)$  would follow on showing that  the product measure $\nu$ is exponentially mixing with respect to $(T_*\circ T_*, \mathcal{R})$.    To establish the latter,   we first recall  (see the start of the proof of Lemma~\ref{productmeasure-exp-m})  that for any $\beta \in \mathbb{R}$ with $|\beta|>1$,  $\mu_\beta$ is exponentially mixing with respect to $(T_\beta, \mathcal{C})$ where $\mathcal{C}$ is any collection of intervals of $K(\beta)$.  Hence by definition,  there exists a constant $0<\gamma<1$ such that
$$
\mu_\beta\big(E \cap T^{-2n}F\big)= \mu_\beta(E) \mu_\beta(F) +O(\gamma^{2n})  \, \mu_\beta(F)
$$
for any $E, F\in \mathcal{C}$. In other words,
$$
\mu_\beta\big(E \cap (T\circ T)^{-n}F\big) = \mu_\beta(E) \mu_\beta(F) +O\big((\gamma^2)^n\big) \,  \mu_\beta(F).
$$
and so  $\mu_\beta$ is exponentially mixing with respect to $(T_\beta \circ T_\beta, \mathcal{C})$.  Then, on mimicking the proof of Lemma~\ref{productmeasure-exp-m}  with $T$ replaced by $T_*\circ T_*$ and $d$ replaced by $d-1$, we conclude that $\nu$ is exponentially mixing with respect to $(T_*\circ T_*, \mathcal{R})$  for any collection  $\mathcal{R}$ of rectangles of $K$ with sides parallel to the axes.

\medskip

\subsubsection{A nifty ``reduction'' argument} \label{redarg}

 In this section we show that when $T$ is an integer matrix transformation, the diagonal assumption in Theorem~\ref{metricresult}  can be relaxed to $T$ is diagonalizable over $\Z$.  So,
 suppose $T$ is diagonalizable over $\mathbb{Z}$.  Then by definition, there exist a nonsingular integer  matrix $P$ and a diagonal integer matrix $D$   such that product matrix relationship $P\cdot T=D \cdot P$ holds.  In turn, there exists an invertible mapping
 \begin{equation} \label{nifty} \phi: \mathbb{T}^d\to\mathbb{T}^d  \quad \hbox{\ such  that \ } \quad \phi\circ T=D\circ \phi  \, . \end{equation}  Obviously, the diagonal entries of $D$ are the eigenvalues of $T$ and we assume  these integers are of modulus strictly larger than $1$.

 Now recall that for diagonal transformations such as $D$, the ``Proof of Theorem~\ref{metricresult} using Proposition~\ref{mainohyes}'' makes key use of the fact that the product measure $\nu$  is $D$-invariant and is exponentially mixing with respect to $(D, \mathcal{C})$ where $\mathcal{C}$ is any collection  of subsets $E$ of $K$ satisfying the bounded property $ ({\boldsymbol{\rm B}}) $.   We claim that the image measure $\nu\circ\phi$ is $T$-invariant and exponentially mixing with respect to $(D, \mathcal{C})$.  The proof of Theorem~\ref{metricresult} can then be modified in the obvious manner to deal with the more general (integer) situation  in which $T$ is diagonalizable over $\mathbb{Z}$.   To establish the claim, first note that  for any measurable set $A\subset \mathbb{T}^d$, on using the fact that $\nu$  is $D$-invariant it follows that
 $$
 \nu\circ\phi(T^{-1}A)=\nu(\phi(T^{-1}A))=\nu(D^{-1}(\phi(A)))
 =\nu(\phi(A))=\nu\circ\phi(A).
 $$
Thus,  $\nu\circ\phi$ is $T$-invariant.
	Next, since $\phi$ is linear the (inverse) image of any collection $\mathcal{C}$ of subsets $E$ of $K$ satisfying the bounded property $ ({\boldsymbol{\rm B}}) $ also satisfies the bounded property $ ({\boldsymbol{\rm B}})$.
	Hence, for any $E, F\in \mathcal{C}$, noting that $\phi(E\cap T^{-n}F)=\phi(E)\cap \phi(T^{-n}F)$, it follows that there exists a constant $0<\gamma<1$ such that %and measurable set $F\subset \mathbb{T}^d$,
	\begin{eqnarray*}
	\nu\circ\phi(E\cap T^{-n}F)& = & \nu\big(\phi(E\cap T^{-n}F)\big)=\nu\big(\phi(E)\cap\phi(T^{-n}F) \big)\\[1ex]
		&=&\nu\big(\phi(E)\cap D^{-n}\circ\phi(F)\big)\\[1ex]
	&=&\nu\big(\phi(E)\big) \, \nu\big(\phi(F)\big)\, + \, O\big(\gamma^n\big) \ \nu\big(\phi(F)\big)\\[1ex]
		&=&\nu\circ\phi(E)  \ \nu\circ\phi(F)  \, +  \, O(\gamma^n) \ \nu\circ\phi(F)  \,
	\end{eqnarray*}
as desired.    Note that the third displayed line uses the fact that $\nu$  is exponentially mixing with respect to $(D, \mathcal{C})$.

\medskip

\begin{remark}	
We remark that for real non-integer matrices,  we cannot in general use the above argument to extend Theorem~\ref{metricresult} to the situation that $T$ is diagonalizable over $\Z$. In short,  the commutative property  $\phi\circ T=D\circ \phi$ may not be true since the product matrix relationship $P\cdot T=D\cdot P$ does not guarantee that %$P(T(x))=D(P(x))$ for which the twice module one for $T$ and $P$ can not exchange.
\[
P(T\vx \ {\rm mod} \ 1) \ {\rm mod} \ 1 = T(P\vx \ {\rm mod} \ 1) \ {\rm mod} \ 1.
\]
\end{remark}

\section{Establishing  dimension results for matrix transformations \label{estdim}}

We begin with a brief account  in which we  bring together various statements concerning Hausdorff measure and dimension that we will utilise in the course of establishing Theorems~\ref{dimresult} and  \ref{dimensionalone}.

\subsection{Preliminaries}
We start by defining Hausdorff measure and dimension
for completeness and for establishing some notation. Let $X$ be a subset of $\R^d$.   For $\rho
> 0$, a countable collection $ \left\{B_{i} \right\} $ of
Euclidean balls in  $\R^{d}$ of diameter $d_i \le   \rho $ for each
$i$ such that $X \subset \bigcup_{i} B_{i} $ is called a $ \rho
$-cover for $X$.  Let $s$ be a non-negative number and define $$
 {\cal H}^{s}_{ \rho } (X)
  \; = \; \inf \left\{ \sum_{i} d_i^s
\ :   \{ B_{i} \}  {\rm \  is\ a\  } \rho {\rm -cover\  of\ } X
\right\} \; , $$ where the infimum is taken over all possible $
\rho $-covers of $X$. The {\it s-dimensional Hausdorff measure}
${\cal H}^{s} (X)$ of $X$ is defined by $$ {\cal H}^{s} (X) =
\lim_{ \rho \rightarrow 0} {\cal H}^{s}_{ \rho } (X) = \sup_{ \rho
> 0} {\cal H}^{s}_{ \rho } (X)
$$ \noindent and the {\it Hausdorff dimension} dim $X$ of $X$ by
$$ \dim_{\rm H} \, X = \inf \left\{ s : {\cal H}^{s} (X) =0 \right\} =
\sup \left\{ s : {\cal H}^{s} (X) = \infty \right\} \, . $$

\vskip 9pt

%Strictly speaking, in the standard definition of Hausdorff measure
%the  $\rho$--cover by cubes is replaced by non--empty subsets in
%$\R^k$ with diameter at most $\rho \,$. It is easy to check that
%the resulting measure is comparable to ${\cal H}^{s} $ defined
%above and thus the Hausdorff dimension is the same in both cases.

\noindent Further details and alternative definitions of Hausdorff measure  and dimension can be found in \cite{F,MAT}.   It is easily verified (see \cite[Corollary~2.4]{F}) that the Hausdorff dimension of a set is invariant under bi-Lipschitz maps.

\begin{lemma} \label{lip}
Let $X$ be a subset of $\R^d$ and $f: X \to \R^d$  be a bi-Lipschitz map; i.e.
$$
c_1 \, |x-y|  \le  | f(x) - f(y) |  \le c_2 \, |x-y|    \qquad (x,y \in X )
$$
where $ 0 < c_1 \le c_2  < \infty $, then   $\dim_{\rm H}  f(X) = \dim_{\rm H}  X \, $.
\end{lemma}

We now  describe a deep and powerful mechanism for obtaining  lower bounds for the Hausdorff dimension of a large class of ``rectangular' $\limsup$ sets.

\subsubsection{Mass Transference Principle for Rectangles}  \label{RR-baby}

The discussion below is tailored to the application we have in mind. It is far from the most
general and powerful setup of the Mass Transference Principle. We begin by describing  the original  `balls to balls' principle which is all that is required for directly proving Theorem~\ref{dimresult} and  Theorem~\ref{dimensionalone}.  However, we will deduce Theorem~\ref{dimresult} from a more general statement concerning  rectangular target sets  and for this we will require the more versatile `rectangle to rectangle' principle.

To set the scene, let $X$ be a locally compact subset of $\R^d$  equipped with a non-atomic probability measure $\mu$. Suppose there
exist constants $ \delta > 0$, $0<a\le 1\le b<\infty$ and $r_0 > 0$ such
that
\begin{equation} \label{MTPmeasure}
 a \, r ^{\delta}  \ \le    \  \mu(B)  \ \le    \   b \, r
^{\delta}
\end{equation}
for any ball $B=B(x,r)$ with $x\in X$ and radius $r\le r_0$. Such a measure is said to be \emph{$\delta$-Ahlfors regular}. It is well known that if $X$ supports a $\delta$-Ahlfors regular measure $\mu$, then   $\dim_{\rm H} X = \delta$ and moreover
that $ \mu$ is strongly equivalent to $\delta$-dimensional Hausdorff measure ${\cal H}^\delta$ -- see \cite{F,MAT} for
the details. The latter implies that \eqref{MTPmeasure} is valid with $\mu$ replaced by  $\cH^s $.  Next, given $s > 0$  and a ball $B=B(x,r)$ we define
the scaled ball
$$
B^s:=B\big(x,r^{\frac{s}{\delta}}\big)\,.
$$
and so by definition $B^{\delta}=B$.  The following  Mass Transference Principle \cite{BV2006} allows us to transfer
$\cH^\delta$-measure theoretic statements for $\limsup$ subsets of
$X$ to general $\cH^s$-measure theoretic statements.

	\begin{theorem}[MTP: balls to balls]\label{MTPB}
		Let $X$ be a locally compact subset of $\R^d$ equipped with a  $\delta$-Ahlfors regular measure $\mu$.    Let $\{B_n\}_{n \in \N}$ be a sequence of balls in $X$ with radius $r(B_n)\to 0$ as $n\to\infty$.
		Let $s \ge  0 $ and suppose that
		$$\mathcal{H}^\delta \big( \limsup_{n\to\infty}B_n^s\big)=\mathcal{H}^\delta(X).$$
		Then,		$$\mathcal{H}^s\big(\limsup_{n\to\infty}B_n\big)=\mathcal{H}^s(X).$$
	\end{theorem}
	
\bigskip

 Note that by the definition of Hausdorff dimension,  Theorem~\ref{MTPB} implies that
\begin{equation}  \label{oj} \dim_{\rm H} \big(\limsup_{n\to\infty}B_n   \big)  \ge   s \, , \end{equation}
and moreover that
$\mathcal{H}^s(\limsup_{n\to\infty}B_n) = \infty   $ if $s < \delta
$.   We now describe a recent result due to Wang $\&$  Wu \cite{WW2021} that  gives a lower bound for the  Hausdorff dimension of $\limsup$ sets defined via rectangles rather than just  balls.  So with this in mind, fix an integer $p \ge 1$ and for $1 \le i \le p$, let
 $X_i$ be a subset of $\R^{d_i}$.  Obviously, if  $B_{i}$ is a ball in $X_i$ then  $\prod_{i=1}^pB_{i}$ is in general a rectangle in the product space $ \prod_{i=1}^p   X_i$.  The following statement  for $\limsup $ sets arising from sequences of such rectangles is a much simplified version of \cite[Theorem 3.4] {WW2021}. As we shall see, it is more than adequate for our purpose.

	\begin{theorem}[MTP: rectangles to rectangles]\label{MTPRG}
	For each $1 \le i \le p$, let
 $X_i$ be a locally compact subset of $\R^{d_i}$ equipped with a  $\delta_i$-Ahlfors regular measure $\mu_i$.  Let $\{B_{i,n}\}_{n \in \N}$ be a sequence of balls in $X_i$ with radius $r(B_{i,n}) \to 0$ as $n\to\infty$ for each $1\le   i\le   p$ and   assume that there exist ${\mathbf{v}}=(v_1,\dots, v_p)\in(\mathbb{R}^+)^p$ and a sequence $\{r_n\}_{n \in \N}$ of positive real numbers such that
  \begin{align}\label{assump}
 r(B_{i,n}) = r_n^{v_i} \quad \text{for all \  }  1\leq i \leq p.
\end{align}
  Suppose that there exists $(s_1, \dots, s_p) \in \prod_{i=1}^p(0,\delta_i)$ such that
		\begin{equation}\label{fullmeasure}		\mu_1\times\cdots\times\mu_p\Big(\limsup_{n\to\infty}{\textstyle \prod_{i=1}^p}B_{i,n}^{s_i}\Big)= \mu_1\times\cdots\times\mu_p \Big( {\textstyle \prod_{i=1}^p}X_i\Big).
		\end{equation}
		%where for each $1\leq i\leq p$, $s_i={u_i \delta_i\over v_i}$ for some $0<u_i<v_i$.
		 %Let  $s_1,\dots, s_p{\color{red} >} 0$ and
%		Then for any choice of ${\mathbf{u}}=(u_1, \dots, u_p)$ with ${\color{red} 0\leq u_i <   v_i}$ and $\frac{\delta_i \, u_i}{{\color{red} v_i}}=s_i$ ($ 1 \le i \le p$), we have that
%		
Then, we have that
		\begin{equation*}
			\dim_{\rm H}\big(\limsup_{n\to\infty} {\textstyle \prod_{i=1}^p}B_{i,n} \big)\ge   \min_{ 1\leq i\leq p} s({\mathbf{u}},{\mathbf{v}},i), \,
		\end{equation*}
%where
%$$\mathcal{A}:=\{{\color{red} v_i}: \ 1\le   i\le   p\} $$
where ${\mathbf{u}}=(u_1,\dots, u_p)$ with $u_i=s_iv_i/\delta_i$ for $1\leq i\leq p$, and
$$ s({\mathbf{u}},{\mathbf{v}},i):=\sum_{k\in \mathcal{K}_1(i)} \delta_k+\sum_{k\in\mathcal{K}_2(i)} \delta_k \big(1-\frac{  v_k-u_k}{v_i}\big)
+\sum_{k\in\mathcal{K}_3(i)}\frac{u_k\delta_k}{v_i} \, ,
$$
\noindent with the sets
$$
\mathcal{K}_1(i):=\{1 \le k \le p : u_k \ge   v_i\}, \qquad \mathcal{K}_2(i):=\{ 1 \le k  \le p :  v_k \le   v_i\},  $$
and
$$ \mathcal{K}_3(i):=\{1, \dots, p\}\setminus \left(\mathcal{K}_1(i) \cup \mathcal{K}_2(i)\right)
		$$
forming a partition of $\{1, \dots, p\}$.
	\end{theorem}

\medskip

Note that  since  the radius $r(B_{i,n})\to  0 $ as $n\to\infty$ for each $1\le   i\le   p$ we automatically have  that $ \lim_{n\to\infty} r_n=0 $.
Also note that if  {\eqref{assump}} holds for some
${\mathbf{v}}=(v_1,\dots, v_p)\in(\mathbb{R}^+)^p$, then it  holds for $c{\mathbf{v}}=(cv_1,\dots, cv_p)$ where $c>0$ is a constant. Thus, the choice of $v$ and therefore $u$ is not unique. However, it is easily seen that the ``dimension number'' $s({\mathbf{u}},{\mathbf{v}},i)$ is not effected; i.e.
\begin{equation*}
s({\mathbf{u}},{\mathbf{v}},i) = s(c{\mathbf{u}},c{\mathbf{v}},i) \, .
\end{equation*}
For the sake of convenience,  we  refer to a collection of rectangles $\{R_n := {\textstyle \prod_{i=1}^p}B_{i,n}  \}_{n\in \mathbb{N}}  $  with sidelengths satisfying {\eqref{assump}} as  a collection of \emph{rectangles with exponent $\mathbf{v}$.}  Thus,    Theorem~{\ref{MTPRG}} can be regarded as the  dimension analogue of the original  Mass Transference Principle for $\limsup$ sets arising from  rectangles with exponent $\mathbf{v}$.

\medskip

%r9a #&#
\begin{remark}
\label{rem9a}
It is worth mentioning that assumption {\eqref{assump}} can be easily  weakened to the following statement:  for any $1\leq i,j \leq p$,
\begin{equation*}
 \lim_{n\to\infty} {\log r(B_{i,n}) \over \log r(B_{j,n})} \ \ \text{ \ \ exists and is finite.}
\end{equation*}
To see this, let  $r_n=r(B_{1,n})$.  Then,   there exists ${\mathbf{v}}=(v_1,\dots, v_p)\in(\mathbb{R}^+)^p$ and  sequences $\{v_{i,n} \}_{n \in  \mathbb{N}  } $  for each  $1\leq i \leq p$,    such  that
\begin{equation*}
r(B_{i,n})=r_n^{v_{i,n}} \quad \text{with} \ \lim_{n\to\infty}v_{i,n}=v_i \ .
\end{equation*}
Now, given $\epsilon > 0$   consider the associated  $\limsup$ set
$
R^*_\epsilon $ obtained by replacing the balls $B_{i,n}$ by balls     $B^*_{i,n}$  with the same centre  but radius $r(B^*_{i,n})=r_n^{v_i + \epsilon}$ $(1\leq i \leq p)$. Then,
\begin{equation*}
R^*_\epsilon:=\limsup_{n\to\infty} {\textstyle \prod_{i=1}^p}B^*_{i,n}    \subset \limsup_{n\to\infty} {\textstyle \prod_{i=1}^p}B_{i,n}
\end{equation*}
and on  applying  Theorem~{\ref{MTPRG}} we obtain a lower bound for $\dim_{\mathrm{H}} R^*_\epsilon$  which converges to the desired  dimensional number as $\epsilon \to 0$.
\end{remark}

\medskip

\begin{remark}
		 In the above setup the rectangles arise as products of balls $B_i$ in $X_i$.  Balls of radius $\rho$ are of course $\rho$-neighbourhoods of special points; namely their centres. The general form of the Mass Transference Principle  of  Wang $\&$ Wu is based on the framework of ubiquitous systems.  This allows them to naturally consider the situation in which the balls $B_i$ are replaced by $\rho$-neighbourhoods of special sets called  resonant sets.  It is easily verified that the $\kappa$-scaling property for resonant sets within the general dimension statement   \cite[Theorem 3.4]{WW2021} is satisfied  with $\kappa=0$ when the resonant sets are points.  In turn,  this together with \cite[Proposition 3.1]{WW2021}  directly yields Theorem~\ref{MTPRG}.
 It is also worth mentioning that if we replace the full measure condition \eqref{fullmeasure} by the stronger `local ubiquity system for rectangles' condition  \cite[Definition 3.2]{WW2021} and also assume that the radii of the balls in the given sequence
are non-increasing,  then we are able to conclude \cite[Theorem 3.2]{WW2021} that
   $$ \cH^{s} \big(\limsup_{n\to\infty} {\textstyle \prod_{i=1}^p}B_{i,n} \big) = \cH^{s} \big( {\textstyle \prod_{i=1}^p} X_i \big)  \quad  {\rm with}  \quad s=\min_{ 1\leq i\leq p} s({\mathbf{u}},{\mathbf{v}},i).
   $$
   In other words, we obtain a complete `rectangle to rectangle' analogue of the original Mass Transference Principle.
   \end{remark}

\medskip

It is easily seen  that in the case of balls,  Theorem~\ref{MTPRG} coincides with the dimension statement \eqref{oj}. Indeed, when $p=1$ we have that  that $\mathcal{K}_1(1)=\mathcal{K}_3(1)=\emptyset$ and $\mathcal{K}_2(1)=\{1\}$.    Hence, it follows that
$$s({\mathbf{u}},{\mathbf{v}}, 1)=\delta_1(1-\frac{ v_1-u_1}{ v_1})=\frac{\delta_1 u_1}{ v_1}= s_1  \, , $$
	and so	Theorem~\ref{MTPRG} implies that
		$$\dim_{\rm H}\limsup_{n\to\infty}B_{1,n}\ge    s_1$$
as claimed.

\subsection{Proof of Theorem~\ref{dimresult} via a dimension theorem for rectangular  targets  \label{pfthm6} }

As mentioned in Remark~\ref{rectlove}, straight after the statement of Theorem \ref{dimresult},  we deduce the theorem from a more  general statement concerning rectangular target sets.  This we now describe and prove.
For $ 1 \le i \le d$, let $\psi_i: \R^+ \to \R^+$ be a real positive function. For convenience, let $\Psi := (\psi_1,\dots,\psi_d)$  and for $n \in \N$ let $\Psi(n) := (\psi_1(n),\dots,\psi_d(n))$.  Fix some point $\va:= (a_1, \ldots,  a_d)  \in \T^d$ and for $n \in \N$, let
$$
 R\big(\va, \Psi(n) \big):= \Big\{ \vx  \in \mathbb{T}^d :  \|x_i-a_i\| \le \psi_i(n)   \ (1\le   i \le   d) \Big\}.
$$
%{\color{blue} Bing: Should it be
%	$$
%	R_n = R\big(\va, \Psi(n) \big):= \Big\{ \vx  \in \mathbb{T}^d :    |x_i-a_i| \le \psi_i(n)\ \text{for}\ 1\le   i\le   d   \Big\}?
%	$$}
Clearly, $R\big(\va, \Psi(n) \big)$ is  a rectangle centred at the fixed point $\va$.
In turn, let
$$
W(T,\Psi,\va):=\big\{\vx\in\mathbb{T}^d: T^n(\vx)\in R\big(\va, \Psi(n) \big) \ \ \text{for infinitely many}\ n\in\mathbb{N}\big\}  \ . $$
It is evident that    the family of rectangular target sets $\{R\big(\va, \Psi(n) \big)\}_{n\ge   1}$ satisfy the bounded property
$ ({\boldsymbol{\rm B}}) $ and indeed the stronger property $ ({\boldsymbol{\rm P}_{\!\va}}) $ based on Gallagher's property $({\boldsymbol{\rm P}}) $  as described in \S\ref{galSV}.   Note that when $\psi:=\psi_1= \dots = \psi_d  $,  the rectangles are  squares and  so
$$
W(T,\Psi,\va)  = W(T,\psi,\va)  \, .
$$
 Also,  it is  easily verified  that if
$$
T = {\rm diag} \, (\beta_1, \ldots, \beta_d  )     \qquad   (\beta_i \in \R)  \,
$$
then
\begin{equation}  \label{yesyes}
W(T,\Psi,\va)=\big\{\vx\in\mathbb{T}^d: |T_{\beta_i}^nx_i-a_i| \le \psi_i(n)  \ (1\le   i \le   d) \ \ \text{for infinitely many}\ n\in\mathbb{N}\big\}  \, ,
\end{equation}
%{\color{blue} Bing: should it be
%	\begin{equation*}  \label{yesyes}
%		W(T,\Psi,\va)=\big\{\vx\in\mathbb{T}^d:  |T_{\beta_i}^nx_i-a_i| \le \psi_i(n) \ \ \text{for}\ 1\le   i\le   d\  \text{and infinitely many}\ n\in\mathbb{N}\big\}  \, ?
%	\end{equation*}
%}
where $T_{\beta_i}$ is the  standard $\beta$-transformation   with $\beta=\beta_i$.

\medskip

It turns out that the Hausdorff  dimension of the shrinking target set $W(T,\Psi,\va)$ is dependent on the set $\UP$ of accumulation points $ \mathbf{t}=(t_1,t_2, \ldots, t_d ) $ of the sequence $\big\{(-\frac{\log\psi_1(n)}{n},\cdots,
-\frac{\log\psi_d(n)}{n})\big\}_{n\ge   1}$.

\medskip

	\begin{theorem}\label{rectangledimresult}
Let $T$ be a real, non-singular matrix transformation of the torus $\mathbb{T}^d$. Suppose that $T$ is diagonal with all  eigenvalues $\beta_1,\beta_2,\dots,\beta_d$  strictly larger than $1$.  Assume that $1<\beta_1\le \beta_2\le\cdots\le \beta_d$.  For $ 1 \le i \le d$, let $\psi_i: \mathbb{R}^+\to \mathbb{R}^+$ be a real positive %{\color{blue} decreasing}
		function and $\va \in \mathbb{T}^d$. Assume that $\UP$ is bounded. Then
		$$\dim_{\rm H} W(T,\Psi, \va)=\sup_{\mathbf{t}\in\UP} \min_{1\le   i\le   d}\big\{\theta_i(\mathbf{t})\big\},$$
		where
		%\begin{equation*}
		%		\delta_i(\mathbf{t}):= \sum_{\beta_k\ge  \beta_i}1+\sum_{\beta_ke^{t_k}<\beta_i}\left(1-\frac{t_k}{\log\beta_i}\right)+\sum_{ \beta_ke^{t_k}\ge   \beta_i>\beta_k }\frac{\log\beta_k}{\log\beta_i}
		%	\end{equation*}
		%	and
		\begin{equation*}
			\theta_i(\mathbf{t}):=
			\sum_{k\in \mathcal{K}_1(i)}1+\sum_{k\in \mathcal{K}_2(i)}\left(1-\frac{t_k}{\log\beta_i+t_i}\right) +\sum_{k\in \mathcal{K}_3(i)}\frac{\log\beta_k}{\log\beta_i+t_i}
		\end{equation*}
and, in turn
			\begin{equation*}
			\mathcal{K}_1(i):= \{1\le   k\le   d: \log \beta_k>\log \beta_i+t_i\}, \ \mathcal{K}_2(i):= \{1\le   k\le   d: \log \beta_k+t_k\le   \log \beta_i+t_i\}, \
			\end{equation*}
and
\begin{equation*}
  \mathcal{K}_3(i):=\{1, \dots, d\}\setminus (\mathcal{K}_1(i)\cup \mathcal{K}_2(i)).
\end{equation*}
	\end{theorem}

\medskip

\begin{remark}
It is easily seen the value of  $\theta_i(\mathbf{t})$ remains  unchanged if replace  $>$ to $\ge   $ in $\mathcal{K}_1(i)$, and/or replace  $\le  $ by $< $ in $\mathcal{K}_2(i)$.
%We remark that for the above formula of $\theta_i(\mathbf{t})$, in the summation conditions, we can change the place of ``$=$" in ``$\le  $" or ``$\ge  $", i.e., we have the following equalities
%\begin{align*}
%\theta_i(\mathbf{t})&:=
%			\sum_{\beta_k>\beta_ie^{t_i}}1+\sum_{\beta_ke^{t_k}\le   \beta_ie^{t_i}}\left(1-\frac{t_k}{\log\beta_i+t_i}\right) +\sum_{\beta_ie^{t_i}\ge  \beta_k>\beta_ie^{t_i-t_k}}\frac{\log\beta_k}{\log\beta_i+t_i}.
%\theta_i(\mathbf{t})\\
%&=
%			\sum_{\beta_k\ge  \beta_ie^{t_i}}1+\sum_{\beta_ke^{t_k}\le   \beta_ie^{t_i}}\left(1-\frac{t_k}{\log\beta_i+t_i}\right) +\sum_{\beta_ie^{t_i}>\beta_k>\beta_ie^{t_i-t_k}}\frac{\log\beta_k}{\log\beta_i+t_i}\\
%			&= \sum_{\beta_k\ge  \beta_ie^{t_i}}1+\sum_{\beta_ke^{t_k}< \beta_ie^{t_i}}\left(1-\frac{t_k}{\log\beta_i+t_i}\right) +\sum_{\beta_ie^{t_i}>\beta_k\ge   \beta_ie^{t_i-t_k}}\frac{\log\beta_k}{\log\beta_i+t_i}.
%\end{align*}
%	\sv{In the formula of $\theta_i(\mathbf{t})$, $t_i$ can be infinity.}
\end{remark}	
\medskip

We now deduce Theorem \ref{dimresult} from Theorem \ref{rectangledimresult}.

	\begin{proof}[Proof of Theorem \ref{dimresult} modulo Theorem \ref{rectangledimresult}]

For the real positive  function $\psi$ in Theorem \ref{dimresult},  we first suppose that its lower order  at infinity is bounded; that is
	$$\lambda=\lambda(\psi):=\liminf_{n\to\infty}\frac{-\log\psi(n)}{n}<+\infty.$$
	
\noindent With this in mind, put $\psi_1=\psi_2=\cdots=\psi_d=\psi$ in the statement of Theorem \ref{rectangledimresult} and note that
any $\mathbf{t}$ in $\UP$ is of the form $\mathbf{t}=(t,t,\dots,t)$  where   $t\in\mathcal{U}(\psi)$ -- the set of accumulation points of the sequence $ \big\{  -\frac{\log\psi(n)}{n} \big\}_{n \ge 1} $. We remark that $\lambda<+\infty$ means that $\UP$ is bounded. Hence, for any $1\le   i \le   d$ we have that
		\begin{equation*}
			\theta_i(t):=\theta_i(\mathbf{t})=
			\sum_{k:\beta_k>\beta_ie^{t}}1+\sum_{k:\beta_k\le   \beta_i}\left(1-\frac{t}{\log\beta_i+t}\right) +\sum_{k:\beta_ie^{t}\ge  \beta_k>\beta_i}\frac{\log\beta_k}{\log\beta_i+t}.
		\end{equation*}
	Now let
		$$k_1:=\max\{1\le   k\le   d: \beta_k\le     \beta_i\}, $$
		and
		$$k_2:=\max\{1\le   k\le   d: \beta_k \le   \beta_ie^t\}.$$

\noindent Then, the above expression for $\theta_i(t)$  becomes
		\begin{equation} \label{po} \theta_i(t)=\sum_{k=k_2+1}^d1+\sum_{k=1}^{k_1}\frac{\log\beta_i}{\log\beta_i+t}
+\sum_{k=k_1+1}^{k_2}\frac{\log\beta_k}{\log\beta_i+t} \end{equation}
and noting that $\beta_k=\beta_i$ for $i+1\le   k\le   k_1$ whenever $k_1 > i$, it follows that
\begin{eqnarray*}\label{thetacal}			\theta_i(t)&=&\frac{\sum_{k=k_2+1}^d(\log\beta_i+t)+i\log\beta_i+\sum_{k=i+1}^{k_2}\log\beta_k}{\log\beta_i+t} \nonumber \\[2ex] &=&\frac{i\log\beta_i-\sum_{k=k_2+1}^d(\log\beta_k-\log\beta_i-t)+\sum_{k=i+1}^{d}\log\beta_k}{\log\beta_i+t}.
\end{eqnarray*}
%\red{[Need to get into the same form as Theorem \ref{dimresult}]}
The upshot of this together with  Theorem \ref{rectangledimresult} is that
		\begin{equation*}
			\dim_{\rm H} W(T,\psi, \va)= \sup_{t \in \mathcal{U}(\psi)}\min_{1\le   i\le   d}\theta_i(t).
		\end{equation*}

\noindent  Now observe that since $\theta_i(t)$ is a decreasing function in $t$, the above right hand side is equal to $ \min_{1\le   i\le   d}\theta_i(\lambda)$ where $\lambda=\lambda(\psi) $ is the lower order at infinity of the function $\psi$.   Thus, under the assumption that $\lambda$ is bounded, we have that
$$
\dim_{\rm H} W(T,\psi, \va)=\min_{1\le   i\le   d}\theta_i(\lambda)
$$
as desired.  To deal with the  case that  $\lambda=\lambda(\psi)=+\infty $, given any real number $M>0$ consider the function    $\psi_M: \mathbb{R}^+\to \mathbb{R}^+   :   x \to e^{-xM} $.  Then, by definition $ \mathcal{U}(\psi_M) = \{ M\}$ and for $M$ sufficiently large
\[
W(T, \psi, \va) \subset W(T, {\psi_M}, \va)  \,
\]
and so it follows that
\begin{equation} \label{poo}
0 \le   \dim_{\rm H} W(T,\psi, \va)\le   \dim_{\rm H} W(T, {\psi_M}, \va)\le   \min_{1\le   i\le   d}\theta_i(M).
\end{equation}
Now with reference to \eqref{po},  we have that $t=M$ and so for $M$ sufficiently large:   $k_2=d$ for $1\le   i \le   d$. Hence,  for  any   $1\le   i \le   d$
$$
\lim_{M\to\infty} \theta_i(M)=\lim_{M\to\infty}\left(\sum_{k=1}^{i}\frac{\log\beta_i}{\log\beta_i+M}
+\sum_{k=i+1}^{d}\frac{\log\beta_k}{\log\beta_i+M}\right) =0.
$$
This together with \eqref{poo} implies that
$$
\dim_{\rm H} W(T,\psi, \va)=0.
$$
	\end{proof}

\subsubsection{Proof of Theorem~\ref{rectangledimresult}  \label{pfrecdim}}

We  start with a brief discussion that sums up various fundamental  notions and  statements that we will require during the course of establishing Theorem~\ref{rectangledimresult}. The statements are concerned with the distribution of the preimages of a fixed ball under a given $\beta$-transformation $T_\beta$.   As usual, let $\beta\in \mathbb{R}$ such that $|\beta|>1$  and let % \slv{[???why not $\cP$ to be consistent with Markov subsystem of  \S\ref{poiu}??]}
	\[
	\mathcal{Q}=\left\{\Big[0, {1\over |\beta|}\Big), \dots, \Big[{k \over |\beta|}, {k+1 \over |\beta|}\Big), \dots, \Big[{\lfloor |\beta| \rfloor \over |\beta|}, 1\Big) \right\}
	\]
	be the natural partition of {$[0,1)$}. The $n$-th refinement of $\mathcal{Q}$ is defined as
	\[
	\mathcal{Q}^{n}:=\left\{ Q_{i_0} \cap T_\beta^{-1}(Q_{i_1}) \cap \cdots \cap T_\beta^{-(n-1)}(Q_{i_{n-1}}): \quad Q_{i_j}\in \mathcal{Q}\ \text{for } 0\le   j \le   n-1\right\}.
	\]
	The elements in $\mathcal{Q}^{n}$ are called \emph{cylinders of order $n$}.
Evidently, the cylinders are disjoint and the restriction of $T_{\beta}^n$ on
each cylinder is continuous linear of slop $\beta^n$. Now given a point $a \in \T$,
 consider the preimage of the  ball $B(a,r)$ under $T_{\beta}^n$. It can be verified
 that this preimage  consists of {disjoint} intervals   whose  lengths are bounded
 above by $2r|\beta|^{-n}$. Indeed, we can write
	\begin{equation}\label{preimages1}
	T_{\beta}^{-n} \big(B(a, r) \big)=\bigcup_{j=1}^{N_{n}}I_{n,j},
\end{equation}
	where each $I_{n,j}$ is an interval lying in some cylinder of order $n$ and  $N_{n}$ is the number of such intervals.

For $\beta>1$, via the work of R\'{e}nyi \cite{R57}, it  follows that the (total) number of cylinders of order $n$ of $T_\beta$ is bounded from above by $\beta^{n+1}/(\beta-1)$. Hence,
\begin{equation}\label{needlabel}
N_{n}\le  \beta^{n+1}/(\beta-1).
\end{equation}
\begin{remark} \label{-need}
It is worth mentioning  that on exploiting the
well-known fact that the topological entropy of $T_\beta$ is $\log  |\beta|$ for any $\beta$ with $|\beta| > 1 $,  we obtain  the weaker bound
\begin{equation}\label{needlabelB}
N_{n} \le  |\beta|^{n(1+\epsilon)}
\end{equation}
 for any $\epsilon > 0$ and $n$ sufficiently large. This suffices for not only proving Lemma \ref{limitexist} below but  more importantly  for establishing the upper bound for the dimension within the context of Theorems~\ref{dimensionalone} and \ref{multiplicative} in which $\beta$ is allowed to be negative.  \end{remark}	

Now, suppose $\beta>1$. Recall that a cylinder $I$ of order $n$  is said to be  \emph{full} for $T_{\beta}$  if  $T_{\beta}^n(I)= \T $.   With this in mind, Bugeaud $\&$ Wang \cite[Theorem 1.2]{BW14} proved that every $(n+1)$ consecutive cylinders of order $n$  contains at least one full cylinder.  This gives rise to the following useful fact that we will make use of on multiple occasions.

\noindent {\bf Fact BW: \ }\label{DF}  \emph{
The distance between any two consecutive full cylinders  is less than $(n+1)\beta^{-n}$.
Furthermore, since any full cylinder intersects $T_{\beta}^{-n} \big(B(a, r) \big)$ it follows that the distance between any two consecutive intervals $I_{n,j}$ and $I_{n, j+1}$ is  less than $(n+3)\beta^{-n}$; i.e.,
$${\rm dist}(I_{n,j},I_{n,j+1})\le   (n+3)\beta^{-n}.$$
}

%\begin{remark}
%The distance between any two consecutive full cylinders  is less than $(n+1)\beta^{-n}$.
%Furthermore, since any full cylinder intersects $T_\beta^{-n}(B(a,r))$ it follows  Hence, the distance between two intervals intersecting two consecutive full cylinders respectively, is less than $(n+3)\beta^{-n}$. This then implies that any two consecutive intervals $I_{n,j}$ and $A_{n, j+1}$ has distance less than $(n+3)\beta^{-n}$, i.e.,
%$$\text{dist}(I_{n,j},A_{n,j+1})\le   (n+3)\beta^{-n}.$$
%
%\end{remark}

We now move onto the task of proving  Theorem~\ref{rectangledimresult}. This will be done  by establishing  the upper and lower bounds  for $ \dim W(T,\Psi,\va) $ separately.

\begin{proposition}\label{upperbound}
		Under the setting of Theorem \ref{rectangledimresult},  we have that
		$$\dim_{\rm H} W(T,\Psi,\va)\le   \sup_{\mathbf{t}\in\UP} \min_{1\le   i\le   d}\{\theta_i(\mathbf{t})\}.$$
	\end{proposition}

We first establish the proposition in the special case that $\UP$ consists of a single point.

\begin{lemma}\label{limitexist}
Under the setting of Theorem \ref{rectangledimresult},
 assume in addition that there exists $\mathbf{t}=(t_1, \dots, t_d)\in (\mathbb{R}^+)^d$ such that %$\psi_i(n)$ satisfies
\[
\lim_{n\to \infty}\frac{-\log\psi_i(n)}{n}=t_i   \quad  {\rm for \ all} \ \   1\le    i \le   d.
\]
Then 	$$\dim_{\rm H} W(T,\Psi,\va)\le    \min_{1\le   i\le   d}\{\theta_i(\mathbf{t})\}.$$
\end{lemma}
%\begin{remark}
%	\sv{$t_i$ can be infinity.}
%\end{remark}
\begin{proof}
Observe that we can re-write \eqref{yesyes} as
\begin{equation}\label{limsupset}	
	W(T, \Psi, \va)=\limsup_{n\to \infty}   \ T_{\beta_1}^{-n} \big( B(a_1, \psi_1(n)) \big)\times\cdots\times T_{\beta_d}^{-n} \big( B(a_d, \psi_d(n)) \big) ,
\end{equation}
where $T_{\beta_i} $ is the standard  $\beta$-transformation   with $\beta = \beta_i$.   As usual,    we do not distinguish between  $\beta$-transformations acting on the unit interval $[0,1)$  or the torus $\T$.
  The proof of Lemma~\ref{limitexist} relies on finding an ``efficient'' covering by balls  of the $\limsup$ set \eqref{limsupset}.   So with this in mind,
%		Choose $\mathbf{t}\in\mathcal{U}$ and $1\le   i\le   d$ such that $$\theta_i(\mathbf{t})=\max_{\mathbf{t}\in\mathcal{U}} \min_{1\le   i\le   d}\big\{\theta_i(\mathbf{t})\big\}.$$
%		Denote by $\{n_q\}_{q\ge   1}$ be the subsequence of $\mathbb{N}$ with
%		 $$\lim_{q\to\infty}(-\frac{\log\psi_1(n_q)}{n_q},\cdots,
%		-\frac{\log\psi_d(n_q)}{n_q})=(t_1,\cdots,t_d)=\mathbf{t}.$$
for  any  $1\le   i\le   d$,
%For any $n\ge   1$, note that $T_{\beta_i}^n$ is a piecewise linear mapping with slope $\beta_i^n$  and recall that the maximal subintervals of $\T$ on which the restriction of $T_{\beta_i}^n$ is linear are called the cylinders of order $n$ associated with $T_{\beta_i}$.
by (\ref{preimages1}) we have that
%that the preimage of the ball $B(a_i, \psi_i(n))$ under  $T_{\beta_i}^n$ consists of intervals with lengths bounded by $2\psi_i(n)\beta_i^{-n}$ from above. Formally,
\begin{equation}\label{preimages}
	T_{\beta_i}^{-n} \big(B(a_i, \psi_i(n)) \big)=\bigcup_{j=1}^{N_{i,n}}I_{n,j}^{(i)}  \, ,
\end{equation}
where each $I_{n,j}^{(i)}$ is an interval  lying in some cylinder of order $n$ and  $N_{i,n}$ is the number of such intervals.
%		with each interval $I_{n,j}^i$ being the intersection of $T_{\beta_i}^{-n} B(a_i, \psi_i(n))$ and a cylinder of order $n$.
Now for $ n \in \N$, let
$$
J_n := \big\{  \vj =(j_1, \ldots, j_d):   1 \le j_i \le N_{i,n} \ (1\le i\le d) \big\}
$$
and for $\vj \in J_n$, let
$$
R_{n,\vj} :=  I_{n,j_1}^{(1)}\times\cdots\times  I_{n,j_d}^{(d)}     \, .
$$
In turn, let
$$
R_n = \bigcup_{\vj \in J_n} R_{n,\vj}  \qquad {\rm and \ }   \qquad  \cR_n := \big\{ R_{n,\vj} \, : \,  \vj  \in J_n \big\}  \, .
$$
Then in view of \eqref{preimages}, we can re-write \eqref{limsupset}  as
$$
W(T, \Psi, \va)=   \bigcap_{N=1}^\infty \bigcup_{n=N}^\infty  R_n \,
$$
and it follows that  for any $N\ge 1$,
$$W(T,\Psi,\va) \, \subset \,  \bigcup_{n=N}^\infty  R_n \, . $$
%$$\bigcup_{n=N}^\infty\bigcup_{j_1=1}^{N_{1,n}}\cdots\bigcup_{j_d=1}^{N_{d,n}} A_{n,j_1}^{(1)}\times\cdots\times  A_{n,j_d}^{(d)} \, . $$
In other words,  the  collection $\{ \cR_n : n = N, N+1, \ldots \} $ of rectangles $R_{n,\vj}$ form a cover for the set $W(T,\Psi, \va)$.
Now, observe that along the direction of the $i$-th axis $ (1\le i\le d)$, by construction for each $1\le   j < N_{i,n}$
% the distance between any two consecutive full cylinders  is less than $(n+1)\beta_i^{-n}$.
the sides  $I_{n,j}^{(i)}$ and $I_{n,j+1}^{(i)}$ are disjoint  and thus the rectangles  in $\cR_n $ are disjoint.   On the other hand, by Fact BW %the distance between consecutive $I_{n,j}^i$ satisfies
$$\text{dist}(I_{n,j}^{(i)},I_{n,j+1}^{(i)})\le   (n+3)\beta_i^{-n} $$
and so along the direction of the  $i$-th axis,  the distance between consecutive rectangles  in $\cR_n$  is at most $(n+3)\beta_i^{-n}$.

%		For fixed $1\le   i\le   d$, there exists a subsequence $\{n_k\}$ such that the limits $\lim\limits_{k\to\infty}\frac{-\log\psi_i(n_k)}{n_k}$ exist for every $1\le   i\le   d$. For simplicity, let us assume that the limits $\lim\limits_{n\to\infty}\frac{-\log\psi_i(n)}{n}$ exist.  Fix $1\le i\le d$.

We now estimate the number of balls $\Bni$ of diameter $2\psi_i(n)\beta_i^{-n}$   (the sidelength of the rectangles in $\cR_n$  along the direction of the $i$-th axis)  needed to cover the set $R_n$. We start by covering a fixed generic rectangle $ R = R_{n, \vj} \in \cR_n$. It is easily verified that we can find a collection $\cB_{i,n}(R)$  of balls $\Bni$ that covers $R$ with
\begin{equation} \label{mn}
\# \cB_{i,n}(R) 	 \  \le  \  2^d  \prod_{1\le k  \le d \ : \atop \psi_k(n)\beta_k^{-n}\ge   \psi_i(n)\beta_i^{-n}}\frac{\psi_k(n)\beta_k^{-n}}{\psi_i(n)\beta_i^{-n}}  \, .
\end{equation}
Indeed, we can simply take the natural cover in which we split $R$ into closed balls $\Bni$ which are disjoint apart from at the boundary.  Now observe that the  collection $\cB_{i,n}(R)$  will also cover other rectangles in $\cR_n$ along the direction of the $k$-th axis ($1 \le k \le d$)  if the separation in that direction is small compared to the diameter of the balls $\Bni$; that is, in view of  Fact BW if
 $$  \ \ (n+3)\beta_k^{-n}< 2 \psi_i(n)\beta_i^{-n} \, . $$
 In particular,  this leads to the following lower bound for the number $M_{i,n}(R) $  of rectangles covered by $\cB_{i,n}(R)$ :
\begin{eqnarray}  \label{mnb}
M_{i,n}{(R)}  & :=  &  \#  \Big\{  R_{n,\vj} \in \cR_n  \, : \, R_{n, \vj}  \subseteq \bigcup_{\Bni \in \cB_{n,i}(R) } \!\! \Bni    \Big\}  \nonumber \\[2ex]
& \ge &  \!\!\!\! \prod_{1\le   k\le   d \ : \atop (n+3)\beta_k^{-n}< \psi_i(n)\beta_i^{-n}} \!\! \frac{2 \psi_i(n)\beta_i^{-n}}{(n+3)\beta_k^{-n}}  \, .
\end{eqnarray}
\\
The upshot is  that there is a collection $\cB_{n,i} $ of balls of $\Bni$  that cover the set $R_n$   with
$$\# \cB_{n,i}   \  \le    \
 \frac{ \#\cR_n}{M_{i,n}{(R)} }    \  \  \#\cB_{n,i}(R)   \  \ =  \ \   \prod_{j=1}^d  N_{j,n}  \ \cdot \   \frac{ \#\cB_{n,i}(R) }{M_{i,n}{(R)} }  \, .    $$
This together with \eqref{needlabel},  \eqref{mn} and \eqref{mnb}  implies that
\begin{eqnarray*}
\# \cB_{n,i}   \ & \le     &
 {\prod_{j=1}^d  \frac{\beta_j^{n+1}}{\beta_j-1}}  \  \cdot \prod_{1\le   k\le   d \ : \atop (n+3)\beta_k^{-n}< \psi_i(n)\beta_i^{-n}} \!\!\!\! \frac{(n+3)\beta_k^{-n}}{2 \psi_i(n)\beta_i^{-n}}   \  \  \cdot  \  \ 2^d  \!\!\!\!  \prod_{1\le k  \le d \ : \atop \psi_k(n)\beta_k^{-n}\ge   \psi_i(n)\beta_i^{-n}}  \!\!\!\! \frac{\psi_k(n)\beta_k^{-n}}{\psi_i(n)\beta_i^{-n}}  \\[3ex]
&=& 2^d \cdot  {\prod_{j=1}^d  \frac{\beta_j^{n+1}}{\beta_j-1}}   \ \ \prod_{k\in \mathcal{K}_{n,1}(i)}\frac{(n+3)\beta_k^{-n}}{2 \psi_i(n)\beta_i^{-n}}   \ \   \prod_{k\in  \mathcal{K}_{n,2}(i)}\frac{\psi_k(n)\beta_k^{-n}}{\psi_i(n)\beta_i^{-n}},
\end{eqnarray*}
where
\begin{eqnarray*}
\mathcal{K}_{n,1}(i) &:= & \left\{1\le   k\le   d: (n+3)\beta_k^{-n}<\psi_i(n)\beta_i^{-n}\right\} \\[2ex]
& = &  \left\{ 1\le   k\le   d: -{ \log(n+3) \over n}+\log \beta_k>-{\log\psi_i(n) \over n}+\log \beta_i\right\},
\end{eqnarray*}
and
\begin{eqnarray*}
 \mathcal{K}_{n,2}(i)&:=&\left\{1\le   k \le   d: \psi_k(n)\beta_k^{-n}\ge   \psi_i(n)\beta_i^{-n}\right\}  \\[2ex] & = & \left\{ 1\le   k \le   d: -{\log \psi_k(n) \over n}+\log \beta_k\le   -{\log\psi_i(n) \over n}+\log \beta_i\right\}.
\end{eqnarray*}

\medskip

\noindent Thus,  given $\rho>0$ and on choosing $N$  sufficiently large so that $2\psi_i(n)\beta_i^{-n}<\rho$  for any $n \ge N$, it  follows from the definition of $s$-dimensional Hausdorff measure that for any $s>0$
\begin{eqnarray} \label{mnbv}
	\mathcal{H}_\rho^s(W(T,\Psi,\va))  \!\! & \le    & \!\!
	\sum_{n=N}^\infty    \    \#  \cB_{n,i}  \ \left(2\psi_i(n)\beta_i^{-n}\right)^s   \nonumber \\[2ex]
& \le   &
	\sum_{n=N}^\infty  2^d \cdot  {\prod_{j=1}^d  \frac{\beta_j^{n+1}}{\beta_j-1}}  \prod_{k\in \mathcal{K}_{n,1}(i)}\frac{(n+3)\beta_k^{-n}}{2 \psi_i(n)\beta_i^{-n}}     \prod_{k\in  \mathcal{K}_{n,2}(i)}\frac{\psi_k(n)\beta_k^{-n}}{\psi_i(n)\beta_i^{-n}} \cdot \left(2\psi_i(n)\beta_i^{-n}\right)^s   \nonumber \\[2ex]
%	&\le&\sum_{n=N}^\infty 2^s\frac{\beta_1^{n+1}}{\beta_1-1}\cdots \frac{\beta_d^{n+1}}{\beta_d-1}\prod_{k:\psi_k(n)\beta_k^{-n}\ge   \psi_i(n)\beta_i^{-n}}\frac{\psi_k(n)\beta_k^{-n}}{\psi_i(n)\beta_i^{-n}}\left(\psi_i(n)\beta_i^{-n}\right)^sM_{i,n}^{-1}\\[2ex]
	&=&C\sum_{n=N}^\infty \exp\left\{ -n \cdot \ell_n \right\},
\end{eqnarray}
\\
where $C:=2^{s+d} \prod_{j=1}^d  \frac{\beta_j}{\beta_j-1} \ $  is a constant  and
\begin{align*}
\ell_n=\ell_n(i) \ := \   -\sum_{j=1}^{d}\log \beta_j \ \   \!\!\ &- \sum\limits_{k\in  \mathcal{K}_{n,1}(i)} \left({\log(n+3) \over n}-\log \beta_k -{\log\psi_i(n) \over n} + \log \beta_i\right) \\[2ex]
& -\sum\limits_{k\in  \mathcal{K}_{n,2}(i)}\left(\frac{\log\psi_k(n)}{n}-\log\beta_k-\frac{\log\psi_i(n)}{n}+\log\beta_i\right)\\[2ex]
&  + \  \ s\left(-\frac{\log\psi_i(n)}{n}+\log \beta_i\right).
\end{align*}
\\
Now note that $\sum_{n=1}^\infty \exp\left\{ -n \cdot \ell_n \right\}$ converges as long as
$$\limsup_{n\to \infty}\ell_n>0 \, , $$	
and that this is equivalent to the condition
that $s$ is strictly larger than the upper limit of

\begin{eqnarray*}
	h_n \!\! &=& \!\! h_n(i) \ := \ \frac{\sum\limits_{j=1}^{d}\log\beta_j+\sum\limits_{k\in  \mathcal{K}_{n,1}(i)} \left({\log(n+3) \over n}-\log \beta_k -{\log\psi_i(n) \over n} + \log \beta_i\right)}{-\frac{\log\psi_i(n)}{n}+\log \beta_i}\\[4ex]
	& &\ \ \ \ \ \ \ \  \ \ \ \qquad + \ \ \  \frac{\sum\limits_{k\in  \mathcal{K}_{n,2}(i)}\left(\frac{\log\psi_k(n)}{n}-\log\beta_k-\frac{\log\psi_i(n)}{n}+\log\beta_i\right)}{-\frac{\log\psi_i(n)}{n}+\log \beta_i} \ \\[4ex]
	& =& \!\! \sum_{k\in \mathcal{K}_{n,1}(i)} 1+ \sum_{k\in \mathcal{K}_{n,2}(i)} \left(1- { -\frac{\log\psi_k(n)}{n}\over -\frac{\log\psi_i(n)}{n}+\log \beta_i}\right) +\sum_{k\in \mathcal{K}_{n,3}(i)} {\log\beta_k  \over -\frac{\log\psi_i(n)}{n}+\log \beta_i},
\end{eqnarray*}
%Denote by $\mathcal{K}_{n,1}(i), \ \mathcal{K}_{n,2}(i),  \ \mathcal{K}_{n,3}(i) $  the partition of $\{1, \dots, d\}$ given by
where
\begin{eqnarray*}
 \mathcal{K}_{n,3}(i)  \!\!\! &:= &  \!\!\! \{1, \dots, d\} \setminus (\mathcal{K}_{n,1}(i) \cup \mathcal{K}_{n,2}(i))\\[2ex]
 &= & \!\!\! \left\{ 1\le   k \le   d:  -{ \log(n+3) \over n}+\log \beta_k\le   -{\log\psi_i(n) \over n}+\log \beta_i < -{ \log \psi_k(n) \over n}+\log \beta_k\right\} \, .
\end{eqnarray*}
%It is easily seen that   $\mathcal{K}_{n,1}(i)$, $\mathcal{K}_{n,2}(i)$, $\mathcal{K}_{n,3}(i)$ is a  partition of $\{1, \dots, d\}$.
So, by the additional assumption  imposed  in the lemma, it follows that
\[
\limsup_{n\to \infty} h_n= \lim_{n\to\infty} h_n= \sum_{k\in \mathcal{K}_1(i)}1+\sum_{k\in \mathcal{K}_2(i)}\left(1-\frac{t_k}{\log\beta_i+t_i}\right) +\sum_{k\in \mathcal{K}_3(i)}\frac{\log\beta_k}{\log\beta_i+t_i}=\theta_i(\mathbf{t}).
\]
\\
The upshot of the above is that for any $1\le   i \le   d$ and $ s>\theta_i(\mathbf{t})$,   we have that
$$
\sum_{n=1}^\infty \exp\left\{ -n \cdot \ell_n \right\} < \infty
$$
 and hence together with \eqref{mnbv} we obtain that
 $$  0 \le \cH^s(W(T,\Psi,\va)) = \lim\limits_{\rho\to0}\mathcal{H}_\rho^s(W(T,\Psi,\va))  \le \lim\limits_{N\to\infty} C\sum_{n=N}^\infty \exp\left\{ -n \cdot \ell_n \right\}  =0 \, . $$
In turn, it follows from the definition of Hausdorff dimension that $
\dim_{\rm H} W(T,\Psi,\va)\le\theta_i(\mathbf{t})  \, .
$
This upper bound estimate is true for any $1\le   i \le   d$,  and so  it implies that
\[
\dim_{\rm H} W(T,\Psi,\va)\le \min_{1\le   i\le   d}\theta_i(\mathbf{t})
\]	
as desired.
\end{proof}

%\begin{proof}[Proof of Lemma \ref{limitinfinity}]
%Applying Lemma \ref{limitinfinity-m}, we obtain Lemma \ref{limitinfinity}.
%\end{proof}

Armed with Lemma~\ref{limitexist}, it is relatively straightforward to prove the general upper bound statement for  the Hausdorff dimension of  $W(T,\Psi,\va)$.

	\begin{proof}[Proof of Proposition~\ref{upperbound}]
To prove the proposition, we first cover the accumulation set $\UP$. For any $\varepsilon>0$, since $\UP$ is bounded, we can find a family $\mathcal{B}_\epsilon$ of finitely many balls of the  form
		$$B=\prod_{i=1}^d[b^{(i)}_B, b^{(i)}_B+\varepsilon]   \qquad  (b^{(i)}_B\ge   0)$$
		 that  cover $\UP$. For $B \in \mathcal{B}_\epsilon$, let
		 $$\mathcal{N}(B)=\left\{n\in\mathbb{N}: \left(\frac{-\log\psi_1(n)}{n},\cdots,
		 \frac{-\log\psi_d(n)}{n}\right)  \ \in \ \prod_{i=1}^d[b^{(i)}_B, b^{(i)}_B+\varepsilon]\right\}.$$   Without loss of generality, we assume that $\#\mathcal{N}(B)=\infty$
 since, otherwise, there is no accumulation point in  the ball $B$.
 We claim that $W(T,\Psi,\va)$ is a subset of
		 $$\bigcup_{B\in\mathcal{B}}\Big\{\vx\in\mathbb{T}^d:  \|T_{\beta_i}^nx_i-a_i\| \le   e^{-nb^{(i)}_B}  \ \ (1\le   i\le   d)\  \ \text{for infinitely many}\ n\in \mathcal{N}(B)\Big\}.$$
		 Indeed, for any $\vx\in W(T,\Psi,\va)$, there exists a sequence $\{n_j\}_{j \in \N}$ depending on $\vx$ such that for any $1\le   i\le   d$
		 $${\|}T_{\beta_i}^{n_j}x_i-a_i{\|} \le \psi_i(n_j) \quad  \ \forall
  \ \  j\ge   1.$$
		 Since there are only finitely many balls $B\in \mathcal{B}_\epsilon$ which cover $\UP$, there exists some $B\in\mathcal{B}_\epsilon$ that contains infinitely many points of $$\Big\{ \big(\frac{-\log\psi_1(n_j)}{n_j},\cdots,
		 \frac{-\log\psi_d(n_j)}{n_j}\big)  \Big\}_{j \in \N}  \, .  $$ Thus, for these infinitely many $j$'s, we have  that for any $1\le   i\le   d$
		 \[
		 {\|}T_{\beta_i}^{n_j}x_i-a_i{\|} \, \le \,   \psi_i(n_j)   \ \le   \ e^{-n_jb^{(i)}_B}  \, .
		 \]
This establishes the claim and  by the countable stability property of Hausdorff dimension, it follows that  $\dim_{\rm H} W(T,\Psi,\va)$ is less than or equal to
		 \[
		 \max_{B\in\mathcal{B}_\epsilon}  \  \dim_{\rm H} \Big\{\vx\in\mathbb{T}^d:  {\|}T_{\beta_i}^nx_i-a_i{\|} \le   e^{-nb^{(i)}_B} \  (1\le   i \le   d) \ \ \text{for infinitely many}\ n\in\mathbb{N}\Big\}.
		 \]
	\\
Now observe that
\[
\lim_{n\to \infty}\frac{-\log e^{-nb_{B}^{(i)}}}{n} \quad  {\rm for \ all} \ \   1\le    i \le   d \, ,
\]

\noindent
and so on
%$$\left\{\left(-\frac{\log e^{-nb_{B}^{(1)}}}{n},\cdots,
%	-\frac{e^{-nb_{B}^{(d)}}}{n}\right)\right\}$ exists	and equals $(b_{B}^{(1)},\dots,b_{B}^{(d)})$$.
		applying Lemma~\ref{limitexist} we deduce that
\[
\dim_{\rm H} W(T,\Psi,\va)\le \max_{B\in\mathcal{B}_\epsilon} \min_{1\le   i\le   d}\theta_i\big((b^{(1)}_{B}, \dots, b^{(d)}_{B})\big).
\]
Then on letting  $\varepsilon\to 0$, by the continuity of $\theta_i(\mathbf{t})$ with respect to $\mathbf{t}$, we  conclude that
\[
\dim_{\rm H} W(T,\Psi,\va)\le \sup_{\mathbf{t}\in\UP}\min_{1\le   i\le   d}\{\theta_i(\mathbf{t})\}.
\]
This completes the proof of Proposition~\ref{upperbound}.
\end{proof}

We now turn out attention to establishing the  lower bound for the Hausdorff dimension of $ W(T,\Psi, \va) $.
%The key lies in setting up   appropriate  statements in order to apply the `rectangle to rectangle'  Mass Transference Principle  of  Wang $\&$ Wu; namely, Theorem~\ref{MTPRG}.
\begin{proposition}\label{lowerbound}
Under the setting of Theorem \ref{rectangledimresult}, we have that
	%	$$\dim_{\rm H} W(T,\Psi, \va)\ge   \sup_{\mathbf{t}\in\UP} \min_{1\le   i\le   d}\{\theta_i(\mathbf{t})\}.$$
$$\dim_{\rm H} W(T,\Psi, \va)\ge   \sup_{\mathbf{t}\in\UP}\min_{1\le   i\le   d}\{\theta_i(\mathbf{t})\}.$$
%$$\dim_{\rm H} W(T,\Psi, \va)\ge  \sup_{\mathbf{t}\in\UP}\min \big\{\min_{i\in \mathcal{L}_1}\{{\theta}_i(\mathbf{t})\}, \ \min_{i\in \mathcal{L}_2}\{\underline{\theta}_i(\mathbf{t})\}\big\},$$
%where
%\[
%\underline{\theta}_i(\mathbf{t})= \sum_{k\in \mathcal{L}_1}1+\sum_{k\in \mathcal{L}_2}\left(1-\limsup_{n\to\infty}\frac{\log\psi_k(n)}{\log\psi_i(n)}\right).
%\]
\end{proposition}

%\red{?? bounded case ?? }

\begin{proof}	
 The proof of Proposition~\ref{lowerbound} relies on constructing a  suitable  $\limsup$ type subset of $W(T,\Psi, \va)$ which enables us to exploit the  `rectangles to rectangles' Mass Transference Principle (Theorem \ref{MTPRG}). With this in mind, by \eqref{limsupset} and \eqref{preimages}, we know that $W(T,\Psi, \va)$ is a limsup set of rectangles with sides given by the intervals $I_{n,j}^{(i)}$ ($1\le   i \le   d$).   Recall, that the sides correspond to  the intersection of $T_{\beta_i}^{-n} \big( B(a_i, \psi_i(n))\big)$ and a cylinder of order $n$  and  as in the proof of Proposition~\ref{upperbound}, we do not distinguish between  $\beta$-transformations acting on the unit interval or the torus. Thus,  if for each $1\le   i \le   d$, we select only those intervals which are intersections of $T_{\beta_i}^{-n} \big( B(a_i, \psi_i(n))\big)$ and full cylinders of order $n$, we will obtain a $\limsup$ type subset of $W(T,\Psi, \va)$; that is   % the full cylinders which intersect the preimages of balls $B(a_i, \psi(n))$ under $T_{\beta_i}^{n}$.
% For each $1\le   i \le   d$,  let $M_i$ be the number of such full cylinders of order $n$ for $T_{\beta_i}$ and $\{x_{n,j_i}^i, 1\le   j_i\le   M_i\}$ be the preimages of $a_i$ under $T_{\beta_i}^n$ on these full cylinders. Then %we have found a limsup type subset of our aiming set $W(T,\psi, \va)$:
 \begin{equation}\label{limsupsubset}
 W(T,\Psi,\va)\supset\bigcap_{N=1}^{\infty}\bigcup_{n=N}^{\infty}\bigcup_{j_1=1}^{M_{1,n}} \cdots\bigcup_{j_d=1}^{M_{d,n}} B(x_{n,j_1}^{{(1)}},\beta_1^{-n}\psi_1(n))\times\cdots\times B(x_{n,j_d}^{(d)},\beta_d^{-n}\psi_d(n)) \, ,
 \end{equation}
 where $\{x_{n,j_i}^{(i)}, 1\le   j_i\le   M_{i,n}\}$ are the preimages of $a_i$ under $T_{\beta_i}^n$ that fall within full cylinders  of order $n$ for $T_{\beta_i}$ and  $M_{i,n}$ is the number of such full cylinders.
Now with \eqref{limsupsubset} and Fact BW in mind,  it follows that for each $1\le   i \le   d$ the enlarged
  collection of balls or rather intervals
 $\big\{B(x_{n,j_i}^{(i)}, (n+3)\beta_i^{-n}): 1\le   j_i\le   M_{i,n}\big\}$
  covers {$\T$},
%  \del{ the unit interval $[0, 1)$} \sv{[and elsewhere!]}
   that is
 \begin{equation}  \label{meadded}
 \T  \ = \  \bigcup_{j_i=1}^{M_{i,n}} B\big(x_{n, j_i}^{(i)}, (n+3)\beta_i^{-n}\big) \, .
 \end{equation}
 %and we can enlarge the radius $\beta_i^{-n}\psi(n)$ to $(n+1)\beta_i^{-n}$ by the exponent $\tau_i\ge   \frac{\log ((n+1)\beta_i^{-n})}{\log (\beta_i^{-n}\psi(n))}$.

 Now fix a point $\mathbf{t}=(t_1,\dots,t_d)\in\UP$.  Then by definition and the fact that $\UP$ is bounded,
 there exists a subsequence $\{n_l\}_{l\in \N} $ such that
  \[
\lim\limits_{l\to\infty}\frac{-\log\psi_i(n_l)}{n_l}=t_i   \quad  {\rm for \ all} \ \   1\le    i \le   d.
\]
  It is easily verified that for any $0<\varepsilon<1$,   there exists $N=N(\varepsilon)>0$ such that
 \begin{equation}\label{enlargecondition}
 \frac{(1-\varepsilon)\log\beta_i}{(1-\varepsilon)\log\beta_i+t_i} \ \le   \ \frac{-\frac{\log(n_l+3)}{n_l}+\log\beta_i}{\log\beta_i+\frac{-\log\psi_i(n_l)}{n_l}}
 \end{equation}
  for all $l\ge   N$ and $1\le   i\le   d$.
 Let
 $$s_i:=\frac{(1-\varepsilon)\log\beta_i}{(1-\varepsilon)\log\beta_i+t_i}\ \quad  (1\le   i\le   d).$$
 %We remark that $s_i=0$ if $t_i=+\infty$.
 Then,  \eqref{enlargecondition} is equivalent to
 $$\big(\beta_i^{-n_l}\psi_i(n_l)\big)^{s_i}\ge   (n_l+3)\beta_i^{-n_l},$$
 which together with \eqref{meadded}  implies that for any  $l\ge   N$ $$\T =\bigcup_{j_i=1}^{M_{i,n_l}} B\big(x_{n_l, j_i}^{(i)}, \big(\beta_i^{-n_l}\psi_i(n_l)\big)^{s_i}\big).$$
 In turn,  it follows that for any  $l\ge   N$
 $$\T^d  \ = \  \bigcup_{j_1=1}^{M_{1,n_l}} \cdots\bigcup_{j_d=1}^{M_{d,n_l}} B\big(x_{n_l,j_1}^{(1)},(\beta_1^{-n_l}\psi_1(n_l))^{s_1}\big)\times\cdots\times B\big(x_{n_l,j_d}^{(d)},(\beta_d^{-n_l}\psi_d(n_l))^{s_d}\big)$$
 and so
 \begin{equation}\label{fulllimsup}
 \T^d   \ = \ \limsup_{n\to\infty}  \bigcup_{j_1=1}^{M_{1,n}} \cdots\bigcup_{j_d=1}^{M_{d,n}} B\big(x_{n,j_1}^{(1)},(\beta_1^{-n}\psi_1(n))^{s_1}\big)\times\cdots\times B\big(x_{n,j_d}^{(d)},(\beta_d^{-n}\psi_d(n))^{s_d}\big).
 \end{equation}

 The upshot is that given the $\limsup$ set of rectangles appearing on the right hand side of \eqref{limsupsubset}, the corresponding  $\limsup$ set of `$(s_1, \ldots, s_d)$-scaled up' rectangles satisfies
 %\red{??Need working on??} This verifies
  \eqref{fullmeasure} with $p=d$,  $X_i = \T$, $\delta_i =1 $ and $\mu_i=m_1$ (one-dimensional Lebesgue measure)  for each $  1 \le i \le d$.
Thus on applying Theorem \ref{MTPRG} with $u_i=(1-\varepsilon)\log\beta_i$ and $v_i=  (1-\varepsilon)\log\beta_i + t_i$ ($1 \le i \le d$),
%    \sv{[what do put $t_i$  equal to or rather $v_i$ if we make the change in Theorem \ref{MTPRG}]}
%
we obtain the lower bound  %$s_0$ defined as
 $$\dim_{\rm H} W(T,\Psi, \va)\ge   \min_{1\le   i\le   d} s(i,\varepsilon)$$
where
\[
s(i,\varepsilon):=\sum_{k\in\mathcal{K}_1(i,\varepsilon)}1+\sum_{k\in\mathcal{K}_2(i,\varepsilon)}\left(1-\frac{t_k}{(1-\varepsilon)\log\beta_i+t_i}\right)+\sum_{k\in\mathcal{K}_3(i,\varepsilon)}\frac{(1-\varepsilon)\log\beta_k}{(1-\varepsilon)\log\beta_i+t_i}
\]
and where $\mathcal{K}_1(i,\varepsilon)$, $\mathcal{K}_2(i,\varepsilon)$, $\mathcal{K}_3(i,\varepsilon)$ is the partition of $\{1, \dots, d\}$ given by
\begin{eqnarray*}
\mathcal{K}_1(i,\varepsilon)& := & \big\{k: (1-\varepsilon)\log\beta_k \ge   (1-\varepsilon)\log\beta_i+t_i\big\}   \\[2ex] \mathcal{K}_2(i,\varepsilon)& := & \big\{k: (1-\varepsilon)\log\beta_k+t_k \le   (1-\varepsilon)\log\beta_i+t_i\big\}, \\[2ex]  \mathcal{K}_3(i,\varepsilon)& := & \big\{1, \dots, d\big\}\setminus \big(\mathcal{K}_1(i,\varepsilon) \cup \mathcal{K}_2(i,\varepsilon)\big).
 \end{eqnarray*}

\noindent Fix $1\le   i\le   d$. On  letting  $\varepsilon\to 0$, we find that  $\mathcal{K}_1(i,\varepsilon)\to \{k: \log\beta_k > \log\beta_i+t_i\}=\mathcal{K}_1(i),$ $\mathcal{K}_2(i,\varepsilon)\to\mathcal{K}_2(i)$  and $\mathcal{K}_3(i,\varepsilon)\to\mathcal{K}_3(i)$.
Thus $$\lim\limits_{\varepsilon\to 0} s(i,\varepsilon)= \theta_i(\mathbf{t})  \quad   {\rm  and }  \quad  \dim_{\rm H} W(T,\Psi, \va)\ge   \min\limits_{1\le   i\le   d} \theta_i(\mathbf{t}).$$
Moreover, since this is valid for any
 $\mathbf{t} \in \UP$ it follows that 	$$\dim_{\rm H} W(T,\Psi, \va)\ge   \sup_{\mathbf{t}\in\UP} \min_{1\le   i\le   d}\{\theta_i(\mathbf{t})\} $$
and we are done.
\end{proof}

\subsubsection{Proof of Theorem~\ref{dimresultinteger} \label{iwill}}

We show that when $T$ is an integer matrix transformation, the diagonal assumption in Theorem~\ref{dimresult}   can be relaxed to $T$ is diagonalizable over $\Z$.  This thereby proves Theorem~\ref{dimresultinteger}.
So, suppose $T$ is diagonalizable over $\mathbb{Z}$.  Then by definition, there exist a diagonal integer matrix $D$   and  an invertible mapping $\phi$ satisfying   \eqref{nifty}.     It is easily versified that $T^n(\vx)\in B(\va,\psi(n))$ if and only if $D^n(\phi(\vx))\in \phi\big(B(\va,\psi(n))\big)$. Since $\phi$ is a bi-Lipschitz map, we can find two positive constants $ 0 < c_1 \le c_2  <  \infty $ such that
	$$
	B\big(\phi(\va), c_1\psi(n)\big) \subset \phi\big(B(\va,\psi(n)\big) \subset B\big(\phi(\va), c_2\psi(n)\big).
	$$
In turn, Lemma~\ref{lip} implies that the Hausdorff dimension of
	$$W(T,\psi, \va):=\{\vx\in\mathbb{T}^d: T^n(\vx)\in B(\va,\psi(n))\ \ \text{for infinitely many}\ n\in\mathbb{N}\}$$
	is the same as that of
	$$\big\{x\in\mathbb{T}^d: D^n(\vx)\in B(\phi(\va),\psi(n))\ \ \text{for infinitely many}\ n\in\mathbb{N}\big\}.$$
	Thus, without loss of generality, we only need to prove the desired dimension result in the case that $T$ is diagonal.

\medskip

%
%\begin{remark} \label{iwilli}
%Of course the above argument can be easily modified to rectangular target sets.  Thus,  if  $T$ is  an integer matrix transformation, then  we can replace the condition that $T$ is diagonal in Theorem~\ref{rectangledimresult}  by $T$ is diagonalizable over $\Z$.
%\end{remark}

%%%%%%%%%%%%%%%%

%%%%%%%%%%%%%%%%%%%%%

\subsection{Proof of Theorem \ref{dimensionalone}  \label{poiu} }

The proof of Theorem~\ref{dimensionalone} will make use of a general statement (namely, Proposition~\ref{Markov} below)  concerning Markov subsystems which may be of independent interest. In short,  these systems provide a ``nice'' approximation to one-dimensional piecewise linear dynamical systems.   To start with, let us recall the notion of  a Markov system for a one-dimensional expanding dynamical system  $(X,T)$.
%\subsection{lower bound for one-dimensional case}
With this in mind, let $X$ be a compact set in $\mathbb{R}$ and $T: X\to X$ be an expanding map. Furthermore, let $\Lambda $ be a subset of $X$.  A partition  $ \cP_\Lambda$ of $\Lambda$ into finite or countable collection of sets $P(k)$ is called a \emph{Markov partition} if $\Lambda:=\bigcap_{n=0}^\infty T^{-n}\big(\cup P(k)\big)\,$  and
\begin{itemize}
	\item[(i)]  the  interior of $P(j)$ and $P(k)$ are disjoint if $j\neq k$, % the set $X\setminus \bigcup_k P(k)  $ is of Lebesgue measure zero, and
	\item[(ii)] $T$ restricted on each $P(j)$ is one to one,
	\item[ (iii)] if  $T(P(j))$ intersects the interior of $P(k)$  for some $j$ and $k$ then
	$P(k)\subseteq  \overline{T(P(j))}.$
\end{itemize}
In turn, the system $(\Lambda, T|_\Lambda , \mathcal{P}_\Lambda)$ is called a  \emph{Markov subsystem of $(X,T)$}.  In the case $ \Lambda = X$, we simply write $(X, T, \mathcal{P})$  and  referred to it  as  a \emph{Markov system}.

An important property regarding Markov subsystems  that we shall utilise is given by  the following statement.   It is a direct consequence of \cite[Theorems 4.2.9 $\&$ 4.2.11]{MU03}.

\begin{proposition}\label{MarkovAlfhor}
	Let $X$ be a compact set in $\mathbb{R}$ and $T: X\to X$ be an expanding map.  Let  $(\Lambda, T|_\Lambda , \mathcal{P}_\Lambda)$ be a  Markov subsystem of $(X,T)$ with finite partition $\mathcal{P}_\Lambda=\{P(i)\}_{1\le   i\le   N}$ whose incidence matrix is primitive. Suppose that for any $1\le   k\le   N$, $T|_{P(k)}$ is $C^{1+\alpha}$ for some $\alpha>0$.  Then the measure   $\cH^{\delta}|_{\Lambda}$ is  $\delta$-Ahlfors regular where  $\delta:=\dim_{\rm H}\Lambda$.
\end{proposition}

The following statement provides a lower bound for $\dim_{\rm H}\Lambda$  in the case $T$ is  piecewise linear. Throughout,  we suppose that   the absolute value of the slope of such a  map $T$ is constant and will be  denoted by $\slT$.
\medskip

%We use one dimensional Markov systems to approximate the $\beta$-transformation dynamical system $([0, 1], T_\beta)$ with $|\beta|>1$.

\begin{proposition}\label{Markov}
	Let $T$ be a piecewise linear map on $[0, 1]$ and assume that $\slT>8$. Then there exists a Markov subsystem $(\Lambda, T|_{\Lambda}, \cP_{\Lambda} )$ of $([0,1], T)$ with a finite partition $\mathcal{P}_\Lambda=\{P(i)\}_{1\le   i\le   m}$  where each $P(i)$ is an interval and $T|_{P(i)}$ is linear,
such that  $$\dim_{\rm H}\Lambda\ge1-\frac{\log 8}{\log\slT}.$$
\end{proposition}
\begin{proof}
	Let $\tilde{\mathcal{P}}=\{\tilde{P}(i)\}_{i=1}^m$  be a partition of $[0,1]$ such that for each  $1\le  i\le  m$ the set $\tilde{P}(i)$ is an interval and $T|_{\tilde{P}(i)}$ is linear. Without loss of generality, we can assume that $$\max\{|\tilde{P}(i)|: \tilde{P}(i)\in  \tilde{\mathcal{P}}\}\le  2\kappa   \quad  {\rm with \ } \quad   \kappa:=\min\{|\tilde{P}(i)|: \tilde{P}(i)\in\tilde{\mathcal{P}}\}  \, . $$ Indeed, if $|\tilde{P}(i)|> 2\kappa$ for some $1\le  i\le  m$, then there exists  $\ell\in\mathbb{N}$ such that
	$$2^\ell\kappa<|\tilde{P}(i)|\le  2^{\ell+1}\kappa.$$
	Hence, we can subdivide $\tilde{P}(i)$ into $2^{\ell}$ equal pieces  and take these subintervals as part of partition rather than $\tilde{P}(i)$. The map $T$ restricted to each piece of the new partition is  still linear and by construction the length of each piece if bounded above by $2 \kappa$.
	
	For any interval $\tilde{P}\in\tilde{\mathcal{P}}$, let
	$$P:=\tilde{P}\cap T^{-1}\Big( \ \overline{\bigcup_{1\le i \le m}\{\tilde{P}(i)\in\tilde{\mathcal{P}}: \tilde{P}(i)\subset T(\tilde{P})\}} \ \Big).$$
	Now since  $T|_{\tilde{P}(i)}$ is linear, the intervals $\tilde{P}(i)$ contained in $T(\tilde{P})$ are adjacent intervals in the partition $\tilde{\mathcal{P}}$. Hence, $P$ is a subinterval of $\tilde{P}$.
	Furthermore, since $T|_{\tilde{P}}$ is linear with slope   $\pm\slT$ we have that  $|T(\tilde{P})|= \slT\cdot |\tilde{P}|\ge\slT\kappa$. So the number of $\tilde{P}(i) \in \tilde{\mathcal{P}}$ that intersect $T(\tilde{P})$ is at least the integer part of $\slT\kappa/2\kappa=[\slT/2]$. Here we use the fact that $|\tilde{P}(i)|\le  2\kappa$ for all intervals in the partition. Thus, on using the fact that $\slT>8$, we have that
\begin{equation}\label{needed}
\# \big\{\tilde{P}(i)\in\tilde{\mathcal{P}}: \tilde{P}(i)\subset T(\tilde{P})\big\}    \ge [\slT/2]-2   \ge 1   \, .
\end{equation}
The upshot of this is that  $$P\neq\emptyset \, .$$
	We now prove that $\mathcal{P}:=\{P(i): \tilde{P}(i)\in\tilde{\mathcal{P}}, \ 1\le  i\le  m\}$ is a Markov partition of $\bigcup_{i=1}^mP(i)$.   The first two conditions are automatically satisfied.   Regarding the third condition,  for any $1\le  j,k\le  m$ with $T(P(j))\cap P(k)\neq\emptyset$,  we first note that
	$T(P(j))\cap \tilde{P}(k)\neq\emptyset$ which in turn implies that $P(j)\cap T^{-1}(\tilde{P}(k))\neq \emptyset$. Then, by the definition of $P(j)$, $\tilde{P}(k)$ is an interval such that $\tilde{P}(k)\subset T(\tilde{P}(j))$ %Thus by the
	%by the definition of $P(j)^*$, we know that $P(k)\subset T(P(j))$
	and $\tilde{P}(j)\cap T^{-1}(\overline{\tilde{P}(k)})\subset P(j)$. So $\tilde{P}(k)\subset T(P(j))$. Therefore, by noting that
	$P(k)\subset \tilde{P}(k)$, we have $P(k)\subset T(P(j))  \subset \overline{T(P(j))} $ and this verifies the third condition.
	
	Next, let
$$f:=T|_{\cup_{i=1}^mP(i)}   \quad {\rm  and } \quad \Lambda:=\bigcap_{n=0}^\infty f^{-n}\big(\cup_{i=1}^mP(i)\big)\, . $$
Then, by construction, the system $(\Lambda, T|_{\Lambda}, \cP_{\Lambda} )$  with $\cP_{\Lambda} :=\{P(i): \tilde{P}(i)\in\mathcal{P}, \ 1\le  i\le  m\}$ is a Markov subsystem of $([0,1], T)$.  It remains to  prove that the Hausdorff dimension of the set $\Lambda$ satisfies the lower bound in the statement of the proposition. For  this,  we work in the symbolic space of the dynamical system under consideration  to  estimate the  topological entropy of $T|_\Lambda$ and then use the fact that the entropy is intimately related to the dimension of $\Lambda$.

The dynamics of $T|_{\Lambda}$ can be coded by the $m\times m$ transition matrix $A=(A_{jk})_{1\le  j,k\le  m}$ with entries
	\begin{eqnarray*}
		A_{jk}=\begin{cases}
			1\ \ \ &\text{if}\ P(k)\subset \overline{T(P(j))},\\[1ex]
			0\ \ \ &\text{otherwise.}
		\end{cases}
	\end{eqnarray*}
	Denote by $\Sigma_A^{\mathbb{N}}\subset\{1,2,\dots,m\}^{\mathbb{N}}$ the corresponding symbolic space induced by $A$ and $\Sigma_A^n$ the set of words of length $n$ in $\Sigma_A^{\mathbb{N}}$. The projection $\pi$ from $\Sigma_A^{\mathbb{N}}$ to $\Lambda$ is given by
	$$  \omega=(\omega_n)_{n\ge0}  \  \ \mapsto \ \
\pi(\omega)=\bigcap_{n=0}^{\infty}f^{-n}(P(\omega_n))  \, .$$

In view of \eqref{needed}, it follows that any given  word of length  $n$ gives rise to at least $\big[\frac{\slT}{2}\big]-2$  words of length  $(n+1)$.  Hence, we have that  $$\#\Sigma_A^n \, \ge \,  m\left(\frac{\slT}{2}-3\right)^{n-1},$$
	which implies that the topological entropy $h_{\rm top}(T|_\Lambda)$ of $T|_\Lambda$ is at least $\log\left(\frac{\slT}{2}-3\right)$.  This together with Bowen's definition of topological entropy  (see  \cite{B1973}, \cite[page 230]{FFW01}) and the fact that the  absolute value of the slope of $T|_\Lambda$ is a constant (namely $\slT > 8$), implies  that
	$$\dim_{\rm H}\Lambda  =    \frac{h_{\rm top}(T|_\Lambda)     } {\log \slT} \ge \frac{\log\left(\frac{\slT}{2}-3\right)}{\log \slT} \ge 1-\frac{\log 8}{\log\slT}  \ .$$
\end{proof}

%\red{Is from below  just for $\beta$-trans or piecewise linear map}
%
%\begin{proposition}\label{Markovlowerboundsv}
%Let $T$ be a piecewise linear map on $[0, 1]$ with constant slope $\beta(T)$.  Assume that $\slT>8$. Then there exists a Markov subsystem $(\Lambda, T|_{\Lambda}, \cP_{\Lambda} )$ such that for any
%  real positive decreasing function $\psi: \mathbb{R}^+\to \mathbb{R}^+$ and $a\in \red{K_{???}}$ $$\dim_{\rm H}W(T_{\Lambda},\psi,a)\ge  \frac{\log\slT - \log 2 }{\log\slT+\lambda}   .$$
%\end{proposition}
%
%\red{OR}

The following result provides a lower bound for the Hausdorff dimension of shrinking target sets associated with  piecewise linear maps.

\begin{proposition}\label{Markovlowerbound}
	%  Let $\beta$ be a real number with $|\beta|>1$ and $T_\beta$ be the corresponding beta-transformation. Let $K(\beta)$ be the support of the associated Parry-Yrrap measure of $T_\beta$. Let $\psi: \mathbb{R}^+\to \mathbb{R}^+$ be a real positive decreasing function and  $a \in K(\beta)$.
	%
	Let $T$ be a piecewise linear map on $[0, 1]$ and assume that $ \slT>8$. Let  $(\Lambda, T|_{\Lambda},\cP_{\Lambda})$ be the associated Markov subsystem arising  from Proposition \ref{Markov}.  Suppose there exists a compact set $K   \supseteq \Lambda $ and an integer $k_0>0$, so that $T^{k_0}(P) \supseteq K$ for any interval $P\in \mathcal{P}_\Lambda$.
	Let $\psi: \mathbb{R}^+\to \mathbb{R}^+$ be a real positive function and $a\in K$. Then  $$\dim_{\rm H}W(T,\psi,a)\ge  \frac{1-{\log 8}/{\log\slT}}{1+\lambda/\log\slT}, $$
where $\lambda=\lambda(\psi) $ is the lower order at infinity of the function $\psi$ and
$$
W(T,\psi, a):=\{x\in [0,1]    : |T^nx-a| \le \psi(n)  \ \ \text{for infinitely many}\ n\in\mathbb{N}\}.
$$
\end{proposition}
\begin{proof}
We are given  that $(\Lambda, T|_{\Lambda}, \cP_{\Lambda} )$   is a Markov subsystem of the dynamical system $([0,1], T)$ coming from Proposition~\ref{Markov}.  Indeed, $\cP_{\Lambda} =\{P(i) \ :  \ 1\le  i\le  m\}$ where each $P(i)$ is an interval and $T|_{P(i)}$ is linear.
 As in the proof of Proposition \ref{Markov}, denote by $\Sigma_A^{\mathbb{N}}\subset\{1,2,\dots,m\}^{\mathbb{N}}$ the corresponding symbolic space of the dynamics of $T|_{\Lambda}$ induced by the transition matrix $A$ and $\Sigma_A^n$ the set of words of length $n$ in $\Sigma_A^{\mathbb{N}}$.
With this in mind, given a  word  $(i_0i_1\cdots i_{n-1})\in\Sigma_A^n$,   let
$$
P(i_0i_1\cdots i_{n-1}):=P(i_0)\cap T^{-1}\big(P(i_1) \big) \cap  \cdots \cap T^{-(n-1)}\big(P(i_{n-1}) \big)
$$
and for  each $n\in \N$, let
$$\mathcal{P}_n:=\Big\{P(i_0i_1\cdots i_{n-1}): (i_0i_1\cdots i_{n-1})\in\Sigma_A^n \Big\}$$
denote the collection of cylinder sets of length $n$.
Now by the Markov property of $\mathcal{P}_{\Lambda}$,  for any cylinder $P(i_0i_1\cdots i_{n-1})\in\mathcal{P}_n$
	\begin{equation}\label{Markovcylinder}
		T^{n-1}\big(P(i_0i_1\cdots i_{n-1}) \big)=P(i_{n-1}).
	\end{equation}
	Then with  $K$ and $k_0$ as in the statement of the proposition,  we have that $T^{n-1+k_0} \big(P(i_0i_1\cdots i_{n-1}) \big)\supseteq K$. It therefore follows that for any $a\in K$, there exists a point $x_{i_0i_1\cdots i_{n-1}}\in P(i_0i_1\cdots i_{n-1})$ such that $T^{n-1+k_0}(x_{i_0i_1\cdots i_{n-1}})=a.$ That is, we can find a preimage of the point $a$ under $T^{n+k_0-1}$ on every cylinder of order $n$.
	So for any point   $x\in B\Big(x_{i_0i_1\cdots i_{n-1}},\frac{\psi(n+k_0-1)}{\slT^{n+k_0-1}}\Big)$,   we have that
\begin{eqnarray*} |T^{n+k_0-1}(x)-a| & = & |T^{n+k_0-1}(x)-T^{n+k_0-1}(x_{i_0i_1\cdots i_{n-1}})| \\[1ex] & = & \slT^{n+k_0-1}|x-x_{i_0i_1\cdots i_{n-1}}| \\[1ex] & < & \psi(n+k_0-1).
\end{eqnarray*}
Therefore,
		\begin{equation}\label{limsupsubsetA}
			\limsup_{n\to\infty}\bigcup_{i_0i_1\cdots i_{n-1}\in \Sigma_A^{n}} B\Big(x_{i_0i_1\cdots i_{n-1}},\frac{\psi(n+k_0-1)}{\slT^{n+k_0-1}}\Big)\subset W(T,\psi,a).
		\end{equation}

On the other hand, by \eqref{Markovcylinder} we have  that
	$$\kappa_*\slT^{-(n-1)}\le  |P(i_0i_1\cdots i_{n-1})|\le  \kappa^*\slT^{-(n-1)},$$
	where $\kappa_*=\min_{1\le  i\le  m}|P(i)|$ and
	$\kappa^*=\max_{1\le  i\le  m}|P(i)|$. Hence
	\begin{equation}\label{fullset1}
			\bigcup_{i_0i_1\cdots i_{n-1}\in \Sigma_A^{n}} B\big(x_{i_0i_1\cdots i_{n-1}}, \ \kappa^*\slT^{-(n-1)}\big)\supseteq \Lambda
	\end{equation}
and so
\begin{equation}\label{fullset}
	\limsup_{n \to \infty}		\bigcup_{i_0i_1\cdots i_{n-1}\in \Sigma_A^{n}} B\big(x_{i_0i_1\cdots i_{n-1}}, \ \kappa^*\slT^{-(n-1)}\big)\supset \Lambda  \, .
	\end{equation}

Now let $\delta :=  \dim_{\rm H} \Lambda $  and note that
$$\Big(\frac{\psi(n+k_0-1)}{\slT^{n+k_0-1}}\Big)^{\frac{s}{\delta}}\ge \kappa^*\slT^{-(n-1)}$$  for any
	$$
	0< s < s_0 := \limsup_{n\to \infty}{ \delta ( \log \kappa^* - (n-1) \log \slT) \over \log \psi(n+k_0-1) -(n+k_0-1)\log \slT }=\frac{\delta}{1+\lambda/\log\slT}\, .
	$$
In other words,  for $ s   < s_0$  the radii of the `$s$-scaled up'  balls associated with \eqref{limsupsubsetA} are at least the size of  the corresponding balls appearing in \eqref{fullset}. It then follows via \eqref{limsupsubsetA},  \eqref{fullset} and
Proposition~\ref{MarkovAlfhor}, that   on  applying  the
  Mass Transference Principle  (the original Theorem \ref{MTPB}) with $\mu = \cH^{\delta}|_{\Lambda}$,   we  have that
	% $$\mathcal{H}^{\delta}\Big( \limsup_{n\to\infty}	\bigcup_{i_0i_1\cdots i_{n-1}\in \Sigma_A^{n}} B\big(x_{i_0i_1\cdots i_{n-1}}, \ \kappa^*\beta^{-(n-1)}\big)\Big)=\mathcal{H}^{\delta }(\Lambda)  \, . $$
%	  That is, enlarge the radius of balls to
%Combining \eqref{limsupsubsetA} and Theorem \ref{MTPB} gives
\begin{equation}\label{last}
 \cH^s \big( W(T,\psi,a)  \big)   =   \cH^s(\Lambda) =  \infty \, .
\end{equation}
The right hand most equality  is valid since $ s < \delta $.
Now \eqref{last} is true for any $s < s_0$ and so together with  Proposition \ref{Markov} it follows that
	$$\dim_{\rm H}W(T,\psi,a)\ge s_0 =\frac{\dim_{\rm H}\Lambda}{1+\lambda/\log\slT}\ge  \frac{1-{\log 8}/{\log\slT}}{1+\lambda/\log\slT}.$$

%*************
%
%
%
%	 In the compact space $\Lambda$ and with Proposition~\ref{MarkovAlfhor} in mind,  we can apply the
%  Mass Transference Principle  (the original Theorem \ref{MTPB})
%	 to the $\limsup$ type set of balls $B\big(x_{i_0i_1\cdots i_{n-1}},\frac{\psi(n+k_0-1)}{\beta^{n+k_0-1}}\big)$ (the left hand set of \eqref{limsupsubsetA}). Now with $\delta :=  \dim_{\rm H} \Lambda $,  the inclusion \eqref{fullset} implies  that
%	 $$\mathcal{H}^{\delta}\Big( \limsup_{n\to\infty}	\bigcup_{i_0i_1\cdots i_{n-1}\in \Sigma_A^{n}} B\big(x_{i_0i_1\cdots i_{n-1}}, \ \kappa^*|\beta|^{-(n-1)}\big)\Big)=\mathcal{H}^{\delta }(\Lambda)  \, . $$
%	  That is, enlarge the radius of balls to $\kappa^*|\beta|^{-(n-1)}$, that is, $\big(\frac{\psi(n+k_0-1)}{|\beta|^{n+k_0-1}}\big)^{\frac{s}{\delta}}\ge \kappa^*|\beta|^{-(n-1)}$, with
%	$$
%	s < s_0 := \limsup_{n\to \infty}{ \delta ( \log \kappa^* - (n-1) \log |\beta|) \over \log \psi(n+k_0-1) -(n+k_0-1)\log |\beta| }=\frac{\delta}{1+\lambda/\log\slT}.
%	$$
%Combining \eqref{limsupsubsetA} and Theorem \ref{MTPB} gives
%	$$\dim_{\rm H}W(T,\psi,a)\ge s_0 =\frac{\dim_{\rm H}\Lambda}{1+\lambda/\log\slT}\ge  \frac{1-{\log 8}/{\log\slT}}{1+\lambda/\log\slT}.$$
\end{proof}

As we shall soon see,  Proposition~\ref{Markovlowerbound} will be instrumental in the proof of Theorem~\ref{dimensionalone}.   Before moving onto the latter, we establish a technical lemma.

\begin{lemma} \label{lemma_liminf_thesame}
	Let $\psi: \mathbb{R}^+\to \mathbb{R}^+$ be a real positive decreasing function and let $\lambda=\lambda(\psi) $ be its  lower order at infinity.  Then, for any positive integer $k$  we have that
	\[
	\liminf_{n\to\infty}\frac{-\log\psi(kn)}{kn} = \lambda.
	\]
\end{lemma}
\begin{proof}
	Recall,
	$
	\lambda:=\liminf\limits_{n\to\infty}\frac{-\log\psi(n)}{n}
	$
	and thus
	there exist infinitely many indices $n\in\N$ such that
	\begin{equation} \label{lemma_liminf_thesame_ie_ub_lambda}
		\psi(n)>\exp\left(-(\lambda+\varepsilon)n\right).
	\end{equation}
	%With elementary transformations, the inequalities~\eqref{lemma_liminf_thesame_ie_preliminary_lb_lambda} and~\eqref{lemma_liminf_thesame_ie_preliminary_ub_lambda} are equivalent respectively to the following:
	%\begin{equation} \label{lemma_liminf_thesame_ie_lb_lambda}
	%\forall n>n_{\varepsilon}, \quad \psi(n)<\exp\left(-(\lambda-\varepsilon)n\right)
	%\end{equation}
	%and
	%\begin{equation} \label{lemma_liminf_thesame_ie_ub_lambda}
	%\exists \ \text{infinitely many  $n\in\N$ s.t.} \quad \psi(n)>\exp\left(-(\lambda+\varepsilon)n\right).
	%\end{equation}
	Now fix a positive integer $k\ge 2 $ and let
	\[
	\xi:=\liminf_{n\to\infty}\frac{-\log\psi(kn)}{kn}.
	\]
	Thus,  for any $\varepsilon>0$ there exists an $N_{\varepsilon}>0$ such that for every $n>N_{\varepsilon}$
	\begin{equation} \label{lemma_liminf_thesame_ie_lb_mu}
		\psi(kn)<\exp\left(-(\xi-\varepsilon)kn\right).
	\end{equation}
	%and there exists infinitely many indices $n\in\N$ such that
	%\begin{equation} \label{lemma_liminf_thesame_ie_ub_mu}
	%\psi(kn)>\exp\left(-(\xi+\varepsilon)kn\right).
	%\end{equation}straightforwardly from the definitions that, in any case, $\xi \ge\lambda$. We claim that, moreover, if $\psi$ is decreasing, then $\xi=\lambda$.
	
	By definition, we trivially have that $\xi \ge   \lambda$. We claim that if $\psi$ is decreasing then we must have equality.  With this in mind,  assume on the contrary that  $\xi  > \lambda$ and set $\varepsilon:=\frac{\xi-\lambda}{4}$. For our fixed $k\ge2$, any arbitrarily positive integer  can be written in the form  $kn+r$ with $n \in\N$ and $r\in\N$ satisfying $0\le  r\le  k-1$.
	By \eqref{lemma_liminf_thesame_ie_ub_lambda}, there is an increasing sequence $(kn_i+r_i)_{i\ge1}$ with $n_i\in \N$ and $0\le  r_i\le  k-1$ such that
	\[
	\psi(kn_i+r_i)>\exp\left(-(\lambda+\varepsilon)(kn_i+r_i)\right).
	\]
	On the other hand, for any $n\in\N$ and $0\le  r\le  k-1$ such that $kn>N_\epsilon$, by the decreasing  property of $\psi$ and \eqref{lemma_liminf_thesame_ie_lb_mu}, we have that
	%we have the following chain of inequalities (where we use~\eqref{lemma_liminf_thesame_ie_lb_mu} at the second step):
	\[
	\psi(kn+r) \le \psi(kn)<\exp\left(-(\xi-\varepsilon)kn\right)=\exp\left(-(\xi-\varepsilon)(kn+r)\right)\exp\left((\xi-\varepsilon)r\right) \, .
	\]
	Thus, for all $i$ large enough  we have that
	\[ \exp\left(-(\lambda+\varepsilon)(kn_i+r_i)\right)<\exp\left(-(\xi-\varepsilon)(kn_i+r_i)\right)\exp\left((\xi-\varepsilon)r_i\right) ,
	\]
	which in turn implies that
\begin{equation} \label{inturn} \exp\Big(\frac{\xi-\lambda}{2}(kn_i+r_i)\Big)<\exp\big((\xi-\varepsilon)r_i\big).
	\end{equation}
	Now note that with  $k$ fixed , the right-hand side of \eqref{inturn}  is bounded since $r_i$ lies in the range from $0$ to $k-1$.  However, since $\frac{\xi-\lambda}{2}>0$, the left-hand side of \eqref{inturn}  tends  to infinity as $i$ tends to infinity and we obtain a  contradiction. The upshot is that we must have $\xi=\lambda$, as claimed.
\end{proof}

\begin{proof}[Proof of Theorem~\ref{dimensionalone}]

We prove Theorem~\ref{dimensionalone} by estimating the upper and lower bounds for the Hausdorff dimension of $W(T_\beta,\psi,a) $ separately.

\medskip

The upper bound for $\dim_{\rm H} W(T,\psi, a)$  essentially follows the same line of argument as within the proof of Lemma~\ref{limitexist}  with $i=1$  and  the estimate \eqref{needlabel} replaced by \eqref{needlabelB}.  In short, for any $n \in \N$ the preimage $T_\beta^{-n} \big(B(a,\psi(n))\big)$ consists of $N_n$ intervals $\{I_{n,j} : 1\le j\le N_n  \}$ with lengths bounded by $2\psi(n)|\beta|^{-n}$ and in view of Remark~\ref{-need},  for any $\epsilon>0$ there exists $N_0\ge   1$ such that for all $n\ge   N_0$
$$N_n\le  |\beta|^{n(1+\epsilon)}.$$
Now for any $ N  \ge 1$, we have that
$$W(T_\beta,\psi,a)\subset \bigcup_{n=N}^\infty\bigcup_{j=1}^{N_n}I_{n,j} \,  . $$
Thus,  given $\rho>0$ and on choosing $N\ge   N_0$  sufficiently large so that $|\beta|^{-N}<\rho$, it follows that for any $s>0$
\begin{eqnarray*}
\mathcal{H}_\rho^s(W(T_\beta,\psi,a))  & \le  &  \sum_{n=N}^\infty\sum_{j=1}^{N_n}|I_{n,j}|^s\le  \sum_{n=N}^\infty |\beta|^{n(1+\epsilon)}(2\psi(n)|\beta|^{-n})^s  \\[2ex]
& \le & \sum_{n=N}^\infty
|\beta|^{n(1+\epsilon)-ns+s\frac{\log\psi(n)}{\log|\beta|}}   \, .
\end{eqnarray*}
	Hence, for any
	$s>\frac{(1+\epsilon)\log|\beta|}{\lambda+\log|\beta|} \, $ we have that $\mathcal{H}^s(W(T_\beta,\psi,a))= 0$ and thus  $$\dim_{\rm H} W(T_\beta,\psi, a)\le \frac{(1+\epsilon)\log|\beta|}{\lambda+\log|\beta|}.$$
	Since $\epsilon>0$ is arbitrary, we obtain the desired  upper bound for the dimension of $ W(T_\beta,\psi, a)$.

%*******************
%
%Note that $T_\beta^n$ is a piecewise linear mapping with constant slope $|\beta|$, and that the preimage $T_\beta^{-n} \big(B(a,\psi(n))\big)$ consists of intervals $\{I_{n,j}\}_{1\le j\le N_n}$ with lengths bounded by $2\psi(n)|\beta|^{-n}$. %Proposition \ref{fullcylinders} tells us that $N_n\le  C|\beta|^n$, where $C$ is a constant independent of $n$.
%It is well-known that the topological entropy of $T_\beta$ is $\log |\beta|$. Thus, for any $\epsilon>0$, there exists $N_0\ge   1$, such that for all $n\ge   N_0$,
%$$N_n\le  |\beta|^{n(1+\epsilon)}.$$	
%	So, for any $N\ge N_0$, noting that
%	$$W(T,\psi,a)\subset \bigcup_{n=N}^\infty\bigcup_{j=1}^{N_n}I_{n,j},$$
%	we have
%	$$\mathcal{H}_\delta^s(W(T,\psi,a))\le  \sum_{n=N}^\infty\sum_{j=1}^{N_n}|I_{n,j}|^s\le  \sum_{n=N}^\infty |\beta|^{n(1+\epsilon)}(2\psi(n)|\beta|^{-n})^s=\sum_{n=N}^\infty|\beta|^{n(1+\epsilon)-ns+s\frac{\log\psi(n)}{\log|\beta|}}$$
%	by choosing $N$ large enough with $|\beta|^{-N}<\delta$ for any given $\delta>0$.  Hence, for any
%	$$s>\frac{(1+\epsilon)\log|\beta|}{\lambda+\log|\beta|},$$ we have $\mathcal{H}^s(W(T,\psi,a))<+\infty$. Therefore, $$\dim_{\rm H} W(T,\psi, a)\le \frac{(1+\epsilon)\log|\beta|}{\lambda+\log|\beta|}.$$
%	Since $\epsilon>0$ is arbitrary, we obtain the upper bound.  \sv{How is this different than proof of Lemma~\ref{upperbound}??  Is Renyi only for $\beta > 1$? }
%	

To prove the complementary lower bound, we make use of  Proposition~\ref{Markovlowerbound}.   With this in mind,  for any real number  $\beta$ with $|\beta|>1$,  the transformation $T_\beta$ can be considered as a piecewise linear mapping of the unit interval  $[0,1]$  with $\beta(T)=|\beta|$.   Strictly, speaking $T_\beta$ is defined on $[0,1)$  but we can naturally include the end point one by defining  $ T_{\beta}(1)=\beta  \ (\text{mod}\ 1)$.    This extension  will not effect the dimension of $W(T_\beta,\psi,a)$ since it  introduces at most a single point.
Now choose $k\in\mathbb{N}$ large enough so that  $|\beta|^k>8$, and note that $$W(T_\beta,\psi,a)\supseteq W(T_\beta^k,\varphi_k,a)   \quad {\rm where } \quad  \varphi_k(n):=\psi(kn)  \, .   $$

	% Denote by $\Lambda_k:=\Lambda(T^k)$ the Markov subsystem for $T^k$ constructed as above.
	Let $(\Lambda_k, T^k_\beta|_{\Lambda_k},\cP_{\Lambda_k})$ be the Markov subsystem  of  $([0,1],T^k_\beta)$ arising from  Proposition \ref{Markov}. The following claim will enable us to establish the hypotheses within  Proposition~\ref{Markovlowerbound}  regarding the existence of a compact set $K \supseteq \Lambda_k $ and an integer $k_0>0$, so that $T^{k_0}_\beta(P) \supseteq K$ for any interval $P\in \cP_{\Lambda_k}$.   As usual,  we let $K(\beta)$ denote  the support of the  Parry-Yrrap measure $\mu_\beta$. Recall, that  $K(\beta)$ is either the unit interval or a finite union of closed intervals -- see Propostion~\ref{Prop:support} in \S\ref{secmetricresult}.

\noindent {\bf Claim.} \emph{For any interval $I \subseteq [0,1]$, there exists an integer $k(I)>0$, such that $T_\beta^{k(I)}(I) \supseteq K(\beta)$.}

\begin{proof}[Proof of Claim]
  We will use the fact that  for any $|\beta|  > 1$, the map  $T_\beta $ is locally eventually onto (or topologically exact); i.e. for every non-degenerate subinterval $I \subseteq K(\beta)$ there exists a non-negative integer $k$ such that $T^k (I)  \supseteq  K(\beta)$.
 For $\beta> 1 $, this is explicitly stated and proved in the  work of  Troubetzkoy $\&$ Varandas \cite[Section 3.3]{TV} and since $K(\beta) =[0,1] $ when $\beta > 1$  it  directly establishes the claim.
    On the other hand, for  $\beta< -1 $ the fact is explicitly stated and proved in the  work of   Liao $\&$ Steiner \cite[Theorem 2.2]{LS12}.    As a consequence,  given any interval $I \subseteq [0,1]$,  if  $I\cap K(\beta)$ contains an interval then we are done. So assume that this is not the situation. Then, $I \cap \big( [0,1] \setminus K(\beta)\big)$ contains an  interval and to continue we consider two situations:
     \begin{itemize}
       \item[(a)] There exists a positive integer $\ell(I)$ such that $
		T_\beta^{\ell(I)}(I) \cap K(\beta) $  contains an interval.  In this case the claim follows  directly  from the locally eventually onto property of $T_\beta $.
       \item[(b)] If (a) does not hold, then for all $n\in \mathbb{N}$, $T^n(I)$ is contained in $[0,1] \setminus K(\beta)$ except for a  finite number of points. Therefore
	\[
	\lim_{n\to\infty}m_1\big(T^{-n}([0,1] \setminus K(\beta))\big)\ge   m_1(I)>0 \, ,
	\]
	where as usual $m_1$ is one-dimensional Lebesgue measure.  However, this contradicts the second assertion of \cite[Theorem 2.2]{LS12}; namely that  $\lim\limits_{n\to\infty}m_1(T^{-n}([0,1] \setminus K(\beta))) = 0 $.
\end{itemize}\end{proof}

	%On the other hand, all intervals $P(i)^*$ have non-empty intersection with the support $K_\beta$. If not, we can delete those intervals $P(i)$ from the partition.
	
On using the above claim, it follows that for any interval $P(i)\in \cP_{\Lambda_k}:=\{P(i): 1\le  i\le  m\}$  there exists an integer $k_0(i)>0$ such that $T_\beta^{k_0(i)}\big(P(i)\big)\supseteq K(\beta)$. The upshot of this is that  the hypotheses within Proposition~\ref{Markovlowerbound}  is satisfied with $K = K(\beta) $ and  $k_0= \max\limits_{1\le i\le m}k_0(i)$. Then on applying Proposition \ref{Markovlowerbound}, we have that
	\begin{eqnarray*}
		\dim_{\rm H}W(T_\beta,\psi,a)\, \ge  \, \dim_{\rm H}W(T_\beta^k,\varphi_k,a)\, \ge    \, \frac{1-{\log 8}/{(k\log\slT)}}{1+\lambda_k/(k\log\slT)}%\\
		%&\ge  & \frac{1}{1+\lambda_k/(k\log\slT)}\left(1-\frac{\log 4}{k\log\slT}\right).
	\end{eqnarray*}
where
 $\lambda_k:=\liminf_{n\to\infty}\frac{-\log\varphi_k(n)}{ n}$ is the lower order at infinity of $\varphi_k$.   Now by Lemma \ref{lemma_liminf_thesame}, since $\psi$ is a real positive decreasing function,  we have that $\lambda_k/k =  \lambda$  and  so on  letting $k\to\infty$ we obtain the desired lower bound for the dimension of $ W(T_\beta,\psi, a)$.
\end{proof}

\section{Final comments}

In this section we discuss various natural problems that arise as a consequence of the results proved in this paper.  The measure results (namely, Theorems~\ref{measure-matrix} - \ref{corInteger}) for matrix transformations are reasonably complete so the problems listed below are essentially concerned with Hausdorff dimension.

%\subsection{Generalization of Theorem~\ref{multiplicative} and Theorem~\ref{rectangledimresult}}
\subsection{Dimension problem for property $ ({\boldsymbol{\rm P}}) $ targets sets  \label{galSV}}

Theorem \ref{rectangledimresult} and Theorem \ref{multiplicative} give the Hausdorff dimension of shrinking target set $W(T,\{E_n\})$  when the targets sets $ \{E_n\}_{n \in \N } $ are a sequence of rectangles or  hyperboloids.     It is easily seen that both these ``shapes'' {when centred at the origin} satisfy the property $ ({\boldsymbol{\rm P}}) $  condition of Gallagher \cite{GallP} adapted for the torus: a subset $E$ of $\mathbb{T}^d$ is said to have property $ ({\boldsymbol{\rm P}}) $ if whenever $\vx=(x_1, \dots, x_d)\in E$ and $ \|  x_i'  \| \le   x_i$ ($1\le   i \le   d$) then $\vx'=(x_1', \dots, x_d')\in E$. Geometrically, the property simply means that the rectangle $B(0,x_1) \times \ldots \times B(0,x_d)$ is contained within $E$.   In short, it would be desirable to extend and thereby unify our dimension results {(with $\va:= (0, \ldots, 0)$ in the first instance)} to target sets satisfying property $ ({\boldsymbol{\rm P}}) $.     We now briefly  describe what we have in mind.

 Let $T$ be a real, non-singular matrix transformation of the torus $\T^d$.  Suppose that $T$ is diagonal and all eigenvalues  $\beta_1,\beta_2,\dots,\beta_d$  are strictly larger than $1$. Assume that $1<\beta_1\le \beta_2\le\cdots\le \beta_d$.  Let $\cP $  be any  collection of subsets $E$ of $\T^d$ satisfying property
$ ({\boldsymbol{\rm P}}) $.  Then,  for  any sequence $ \{E_n \}_{n \in \mathbb{N}} $ in $\mathcal{P}$, Theorem \ref{metricresult} implies that
$
m_d\big(W(T,\{E_n\})\big)   =   0 $  if   $ \sum_{n=1}^\infty m_d\big(E_n\big)<\infty   \, .
$
Thus, whenever the measure sum converges,  it is natural to ask for   the Hausdorff dimension of $W(T,\{E_n\})$.  Given that property
$ ({\boldsymbol{\rm P}}) $  is intimately tied up with rectangles, it is not unreasonable to expect that $\dim_{\rm H} W(T,\{E_n\})$  is in someway related to the Hausdorff dimension of the `rectangular' shrinking targets sets  $W(T,\Psi, \va)$ given by Theorem~\ref{rectangledimresult}. With this in mind, for any sequence $ \{E_n \}_{n \in \mathbb{N}}$ in  $\cP$, we propose the following candidate for the  dimension formula:
\begin{equation}\label{conjectureformula}
\dim_{\rm H}W(T,\{E_n\}) \  = \!\!
\sup_{\Psi \, :  \,  \forall \,  n \in \N \atop R(0, \Psi(n))\subseteq E_n } \!\!\!\!
 \dim_{\rm H}W(T,\Psi, 0).
\end{equation}
%\sv{
%where given  $\Psi := (\psi_1,\dots,\psi_d)$  and   $n \in \N$
%$$
% R_n\big( \Psi \big):= \Big\{ \vx  \in \mathbb{T}^d :  |x_i| \le \psi_i(n)   \ (1\le   i \le   d) \Big\}.
%$$}
%\sv{[OR]}
Here, as in \S\ref{pfthm6},  given  $\Psi := (\psi_1,\dots,\psi_d)$  and some point  $\va:= (a_1, \ldots,  a_d)  \in \T^d$, for  $n \in \N$ we let
$$
 R\big(\va, \Psi(n) \big):= \Big\{ \vx  \in \mathbb{T}^d :  \|x_i-a_i\| \le \psi_i(n)   \ (1\le   i \le   d) \Big\}.
$$
\\
Observe, that since in \eqref{conjectureformula} the supremum is over   $\Psi$ such that the corresponding rectangles $R\big(0, \Psi(n) \big)$ are a subset of the sets $ E_n$ satisfying property $ ({\boldsymbol{\rm P}}) $, we automatically  obtain the desired lower bound statement:
 \begin{equation}\label{conjectureformulalb}\dim_{\rm H}W(T,\{E_n\})  \ \geq  \!\!  \sup_{\Psi \, :  \,  \forall \,  n \in \N \atop R(0, \Psi(n))\subseteq E_n }   \!\!\!\!
 \dim_{\rm H}W(T,\Psi, 0).
 \end{equation}
Thus, establishing \eqref{conjectureformula} boils down to establishing the complimentary  upper bound statement.

It is not difficult to see that the dimension formula \eqref{conjectureformula} holds when the targets sets $ \{E_n\}_{n \in \N } $ are a sequence of rectangles as in Theorem~\ref{rectangledimresult} or  hyperboloids as in Theorem~\ref{multiplicative}.    The former is obvious. Regarding the latter, for $n \in \N$ we let  $\Psi(n):=(1, \cdots, 1, \psi(n))$. Then,
 \begin{equation}\label{choice_rect}
 R\big(0, \Psi(n) \big)=B(0,1)\times\cdots\times B(0,1)\times B(0, \psi(n))
 \end{equation}
%Now we apply Theorem~\ref{rectangledimresult}. We have
 and with reference to Theorem~\ref{rectangledimresult}
 $$\UP=\left\{(0,0,\cdots,0,t_d): \text{$t_d$ is an accumulation point of $\Big\{-\frac{\log\psi(n)}{n}\Big\}_{n\geq 1}$} \right\}.$$
 Furthermore, for $1\leq i\leq d-1$, we have that $\mathcal{K}_1(i)=\{i+1, i+2, \cdots, d\}$, $\mathcal{K}_2(i)=\{1, 2,\cdots, i\}$, and $\mathcal{K}_3(i)=\emptyset$. Hence, $$\theta_1(0,0,\cdots,0,t_d) \, = \, \cdots  \, =  \, \theta_{d-1}(0,0,\cdots,0,t_d)  \, = \, d \, .$$
For $i=d$,  we have that  $\mathcal{K}_1(d)=\mathcal{K}_3(d)=\emptyset$ and $\mathcal{K}_2(d)=\{1, 2,\cdots, d\}$. Thus,
$$\theta_d(0,0,\cdots,0,t_d)  \, = \, d-1+\frac{\log\beta_d}{t_d+\log\beta_d}.$$ Therefore, on applying  Theorem~\ref{rectangledimresult} we obtain that
\begin{eqnarray*}
\dim_{\rm H}W(T,\Psi, 0) & = & \sup_{t_d}\min_{1\leq i\leq d}\theta_i(0,0,\cdots,0,t_d)  \,  =  \,   \sup_{t_d}\left\{d-1+\frac{\log\beta_d}{t_d+\log\beta_d}\right\}  \\ [2ex]
&= & d-1+\frac{\log\beta_d}{\lambda+\log\beta_d}  \, ,
\end{eqnarray*}
where $\lambda$ is the lower order at infinity of $\psi$. Since this dimension formula coincides with the dimension of $W^\times(T,\Psi, 0)$ given by Theorem~\ref{multiplicative}, and we always have the lower bound (\ref{conjectureformulalb}), we conclude that the supremum in (\ref{conjectureformula}) is attained by the choice of rectangles given by  (\ref{choice_rect}).   In other words, the dimension formula \eqref{conjectureformula} holds when the targets sets $ \{E_n\}_{n \in \N } $ are a sequence of  hyperboloids as in Theorem~\ref{multiplicative}.

The following is an extension of  Gallagher's  property $ ({\boldsymbol{\rm P}}) $  condition that naturally incorporates ``shapes''  not necessarily centred at the origin.   Given $ \va \in \T^d$,   a subset $E$ of $\mathbb{T}^d$ is said to have property $ ({\boldsymbol{\rm P_{\! \va}}}) $ if whenever $\vx=(x_1, \dots, x_d)\in E$ and   $  \|x_i' - a_i\| \le x_i$ ($1\le   i \le   d$) then $\vx'=(x_1', \dots, x_d')\in E$.  Now with this in mind, let $\cP_{\va} $  be any  collection of subsets $E$ of $\T^d$ satisfying property $({\boldsymbol{\rm P_{\! \va}}})$.
 Then,  for  any sequence $ \{E_n \}_{n \in \mathbb{N}} \in \cP_{\va}$,  we propose that \eqref{conjectureformula} holds with the origin replaced by $\va$. Clearly, such a statement  would unify in full our dimension results for rectangular and hyperboloid target sets; that is,  not just for when $\va:= (0, \ldots, 0)$.

 % In order to obtain a uniform formula for $\dim_{\rm H}W\big(T,\{E_n\}\big)$ it may be necessary to impose some restriction on  the class of targets sets satisfying property $ ({\boldsymbol{\rm P}}) $.  With this in mind ..
%
%
%In the first instance, it may be beneficial to restrict our class of targets sets satisfying property $ ({\boldsymbol{\rm P}}) $.  in the .

%\subsection{Extending Theorem~\ref{rectangledimresult}}

\subsection{Dimension problem for diagonal matrices with negative entries}  \label{negohya} %{Distribution of full cylinders for negative $\beta$-transformation}

In the one dimensional case, Theorem~\ref{dimensionalone} extends Theorem~\ref{dimresult} by incorporating negative eigenvalues. Naturally,  it would be desirable to obtain the higher dimensional analogue of Theorem~\ref{dimensionalone}.  Indeed, this would clearly follow if we could extend Theorem~\ref{rectangledimresult} (the ``rectangular'' generalization Theorem~\ref{dimresult}) to the situation that all eigenvalues of $T$ are of modulus strictly larger than $1$.  Formally,  we would expect the following statement to hold in which the conditions on the eigenvalues in  Theorem~\ref{rectangledimresult} are replaced by conditions on the modulus of the eigenvalues.

%Recall that the shrinking target set $W(T,\Psi, \va)$ is defined by \eqref{yesyes}.

\noindent{\bf Claim 1.}
	{\it Let $T$ be a real, non-singular matrix transformation of the torus $\T^d$.  Suppose that $T$ is diagonal and all eigenvalues  $\beta_1,\beta_2,\dots,\beta_d$  are of modulus strictly larger than $1$. Assume that $1<|\beta_1|\le |\beta_2|\le\cdots\le |\beta_d|$. For $ 1 \le i \le d$, let $\psi_i: \mathbb{R}^+\to \mathbb{R}^+$ be a real positive decreasing function and $\va \in K=\prod_{i=1}^d K(\beta_i)$. Assume that the set $\UP$ of accumulation points $ \mathbf{t}=(t_1,t_2, \ldots, t_d) $ of the sequence $\big\{(-\frac{\log\psi_1(n)}{n},\cdots,
-\frac{\log\psi_d(n)}{n})\big\}_{n\ge   1}$ is bounded. Then
	$$\dim_{\rm H} W(T,\Psi, \va)=\sup_{\mathbf{t}\in\UP} \min_{1\le   i\le   d}\big\{\theta_i(\mathbf{t})\big\},$$
	where
%	\begin{equation*}
%		\theta_i(\mathbf{t}):=
%		\sum_{|\beta_k|>|\beta_i|e^{t_i}}1+\sum_{|\beta_k|e^{t_k}\le   |\beta_i|e^{t_i}}\left(1-\frac{t_k}{\log|\beta_i|+t_i}\right) +\sum_{|\beta_i|e^{t_i}\ge  |\beta_k|>|\beta_i|e^{t_i-t_k}}\frac{\log|\beta_k|}{\log|\beta_i|+t_i}.
%	\end{equation*}
	\begin{equation*}
			\theta_i(\mathbf{t}):=
			\sum_{k\in \mathcal{K}_1(i)}1+\sum_{k\in \mathcal{K}_2(i)}\left(1-\frac{t_k}{\log |\beta_i|+t_i}\right) +\sum_{k\in \mathcal{K}_3(i)}\frac{\log|\beta_k|}{\log|\beta_i|+t_i}
		\end{equation*}
and, in turn
			\begin{equation*}
			\mathcal{K}_1(i):= \{1\le   k\le   d: \log |\beta_k|>\log |\beta_i|+t_i\}, \ \mathcal{K}_2(i):= \{1\le   k\le   d: \log |\beta_k|+t_k\le   \log |\beta_i|+t_i\}, \
			\end{equation*}
and
\begin{equation*}
  \mathcal{K}_3(i):=\{1, \dots, d\}\setminus (\mathcal{K}_1(i)\cup \mathcal{K}_2(i)).
\end{equation*}

}

\bigskip

% \red{***************   NEED WORKING ON  **************}

The key problem with allowing negative eigenvalues is that we do not have  an analogue of  Fact BW in Section 4.2.1 for negative $\beta$-transformations. This fact played a key role our proofs of the  upper bound (Proposition~\ref{upperbound}) and lower bound (Proposition~\ref{lowerbound})   statements for the Hausdorff dimension of  $\dim_{\rm H} W(T,\Psi, \va)$.   However,  by exploiting the framework of  Markov subsystems  used  in proving  Theorem~\ref{dimensionalone}, it is not too difficult to establish  the lower bound of the above claim; that is to say, we can bypass Fact~BW altogether and prove that $$\dim_{\rm H} W(T,\Psi, \va)\ge   \sup_{\mathbf{t}\in\UP} \min_{1\le   i\le   d}\big\{\theta_i(\mathbf{t})\big\}  \, . $$
\\
Indeed, for  each $1\le   i\le   d$, by Proposition \ref{Markov} there exists a Markov subsystem $(\Lambda^{(i)}, T_{\beta_i}|_{\Lambda^{(i)}}, \mathcal{P}_{\Lambda^{(i)}})$ of $([0, 1), T_{\beta_i})$  under the assumption that $|\beta_i| > 8$.   Also, in view of Proposition \ref{MarkovAlfhor} we know that  the measure $\cH^{\delta_i}|_{\Lambda^{(i)}}$ is $\delta_i$-Ahlfors regular where $\delta_i=\dim_{\rm H}\Lambda^{(i)}$. Now, let  $S_i=T_{\beta_i}|_{\Lambda^{(i)}}$ and consider restricted shrinking target set
	$$W^*(T, \Psi, \va):=\Big\{\vx\in\prod_{i=1}^d\Lambda^{(i)}:|S_i^nx_i-a_i| \le \psi_i(n)  \ (1\le   i \le   d) \ \ \text{for infinitely many}\ n\in\mathbb{N}\Big\}  \,  . $$
Then, by definition,
$$
W^*(T, \Psi, \va)\subset W(T, \Psi, \va).	
$$
The first goal is obtain  a lower bound for $\dim_{\rm H} W^*(T,\Psi, \va)$. For this, we follow the basic strategy used in proving  Proposition~\ref{lowerbound}.  However, the key in executing the strategy lies in the fact that each map $S_i$
$(1\le   i\le   d)$
%(or more precisely the Markov subsystem $(\Lambda^{(i)}, S_i,\mathcal{P}_{\Lambda^{(i)}})$)
 satisfies the hypothesis of  Proposition~\ref{Markovlowerbound} --  this follows on using  the same arguments used at the end of the proof of Theorem~\ref{dimensionalone} to show that the
 Markov subsystem $(\Lambda_m, T^m_\beta|_{\Lambda_m},\cP_{\Lambda_m})$ of  $([0,1],T^m_\beta)$ arising from  Proposition \ref{Markov} satisfies the hypotheses of   Proposition~\ref{Markovlowerbound}. Then, on naturally adapting the arguments leading to (58) within the proof of Proposition~\ref{Markovlowerbound}, it follows that for each $1\leq i\leq d$ there exists an integer $k_0^{(i)}$  such that
%\begin{equation}\label{rectangles-Sec5} %\bigcap_{N=1}^{\infty}\bigcup_{n=N}^{\infty}
% W^*(T,\Psi,\va)\supset\limsup_{n\to\infty}\bigcup_{j_1=1}^{M_{1,n}} \cdots\bigcup_{j_d=1}^{M_{d,n}} B\big(x_{n,j_1}^{{(1)}}, {\psi_1(n-1+k_0^{(1)}) \over |\beta|_1^{n-1+k_0^{(1)}}}\big)\times\cdots\times B\big(x_{n,j_d}^{(d)},{\psi_d(n-1+k_0^{(d)}) \over |\beta|_d^{n-1+k_0^{(d)}}}\big) \, ,
% \end{equation}
 \begin{equation}\label{rectangles-Sec5} %\bigcap_{N=1}^{\infty}\bigcup_{n=N}^{\infty}
 W^*(T,\Psi,\va)\supseteq\limsup_{n\to\infty}\bigcup_{j_1=1}^{M_{1,n}} \cdots\bigcup_{j_d=1}^{M_{d,n}} \prod_{i=1}^d B\Big(x_{n,j_i}^{{(i)}}, {\psi_i(n-1+k_0^{(i)}) \over |\beta_i|^{n-1+k_0^{(i)}}}\Big),
 \end{equation}
where $\{x_{n,j_i}^{(i)}, 1\le   j_i\le   M_{i,n}\}$ are the preimages of $a_i$ under $S_{i}^n$ that fall within  cylinders  of order $n$ for $S_{i}$ and  $M_{i,n}$ is the number of such cylinders.   This is the  analogue of the inclusion \eqref{limsupsubset} in the proof of Proposition~\ref{lowerbound}. Now in view of \eqref{fullset1} within the proof of Proposition \ref{Markovlowerbound}, it follows that for each $1\leq i\leq d$ there exists a contant $\kappa_i^*$ such that
 $$
 \bigcup_{j_i=1}^{M_{i,n}}   B  \big( x_{n,j_i}^{(i)},  \kappa_i^* |\beta_i|^{-n} \big)   \supseteq \Lambda^{(i)}  \, .
 $$
In particular, this leads to the following  analogue of \eqref{fulllimsup}  in the proof of Proposition~\ref{lowerbound}, for any fixed $\mathbf{t}=(t_1, \ldots, t_d) \in\UP$:
\begin{equation*}\label{fulllimsup*}
 \prod_{i=1}^d\Lambda^{(i)}   \ \subseteq\ \limsup_{n\to\infty}  \bigcup_{j_1=1}^{M_{1,n}} \cdots\bigcup_{j_d=1}^{M_{d,n}} \prod_{i=1}^d B\Big(x_{n,j_i}^{{(i)}}, \Big({\psi_i(n-1+k_0^{(i)}) \over |\beta_i|^{n-1+k_0^{(i)}}}\Big)^{s_i}  \ \Big) \,
 \end{equation*}
% \begin{equation*}\label{fulllimsup*}
% \prod_{i=1}^d\Lambda^{(i)}   \ \subseteq\ \limsup_{n\to\infty}  \bigcup_{j_1=1}^{M_{1,n}} \cdots\bigcup_{j_d=1}^{M_{d,n}} B\big(x_{n,j_1}^{(1)},(\kappa_1^*|\beta|_1^{-n}\psi_1(n))^{s_1}\big)\times\cdots\times B\big(x_{n,j_d}^{(d)},(\kappa_d^*|\beta|_d^{-n}\psi_d(n))^{s_d}\big) \,
% \end{equation*}
 where
 $$
 0< s_i < s_0(i)  := \frac{\delta_i}{1+t_i/\log |\beta_i|}\, .
 $$
 The upshot is that given the $\limsup$ set of rectangles appearing on the right hand side of \eqref{rectangles-Sec5}, the corresponding  $\limsup$ set of `$(s_1, \ldots, s_d)$-scaled up' rectangles satisfies
 %\red{??Need working on??} This verifies
  \eqref{fullmeasure} with  $p=d, X_i=\Lambda^{(i)}, \delta_i=\dim_{\rm H}\Lambda^{(i)}$ and $\mu_i=\cH^{\delta_i}|_{\Lambda^{(i)}}$  for each $  1 \le i \le d$.
Thus on applying Theorem \ref{MTPRG} with $u_i=(1-\varepsilon)\log|\beta_i|$ and $ v_i=t_i$ $(1 \le i \le d)$, we find as in the proof of Proposition~\ref{lowerbound}, that
$$\dim_{\rm H} W(T,\Psi, \va) \ \ge \   \dim_{\rm H} W^*(T,\Psi, \va) \ \ge  \  \sup_{\mathbf{t}\in\UP} \min_{1\le   i\le   d}\big\{\widehat{\theta}_i(\mathbf{t})\big\}  \,  ,$$
where
$$\widehat{\theta}_i(\mathbf{t}):=\sum_{k\in \mathcal{K}_1(i)}\delta_k+\sum_{k\in \mathcal{K}_2(i)}\delta_k\left(1-\frac{t_k}{\log |\beta_i|+t_i}\right) +\sum_{k\in \mathcal{K}_3(i)}\frac{\delta_k\log|\beta_k|}{\log|\beta_i|+t_i}
		$$
and $\mathcal{K}_1(i)$, $\mathcal{K}_2(i)$, $\mathcal{K}_3(i)$ are defined as  in  Claim 1.
 %due to the fact that $(\Lambda^{(i)}, S_i,\mathcal{P}_{\Lambda^{(i)}})$ is Markov, the lengths of all cylinders of same order are comparable
%	
%	
%	
%	 we can apply Theorem \ref{MTPRG} to the setup in which
%$$
%X_i=\Lambda^{(i)}.
%$$
%$$W(T,\Psi,\va)=\big\{\vx\in\mathbb{T}^d: |T_{\beta_i}^nx_i-a_i| \le \psi_i(n)  \ (1\le   i \le   d) \ \ \text{for infinitely many}\ n\in\mathbb{N}\big\}  \, ,$$
%$$
%S=T|_{\prod_{i=1}^d\Lambda^{(i)}}.
%$$
%% to the product space $(\prod_{i=1}^d\Lambda^{(i)},\prod_{i=1}^d\mu_i)$,
%to obtain a lower bound of $\dim_{\rm H} W(S,\Psi, \va)$, with

 To obtain the desired lower bound,  for each $1\leq i \leq d$ we need to (i)  overcome  the underlying assumption that $|\beta_i| > 8$ in the argument above  and (ii)  replace $\widehat{\theta}_i(\mathbf{t})$ by $\theta_i(\mathbf{t})$ in the above lower bound estimate; i.e., replace $\delta_i $ by  $1$ in the definition of $\widehat{\theta}_i(\mathbf{t})$.   As in the proof of Theorem~\ref{dimensionalone},  we  deal with  (i) by working with a high enough iterate  $S_i^m:=T_{\beta_i}^m|_{\Lambda_m^{(i)}}$ of the map $S_i$ and replacing  $\psi(n) $  by $ \varphi_m(n):=\psi(mn)$ and then letting $m$ become arbitrarily large.   This also deals with (ii) since  by Proposition~\ref{Markov},  $\dim_{\rm H} \Lambda_m^{(i)} \to 1 $ as $m \to \infty$.

\subsection{ Theorem \ref{rectangledimresult} for unbounded $\UP$}
In the statement of Theorem \ref{rectangledimresult}, we require that the set  $\UP$ of accumulation points is bounded. %Such assumption is needed to apply Wang and Wu's Mass Transference Principle, Theorem .
In short, this   allows us to directly exploit  the ‘rectangles to rectangles’ Mass Transference Principle (Theorem 11).
However, this is  a matter of convenience  and  it should be possible to obtain a general form of Theorem \ref{rectangledimresult}  (and indeed Claim~1 in \S\ref{negohya}) without assuming that $\UP$ is bounded.
 Indeed, by adapting the arguments used in  this paper we can ``directly'' establish various partial statements. These we now briefly describe.

For $\mathbf{t}=(t_1, \ldots,t_d)\in (\mathbb{R}^+\cup\{\infty\})^d$,  let $$\mathcal{L}_1(\mathbf{t}):=\{1\le   i \le   d: t_i<+\infty\}  \ \quad \text{and}  \  \quad \mathcal{L}_2(\mathbf{t}):=\{1\le   i \le   d: t_i=+\infty\}.$$
In turn,  with ${\theta}_i(\mathbf{t})$ as in the statement of Theorem~\ref{rectangledimresult}, define  $\tilde{\theta}_i(\mathbf{t})$ to be the reduced value of ${\theta}_i(\mathbf{t})$ obtained  by removing the ``infinite'' directions associated with  $\mathcal{L}_2(\mathbf{t})$; i.e.,
	\begin{equation*}
			\tilde{\theta}_i(\mathbf{t}):=
			\sum_{k\in \mathcal{K}_1(i)\setminus \mathcal{L}_2(\mathbf{t})}1+\sum_{k\in \mathcal{K}_2(i)\setminus \mathcal{L}_2(\mathbf{t})}\left(1-\frac{t_k}{\log |\beta_i|+t_i}\right) +\sum_{k\in \mathcal{K}_3(i)\setminus \mathcal{L}_2(\mathbf{t})}\frac{\log|\beta_k|}{\log|\beta_i|+t_i}
		\end{equation*}
%\begin{theorem}\label{rectangledimresult}
%  Let $T$ be a diagonal matrix transformation with all the diagonal entries $\beta_1,\beta_2,\dots,\beta_d$ strictly larger than $1$. Assume that $1<\beta_1\le \beta_2\le\cdots\le \beta_d$. For $ 1 \le i \le d$, let $\psi_i: \mathbb{R}^+\to \mathbb{R}^+$ be a real positive %{\color{blue} decreasing}
%		function and $\va \in \mathbb{T}^d$. Then
Then, under the setting of  Theorem \ref{rectangledimresult}  \emph{but without  the assumption that $\UP$ is bounded}, we are able to adapt the proofs of Propositions~\ref{upperbound}~and~\ref{lowerbound} to show that:
\begin{equation}\label{upperlower}
\sup_{\mathbf{t}\in\UP}  \min\big\{ \min_{i\in\mathcal{L}_1(\mathbf{t})}\{\tilde{\theta_i}(\mathbf{t})\}, \ \#\mathcal{L}_1(\mathbf{t})\big\}\le  \dim_{\rm H} W(T,\Psi,\va)\le    \sup_{\mathbf{t}\in\UP}\min \big\{\min_{i\in \mathcal{L}_1(\mathbf{t})}\{{\theta}_i(\mathbf{t})\}, \ \#\mathcal{L}_1(\mathbf{t})\big\}.
\end{equation}
\\
%The difference between the upper and lower bounds is that ${\theta}_i(\mathbf{t})$ is replaced by $\tilde{\theta_i}(\mathbf{t})$. These estimates in general could not give us the precise dimensional number of $W(T, \Psi, \va)$
%since the upper and lower bounds may not coincide.
Clearly, in the case $\UP$ is bounded the  upper and lower estimates in  \eqref{upperlower} coincide.  In the unbounded case,  this is not necessarily true and so the estimates do not in general provide a precise formula for the dimension.  We illustrate this with a concrete example.
 %Nevertheless, even in this case there are circumstances when we are able to combine  \eqref{upperlower} with ``other'' arguments to  establish a precise formula.  For example,
Let
$d=2$ and  $T$ to be the  diagonal matrix with entries $\beta_1=2$ and $\beta_2=3$.   Also, given a real number  $t_1> 0$, let $\psi_1(n)=e^{-nt_1}$ and $\psi_2(n)=e^{-n^2}$. Then,
it is easily verified that  \eqref{upperlower}  implies that for any $t_1 >0$
\begin{equation} \label{notgood} \frac{\log 2}{\log 2+t_1}    \, \leq   \,   \dim_{\rm H}W(T,\Psi,\va) \,  \le  \,   \min \Big\{  1 , \frac{\log 2+\log 3}{\log 2+t_1} \Big\}
.
\end{equation}
To the best of our knowledge the  precise formula for the dimension is unknown and is not a consequence of know results in the theory of Diophantine approximation.   In a forthcoming paper \cite{Dorsa}, by using  the `old school'' approach of constructing optimal Cantor-type subsets of the set under consideration and applying the Mass Distribution Principle \cite[Section 4.1]{F}, it is shown that
$$
\dim_{\rm H}W(T,\Psi,\va)  =  \min \Big\{  1 , \frac{\log 2+\log 3}{\log 2+t_1} \Big\}  \; ;  $$
that is,    the upper bound in \eqref{notgood} is sharp.   In general,  we are therefore  lead to believe that the   upper bound in \eqref{upperlower}  is sharp.   In \cite{Dorsa},  we show that this is indeed the case.

For the sake of completeness, we mention that in \cite{Dorsa} we also address the analogous ``unbounded'' problem in the classical theory of simultaneous Diophantine approximation.  For example, given a real number $\tau >0$,  let $S(\tau)$ denote the set of $(x_1,x_2)\in \mathbb{R}^2 $ for which the inequalities
$$
\|n x_1 \| \, < \, n^{-\tau}   \qquad   {\rm and}  \qquad \|n x_2 \| \, < \, e^{-n}
$$
hold for infinitely many $n \in \N$.   Then it follows from  known ``classical'' statements (see for example \cite{Rynne})  that
$
\dim S(\tau) = 1 $  for $1/2 \le \tau  \le 1 $,   and  that for $\tau > 1 $
%$$
%\dim S(\tau) = 1 \qquad   {\rm for}  \qquad   1/2 \le \tau  \le 1 ,
%$$ and for $\tau > 1 $
$$
\frac{2}{1+\tau}   \, \le \, \dim S(\tau)   \, \le \,   \min \Big\{   1, \frac{3}{1+\tau}  \Big\}  \, .
$$
However, to the best of our knowledge we do not have a precise formula for the dimension when $\tau > 1$.  In \cite{Dorsa},  it is shown that for $\tau \ge 1/2 $
$$
 \dim S(\tau)   \, = \,   \min \Big\{   1, \frac{3}{1+\tau}  \Big\}  \, .
$$

%$\dim S(\tau) = 1 $  for $\tau  \le 2 $   and that  $ \dim S(\tau) = 3/(1+ \tau)  $  for $ \tau > 2 $.

%As a consequence, when $\UP$ is bounded, we have
%		$$\dim_{\rm H} W(T,\Psi, \va)=\sup_{\mathbf{t}\in\UP} \min_{1\le   i\le   d}\big\{\theta_i(\mathbf{t})\big\}.$$
%\end{theorem}

%\subsection{Other related directions of study}
\subsection{Badly approximable sets}

Let $T$ be a real, non-singular matrix transformation of the torus $\mathbb{T}^d$. Suppose that all eigenvalues of $T$ are of modulus  strictly larger than $1$ and let $\mathcal{C}$ be any  collection of subsets $E$ of $\mathbb{T}^d$ satisfying the bounded property $ ({\boldsymbol{\rm B}}) $.  For  any sequence $ \{E_n \}_{n \in \mathbb{N}} $ of subsets  in $\mathcal{C}$,  we can consider the \emph{badly approximable set with respect to the sequence} $ \{E_n \}_{n \in \mathbb{N}} $ as follows:
$$
\Bad\big(T,\{E_n\}\big) :=  \big\{\vx\in  \mathbb{T}^d:   \ \exists \ n_0(\vx)\in\mathbb{N}\   \ \text{such that} \ \ T^n(\vx)\not\in E_n\  \ \forall   \ n\geq n_0(\vx)\big\} \, .
$$
It is  easily seen that the set $\Bad\big(T,\{E_n\}\big) $ is the  complement of shrinking target set $W\big(T,\{E_n\}\big)$ and consists of points $\vx\in  \mathbb{T}^d$ whose orbit under $T$ eventually avoids the given sequence of subsets $E_n$ in $\cC$.  Hence,
 Theorem \ref{measure-matrix} provides us a criterion on the zero-one $d$-dimensional Lebesgue measure of $\Bad\big(T,\{E_n\}\big)$.  Indeed, in the case $T$ is diagonal and all eigenvalues are strictly larger than $1$, it follows  via Theorem~\ref{metricresult} that if $\sum_{n=1}^\infty m_d\big(E_n\big)=\infty$, then
\begin{equation}\label{Bad-measure}
m_d\Big(\Bad\big(T,\{E_n\}\big)  \Big) =     m_d\Big(\T^d \setminus W(T,\{E_n\})\Big)   =   0\, .
% \qquad { \rm if \ } \qquad  \sum_{n=1}^\infty m_d\big(E_n\big)=\infty.   \,
\end{equation}
Thus, whenever the measure sum diverges,  it is natural to ask for   the Hausdorff dimension of $\Bad\big(T,\{E_n\}\big)$.
We suspect that for a large class of subsets in $\cC$, such as those satisfying the stronger property $ ({\boldsymbol{\rm P}}) $, the associated badly approximable sets are of full dimension.   Indeed, it is  plausible that they are winning sets in the sense of Schmidt's framework of  $(\alpha,\beta)$--games  - see \cite[\S 1.7.2]{durham} and references within.  Note that there are obvious cases for which $\Bad\big(T,\{E_n\}\big)$ is empty  (for example if  $E_n =  \mathbb{T}^d$ for all $n \in \N$) and these should naturally be
excluded.

To give a little background and to motivate  a concrete problem,   we consider the special case when the sequence $\{E_n \}_{n \in \mathbb{N}} $ corresponds to balls. More precisely, given  $\va \in \T^d$ and  a decreasing  function $\psi: \mathbb{R}^+\to \mathbb{R}^+$, let
\begin{eqnarray*}
\Bad(T,\psi,\va)  & :=  &   \T^d \setminus W(T,\psi,\va) \\[1ex]
& = & \big\{\vx\in  \mathbb{T}^d:   \liminf_{n\to\infty} \ \psi(n)^{-1} \ \|T^n\vx-\mathbf{a}\|>1 \big\}  \, .
\end{eqnarray*}
Also, let
$$
\Bad(T,\va)   :=   \big\{\vx\in  \mathbb{T}^d:   \liminf_{n\to\infty}  \|T^n\vx-\mathbf{a}\|>0 \big\}  \,  .
$$
Then, it is easily seen that if  $\psi(n)\to 0$ as $n\to\infty$, then
$${\bf Bad}(T,\mathbf{a})\subset \Bad\big(T, \mathbf{a}, \psi\big) $$
and so if the badly approximable set  $ {\bf Bad}(T,\mathbf{a})$ has full dimension then so does $\Bad(T,\psi,\va)$.    With this is mind,   Dani \cite{Dani1988} showed that if $T$ is a non-singular, semi-simple integer matrix
and $\mathbf{a} \in  \mathbb{Q}^d/\mathbb{Z}^d$, then $\dim_{\rm H}  \Bad(T,\va) = d$. In fact, he showed that  ${\bf Bad}(T, \mathbf{a})$ is winning.   Dani's winning result was later extended by
 Broderick, Fishman $\&$ Kleinbock \cite{BFK2011} to any non-singular, integer matrix transformation
 and $\mathbf{a}  \in \mathbb{T}^d$. Regarding non-integer matrix transformations,  we  have a complete dimension result in dimension one.   Indeed, for any $\beta \in (1,2]$ and  $a\in \mathbb{T}$,  F\"{a}rm, Persson $\&$ Schmeling \cite{FPS2010} have shown that ${\Bad}(T_\beta, a)$ is  ``strong''  winning and hence has full Hausdorff dimension.  Subsequently, Yang $\&$ Wang \cite{YW2019} extended the full Hausdorff dimension result to  any $\beta>1$.
   To the best of our knowledge, the problem of determining $ \dim_{\rm H} {\Bad}(T, \va)$ when $T$ is a real, non-singular matrix transformation of $\mathbb{T}^d$ with $d \ge 2 $  is open.  In fact, it seems that the dimension result is currently unknown even in the case that $T$ is a diagonal matrix with all eigenvalues strictly larger than $1$.

 Now let us consider the  special case when the sequence $\{E_n \}_{n \in \mathbb{N}} $ corresponds to hyperboloids.  For the sake of simplicity, suppose that $T = {\rm diag} \, (t_1, \ldots, t_d  ) $ is an integer, diagonal  matrix with $t_i   \ge 2$.   Then in line with the discussion above for balls, given  $\va =(a_1, \ldots, a_d  ) \in \T^d$ and  a  decreasing  function $\psi: \mathbb{R}^+\to \mathbb{R}^+$, we consider the sets \begin{eqnarray*}
\Bad^\times(T,\psi,\va)  & :=  &   \T^d \setminus W^\times(T,\psi,\va) \\[1ex]
& = &
\Big\{\vx\in \mathbb{T}^d: \liminf_{n\to\infty} \ \psi(n)^{-1} \ \prod_{1\le i \le d} \|t_i^nx_i - a_i\| > 1 \Big\} \,
\end{eqnarray*}
and
$$
\Bad^\times(T,\va)   :=   \Big\{\vx\in  \mathbb{T}^d:   \liminf_{n\to\infty}  \prod_{1\le i \le d} \|t_i^nx_i - a_i\|>0 \Big\}  \,  .
$$
If  $\psi(n)\to 0$ as $n\to\infty$, then
${\bf Bad}^\times (T,\mathbf{a})\subset \Bad^\times \big(T, \mathbf{a}, \psi\big) $
and so the aim is to show that the multiplicative badly approximable set  $ \Bad^\times(T,\mathbf{a})$ has full dimension.  Given the one dimension result for balls (namely that $\dim {\Bad}(T_\beta, a)=1$), this is relatively straightforward to establish.   Indeed, we start with  the  observation  that
$$
\liminf_{n\to\infty} \prod_{1\le i \le d} \|t_i^nx_i - a_i\|  \ge  \prod_{1\le i \le d}  \liminf_{n\to\infty}\|t_i^nx_i - a_i\|.
$$
Thus, it follows that
$$
 \prod_{1\le i \le d} \Bad(t_i, a_i) \subseteq \Bad^{\times}(T,\va),
$$
where for each  $1\leq i\leq d$
$$
\Bad(t_i, a_i):=\Big\{x_i\in \T: \liminf_{n\to\infty} \|t_i^nx_i - a_i\| >0 \Big\}.
$$
In turn, since each $\Bad(t_i, a_i)$  has Hausdorff dimension $1$, we obtain that
$$
\dim_{\rm H}\Bad^{\times}(T,\va)\geq \dim_{\rm H}\prod_{1\le i \le d} \Bad(t_i, a_i) \geq  \sum_{1\le i \le d}  \dim_{\rm H} \Bad(t_i, a_i)=d  \, .
$$
The complementary upper bound statement is trivial. Thus, $ \dim_{\rm H}\Bad^{\times}(T,\va) = d $ as desired.

\medskip

We now describe a class of sequences $ \{E_n \}_{n \in \mathbb{N}} $ that naturally unify the above badly approximable sets for balls and hyperboloids.  At the same time it allows us to state a concrete problem.
Suppose that $E$ is a subset of  $\mathbb{T}^d$ satisfying the bounded property $ ({\boldsymbol{\rm B}}) $ and furthermore suppose that $E$ contains the origin.  Next,  given $ \va \in \T^d$ and  a decreasing function $\psi: \mathbb{R}^+\to \mathbb{R}^+$,   for each $ n \in \N$ let  $$E_n(\va,\psi) =\va + \psi(n)E := \big\{ \va + \psi(n) \vx  : \vx \in E \big\}   \, .  $$
Note that if $E$ satisfies  Gallagher's property $ ({\boldsymbol{\rm P}}) $  condition then for each $n \in \N$, the set $E_n(\va,\psi)$ satisfies the general property $ ({\boldsymbol{\rm P_{\! \va}}}) $ introduced in \S\ref{galSV}. Now let
$$
\Bad(T,\psi,\va, E):= \big\{\vx\in  \mathbb{T}^d:   \ \exists \ n_0(\vx)\in\mathbb{N}\   \ \text{such that} \ \ T^n(\vx)\not\in  E_n(\va,\psi) \  \ \forall   \ n\geq n_0(\vx)\big\}   \,
$$
denote the badly approximable set with respect to the sequence $ \{E_n(\va,\psi) \}_{n \in \mathbb{N}} $.  Furthermore, let
$$
\Bad(T,\va, E):= \big\{\vx\in  \mathbb{T}^d:   \ \exists \ c(\vx)>0 \   \ \text{such that} \ \ T^n(\vx)\not\in   \   \va + c(\vx) E \   \ \forall   \ n \in \N \big\}  \, .
$$
It is easily verified, that if we take $E $ to be the ball $ B(\mathbf{0}, 1)$  (resp. the hyperbola $H(\mathbf{0}, 1) $)   then  $ \Bad(T,\psi,\va, E)$ coincides with   $\Bad(T,\psi,\va)$  (resp.  $\Bad^\times (T,\psi,\va)$) and $\Bad(T,\va, E)$ coincides with $\Bad(T,\va)$  (resp. $\Bad^\times (T,\va)$).  We suspect that the badly approximable set $\Bad(T,\va, E)$ is of full dimension and thus by default $\Bad(T,\psi,\va, E) $ is also of full dimension.  More precisely,  we would expect the following statement to hold.

\medskip

\noindent{\bf Claim 2.}
	{\it Let $T$ be a real, non-singular matrix transformation of the torus $\T^d$.  Suppose that $T$ is integer and all eigenvalues  $\beta_1,\beta_2,\dots,\beta_d$  are of modulus strictly larger than $1$. Then }
$$
\dim_{\rm H} \Bad(T,\va, E) = d \, .
$$

\noindent It is plausible that the claim is true without the assumption that $T$ is integer. However, as mentioned above,  without the integer assumption the problem is currently open even for balls (i.e., when  $E =B(\mathbf{0}, 1)$).   A potentially interesting  starting point towards establishing the claim would be to consider the situation  in which $E$ satisfies Gallagher's property $ ({\boldsymbol{\rm P}}) $ condition and $T$ is diagonal with all eigenvalues strictly larger than $1$.

\bigskip

We now briefly consider another aspect of the badly approximable theory.  Let $ c \in (0,1) $ and  with $\Bad(T,\va, E)$ in mind,  consider the set
$$
\Bad_c(T,\va, E):= \big\{\vx\in  \mathbb{T}^d:    \ T^n(\vx)\not\in   \   \va + c E \   \ \forall   \ n \in \N \big\}  \, .
$$
In short, we  fix the so called badly approximable constant $c(\vx)$ appearing in $\Bad(T,\va, E)$. Then, by definition
$$
\Bad(T,\va, E)  = \bigcup_{0<c<1} \Bad_c(T,\va, E)  \, .
$$
When $E$ is an open set, the corresponding set $\Bad_c(T,\va, E)$ is
often referred to as a survivor set
in the study of (open) dynamical systems. The associated open set $\va+cE$ is referred to as  a hole and we are interested in points whose orbit under $T$ avoid the hole.
 In general, it is difficult to give an exact formula for $\dim_{\rm H}\Bad_c\big(T,\va,E\big)$ and we are interested in determining  how $\dim_{\rm H}\Bad_c\big(T,\va,E\big)$  varies with respect to the positioning of the hole which is governed by  $\va $ and its size which is governed by  $0<c<1$. So with this in mind, Urba\'{n}ski \cite{U1986} proved that  if $T$ is an expanding map of $\T$ and $E=[0, 1]$ then the dimension function $c\mapsto \dim_{\rm H}\Bad_c\big(T,\mathbf{0},E\big)$     is a devil's staircase. The same statement was shown to hold by Nilson \cite{N09} in the case  $T$ is the doubling map, and by Kalla, Kong, Langeveld $\&$ Li \cite{KKLL2020} in the case  $T$ is a $\beta$-transformation with $\beta \in (1,2]$.  The problem of extending the latter to all $\beta >1 $ and indeed to higher dimensions is clearly a natural path to pursue. In the first instance, establishing a statement of the following type would in our opinion represent serious progress.

 \medskip

\noindent{\bf Claim 3.}
	{\it Let $T$ be a real, non-singular matrix transformation of the torus $\T^d$.  Suppose that $T$ is diagonal and all eigenvalues  $\beta_1,\beta_2,\dots,\beta_d$ are strictly larger than $1$. Furthermore,  let $E$ be a subset of $\T^d$ satisfying  Gallaghers's property $ ({\boldsymbol{\rm P}}) $ condition.  Then the dimension function
$$
c\mapsto \dim_{\rm H}\Bad_c\big(T,\mathbf{0},E\big)
$$
is a devil's staircase.
}
\\

\noindent
Indeed, establishing the claim in the  case $T$ is  integer  and $E=B(\mathbf{0}, 1)$  would be most  desirable.

\subsection{Shrinking targets restricted to  manifolds}

For the sake of simplicity, through out this section $T$ will be an integer, non-singular matrix transformation of the torus $\mathbb{T}^d$. Also, we suppose that $T$ is  diagonal with  eigenvalues $1< \beta_1 \le \beta_2 \le \dots  \le \beta_d$.   Finally, given $\psi: \mathbb{R}^+\to \mathbb{R}^+$ we consider the ``basic'' shrinking target set
$$W(T,\psi)= W(T,\psi,0):=\{\vx\in\mathbb{T}^d: T^n(\vx)\in B(0,\psi(n))\ \ \text{for infinitely many}\ n\in\mathbb{N}\} \, . $$

\noindent In view of  Theorems~\ref{corInteger} and \ref{dimresult},  we have a complete description of the ``size'' of $W(T,\psi)$ in terms of both Lebsegue measure and Hausdorff dimension.  Indeed, the former implies    that
 \begin{eqnarray}  \label{May1}
		m_d\big(W(T,\psi)\big)=
		\begin{cases}
			0 &\text{if}\ \sum_{n=1}^\infty \psi(n)^d<\infty\\[2ex]
			1 &\text{if}\ \sum_{n=1}^\infty \psi(n)^d=\infty,
		\end{cases}
	\end{eqnarray}
 while the latter implies that
 \begin{equation}\label{May2}
 \dim_{\rm H} W(T,\psi)=\min_{1\le   i\le   d}\theta_i(\lambda)  \, .
 \end{equation}
 \\
We now add a little twist which is very much  in line with the classical theory of Diophantine approximation on manifolds -- see \cite[Section 6]{durham} for background and further references.
Suppose that the coordinates of the  point $\vx$ in  $\T^d$  are confined by functional relations or equivalently are restricted to a sub-manifold $\cM$ of $\T^d$. We then consider the following two natural problems.

\medskip

~\hspace*{4ex} {\em Problem 1.} To develop a Lebesgue theory for
$ \cM \cap W(T,\psi) $.

\medskip

~\hspace*{4ex} {\em Problem 2.} To develop a Hausdorff theory for
$ \cM \cap W(T,\psi) $.

\medskip

\noindent In short, the aim is to establish analogues of \eqref{May1} and \eqref{May2} for the set
$
\cM \cap W(T,\psi)
$.
The fact that the points $\vx \in \T^d$ of
interest are of dependent variables, which reflects the fact that
$\vx \in {\cal M}$, introduces various difficulties  even in the specific case that ${\cal M}$ is a planar
curve $\cC$.   However, in this case we have recently obtained a reasonably complete theory.  Briefly,
assume that $d=2$ and that the planar curve
$$
\cC=\cC_f:= \{ (x,f(x))  : x \in [0,1] \}
$$
is  the graph of a bi-Lipschitz function $f: [0,1] \to \R$. Let $m$ denote the normalised, induced one dimensional Lebesgue measure on $\cC$. Then, the main measure result in  our forthcoming paper \cite{LLVZ} implies that
\begin{eqnarray*}
		m\big(\cC \cap W(T,\psi)\big)=
		\begin{cases}
			0 &\text{if}\ \sum_{n=1}^\infty \psi(n)^d<\infty\\[2ex]
			1 &\text{if}\ \sum_{n=1}^\infty \psi(n)^d=\infty.
		\end{cases}
	\end{eqnarray*}
As usual, let $\lambda $ be the  lower order at infinity of $\psi$ and recall that $1< \beta_1 \le \beta_2$.   Then, the main dimension result in  \cite{LLVZ} implies the following statement. \emph{Assume that $0 \le \lambda  \le  \log \beta_2$. Then
\begin{equation}\label{dimdim}
\dim \big(\cC \cap W(T,\psi) \big) \, \leq  \,  \frac{1-\frac{\lambda}{\log \beta_2}}{1+ \frac{\lambda}{\log \beta_2}}\, ,
\end{equation}
and we have equality in \eqref{dimdim} for $0 \le \lambda  \le  \log \beta_2 - \log \beta_1$.
Moreover, if $\cC $ is a line with rational slope then we also have equality in \eqref{dimdim} for $ \log \beta_2 - \log \beta_1 < \lambda \le  \log \beta_2  $, conditional on the  validity of the $abc$-conjecture.}

\medskip

\begin{remark}

Let $\cC$  be the diagonal line $L:=\{ (x,x): x \in [0,1]\} $
and   $T$ to be \slv{be} the  diagonal matrix with entries $\beta_1=2$ and $\beta_2=3$.  Then,  a simple consequence of the above dimension result is the following  number theoretic statement which may be of independent  interest:  for $ 0 \leq \tau \leq 1 $ the set
 $$
\left\{ x\in [0,1] \ : \ \max\left\{\|2^nx\|,\|3^n x \|\right\}< 3^{-n \tau} \rm{ \ for \  infinitely \  many \ } n\in\N  \right\}
$$
has Hausdorff dimension $(1-\tau)/(1+ \tau)$. For $\tau > 1 - (\log2/\log3) $,  our proof is conditional on the validity of the $abc$-conjecture.

\end{remark}
\medskip

 Two things are worth mentioning.  Firstly, for  $ \log \beta_2 - \log \beta_1 < \lambda \le  \log \beta_2$ we suspect that we also  have equality in \eqref{dimdim} for all ``bi-Lipschitz'' planar curves  (not just rational lines) and almost certainly the use of the $abc$-conjecture is an overkill.  Secondly, to the best of our knowledge,  beyond the planar case very little seems to be  known and  in our opinion Problems~1~$\&$~2 represent interesting and potentially fruitful avenues of research.

%A further problem is to study the approximation on manifold. One could naturally ask the size of the intersection of $W\big(T,\{E_n\}\big) $ with a manifold in $\mathbb{T}^d$. Let $\psi: \mathbb{R}^+\to \mathbb{R}^+$ be a real positive decreasing function and $\va \in \mathbb{T}^d$. In a forthcoming paper \cite{LLVZ}, we will give the Lebesgue measure and the Hausdorff dimension of the intersection of
%$$W(T,\psi,\va):=\{\vx\in\mathbb{T}^d: T^n(\vx)\in B(\va,\psi(n))\ \ \text{for infinitely many}\ n\in\mathbb{N}\}$$
%with a curve in $\mathbb{T}^d$. The study of the intersection of $W\big(T,\{E_n\}\big) $ with more general manifolds is largely open.

\vspace*{6ex}

\noindent{\it Acknowledgments}. BL was supported partially by NSFC 12271176. We would like to thank J\'er\^ome Buzzi, Gerhard Keller, Beno\^it Saussol, and Baowei Wang for the many useful discussions.  Also we would like thank Yubin He for reading the paper carefully and pointing out various inconsistencies. Almost certainly there are more errors of one form or another and of course we take  full responsibility.
SV  would like to take this opportunity to thank the wonderful Bridget and the dynamic duo Ayesha and  Iona for absolutely everything  over the difficult and weird covid years.   The duo recently turned twenty-one and their curiosity and optimism  remains a marvel -- long may it last champions and   remember those magical words of  Glinda from the land of Oz   ``You've always had the power my dear, you just had to learn it yourself''

%\subsection{lower bound for higher-dimensional case}

\newpage

\begin{appendix}

\begin{center}  { \Large
Appendix:  \  Proof of Theorem~\ref{multiplicative}
 \\ ~ \hspace*{10ex} }
 \\ ~ \vspace*{-5ex} \\
{\large  by \\ ~ \hspace*{10ex}  }
\\ ~ \vspace*{-4ex} \\
{\large   Baowei Wang}
\end{center}

\bigskip

%\subsubsection{Proof of Theorem~\ref{multiplicative} \label{Th-multiplicative}}

We  start with stating two lemmas that we will make use of  during the course of establishing Theorem~\ref{multiplicative}.  As in the main body of the paper, balls  are always with respect to the maximum  norm and thus correspond to a hypercubes.  Indeed, the diameter $d(B)$ of a ball $B$ can equivalently be interpreted as  the side length of a hypercube.

\begin{lemma}{\rm{(\cite[Lemma 1]{BD1978})} \label{BD}}
  Let $d \in \N$ and $\delta$ be a sufficiently small positive number. Then, for any $\va= (a_1, \ldots, a_d)  \in \T^d$ and $s\in (d-1, d)$ the set
  $$H_d(\va,\delta)=\{\vx=(x_1,\dots,x_d)\in \T^d: \|x_1 -a_1\|\cdots \|x_d -a_d\|<\delta \}$$
  has a covering $\mathcal{B}$ by $d$-dimensional balls $B$ such that
  $$\sum_{B\in\mathcal{B}} d(B)^s\ll \delta^{s-d+1},$$
  where $d(B)$ is the length of a side of $U$ and $\ll$ implies an inequality with a factor independent of $\delta$.
\end{lemma}

The above lemma does not precisely correspond to the  Bovey-Dodson statement \cite[Lemma~1]{BD1978}.   However, it is readily verified that  in establishing Lemma~\ref{BD} we can, without loss of generality, ignore  the `shift'   $\va \in \T^d$.  Then, the problem reduces to finding an appropriate cover by balls of the set
$
\{(x_1,\dots,x_d)\in [0,1/2]^d: x_1 \cdots x_d <\delta \}  \, .   $
In short, for this task  the Bovey-Dodson statement  is directly applicable.

\begin{lemma} {\rm(\cite[Corollary 7.12]{F}) \label{cor7.12}}
	Let $F$ be any subset of $\mathbb{R}^d$, and let $E$ be a subset of the $x_d$-axis. Assume that
$$\dim_{\rm H}F\cap L_x\ge   t$$
for all $x\in E$, where $L_x$  is  the plane parallel to all other axis through the point $(0,\dots,0,x)$. Then
$$\dim_{\rm H}F\ge   t+\dim_{\rm H}E.$$
\end{lemma}

\medskip

We now move onto the task of proving  Theorem~\ref{multiplicative}. This will be done  by establishing  the upper and lower bounds  for $ \dim_{\rm H} W^\times (T,\psi, \va)$ separately. Recall, that
$$W^{\times}(T,\psi,\va):=\{\vx\in\mathbb{T}^d: T^n(\vx)\in H(\va,\psi(n))  \ \ \text{for infinitely many}\ n\in\mathbb{N}\}$$
where $ H(\va,\psi(n))  $ is the hyperboloid region  given by \eqref{iona}.

\begin{proposition}\label{upperboundmult}
		Under the setting of  Theorem~\ref{multiplicative},  we have that
		$$\dim_{\rm H} W^\times (T,\Psi,\va) \ \le   \ d-1+\frac{\log|\beta_d|}{\lambda+\log|\beta_d|}.$$
	\end{proposition}

\begin{proof}
Observe that we can re-write $W^\times (T,\Psi,\va)$ as
\begin{equation} \label{limin}
	W^\times (T,\Psi,\va) =\limsup_{n\to \infty}   \ E_n^\times(T,\psi,\va)
\end{equation}
where
$$
E_n^\times(T,\psi,\va):=\{\vx\in\mathbb{T}^d: T^n(\vx)\in H(\va, \psi(n)) \}=\Big\{\vx\in\mathbb{T}^d: \prod_{i=1}^d  \|T_{\beta_i}^nx_i-a_i\| <\psi(n)\Big\}   \,   .    $$
As in the main body of the paper,  $T_{\beta_i} $ is the standard  $\beta$-transformation   with $\beta = \beta_i$ and we do not distinguish between  $\beta$-transformations acting on the unit interval $[0,1)$  or the torus $\T$.
  The proof of the proposition  relies on finding an ``efficient'' covering by balls  of the $\limsup$ set \eqref{limin}.   So with this in mind,   for $n \in \N$, we first obtain an efficient cover of the set $ E_n^\times(T,\psi,\va)$.

 For any $1 \le   i \le   d$, let $\{C_{n,j}^{(i)}  : 1\le j\le N_{i,n}\}$ be the cylinders of order $n$ associated with the transformation $T_{\beta_i} $ .  By definition, these $N_{i,n}$ intervals are disjoint  and cover $\T$. Hence
	%Proposition \ref{fullcylinders} (ii)\& (iii),  $|\beta_i|^n\ll N_i\ll |\beta_i|^n(1+\epsilon)$ for any $\epsilon>0$.
	%By Proposition \ref{fullcylinders} (ii)\& (iii), we know that $T_{\beta_i}^n$ is linear with slope $|\beta_i^n|$ on each $C_{n,j}^i$ and $|\beta_i|^n\ll N_i\ll |\beta_i|^n$.
$$\mathbb{T}^d=\bigcup_{j_1=1}^{N_{1,n}}\cdots \bigcup_{j_d=1}^{N_{d,n}}C_{n,j_1}^{(1)}\times\cdots\times C_{n,j_d}^{(d)}  \, ,  $$ where the $d$-dimensional ``rectangles''  $ C_{n,j_1}^{(1)}\times\cdots\times C_{n,j_d}^{(d)}$  are disjoint.    For $n \in \N$,  let
$$
J_n := \big\{\vj=(j_1, \ldots, j_d):   1 \le j_i \le N_{i,n} \ (1\le i\le d) \big\}  \,
$$
and for $\vj \in J_n $, let
$$E^\times_{n,\mathbf{j}} (T,\psi,\va)   :=\Big\{\vx\in C_{n,j_1}^{(1)}\times\cdots\times C_{n,j_d}^{(d)}: \prod_{i=1}^d \|T_{\beta_i}^nx_i-a_i  \|   <\psi(n)\Big\}  \, .  $$
It follows that
$$
E^\times_{n} (T,\psi,\va)  = \bigcup_{\vj \in J_n }   E^\times_{n,\mathbf{j}} (T,\psi,\va)   \, .
$$

	By Lemma \ref{BD}, with  $\delta=\psi(n)$ and $n$ sufficiently large,  for any $ s \in (d-1,d)$ there exists a covering $\mathcal{B}_n$ of the hyperboloid $H\big(\va,\psi(n)\big)$ by balls $B$ such that
	\begin{equation}\label{cov-estimate}
		\sum_{B\in\mathcal{B}_n}d(B)^s\ll \psi(n)^{s-d+1}.
	\end{equation}
By definition $$E^\times_{n,\mathbf{j}}(T,\psi,\va)=\Big(T^n|_{C_{n,j_1}^{(1)}\times\cdots\times C_{n,j_d}^{(d)}}\Big)^{-1} \left(H\big(\va,\psi(n)\big)\right),$$
and so it follows that
$$E^\times_{n,\mathbf{j}}(T,\psi,\va)\subset \Big(T^n|_{C_{n,j_1}^{(1)}\times\cdots\times C_{n,j_d}^{(d)}}\Big)^{-1}\big(\bigcup_{B\in\mathcal{B}_n}B\big)
=\bigcup_{B\in\mathcal{B}_n}\Big(T^n|_{C_{n,j_1}^{(1)}\times\cdots\times C_{n,j_d}^{(d)}}\Big)^{-1}\left(B\right).$$
 On making use of the fact that for each  $1\le   i \le   d$, the $n$-th iteration of $T_{\beta_i}$ on $C_{n,j_i}^{(i)}$ is an affine function, it can be verified that for any $ B\in\mathcal{B}_n $:
$$R_{n,\mathbf{j}} (B) := \Big(T^n|_{C_{n,j_1}^{(1)}\times\cdots\times C_{n,j_d}^{(d)}}\Big)^{-1}\left(B\right)   $$
corresponds  to  either the empty set  or to a rectangle with  side length  $ |\beta_i|^{-n}  d (B)$    along the $x_i$-th axis.  The upshot is that
	$$E_n^\times(T,\psi,\va)\subset \bigcup_{\vj \in J_n }     \bigcup_{B\in\mathcal{B}_n}R_{n,\mathbf{j}}(B) \, ,$$
	and so for $N$ large enough
	$$W^\times (T,\psi, \va)  \  \subset \  \bigcup_{n=N}^\infty E_n^\times(T,\psi,\va)   \ \subset  \ \bigcup_{n=N}^\infty\bigcup_{j_1=1}^{N_{1,n}}\cdots \bigcup_{j_d=1}^{N_{d,n}} \bigcup_{B\in\mathcal{B}_n}R_{n,\mathbf{j}}(B) \, .$$

For any $ 1 \le i \le d $,   whenever $R_{n,\mathbf{j}} (B)$ is non-empty,  as already  mentioned above  the side length of the rectangle   along the $x_i$-th axis is $ |\beta_i|^{-n}  d (B)$  and by assumption $ |\beta_i|^{-n}  d (B)   \ge  |\beta_d|^{-n}  d (B)$.  We now cover the rectangle  by  balls with  diameter equal to the shortest side length of the rectangle.    A straightforward geometric argument shows that we can find a collection $\cC_n$ of balls with diameter $|\beta_d|^{-n}  d (B)$ that cover
$R_{n,\mathbf{j}}(B)$  with
$$ \# \cC_n  \ \le  \ \prod_{i=1}^d\left(\frac{|\beta_i|^{-n}d(B)}{|\beta_d|^{-n}d(B)}+1\right)
\ = \ \prod_{i=1}^d\left(\frac{|\beta_d|^{n}}{|\beta_i|^{n}}+1\right) \ \le \  2^d\prod_{i=1}^d\frac{|\beta_d|^{n}}{|\beta_i|^{n}}.$$
%Secondly, since $W^\times (T,\psi, \va)=\cap_{N=1}^\infty\cup_{n=N}^\infty E_n^\times(T,\psi,\va)$, for large enough $N$,
%$$W^\times (T,\psi, \va)\subset \bigcup_{n=N}^\infty\bigsqcup_{j_1=1}^{N_1}\cdots \bigsqcup_{j_d=1}^{N_d}\bigcup_{C\in\mathcal{C}}C^*_{n,\mathbf{j}}.$$
 Thus, given $\rho>0$ and on choosing $N$  sufficiently large so that $|\beta_d|^{-n}d(B)<\rho$  for all $B\in \mathcal{B}_n$ and for any $n \ge N$,  it follows from the definition of $s$-dimensional Hausdorff measure  that for any $ s > 0$
\begin{eqnarray} \label{rasq}
 \mathcal{H}^s_\rho(W^\times (T,\psi, \va))
 &\le& \sum_{n=N}^\infty  \  \sum_{\vj \in J_n } \ \sum_{B\in\mathcal{B}_n}  \# \cC_n \
 \left(|\beta_d|^{-n}d(B)\right)^s  \nonumber \\[2ex]
&\le &\sum_{n=N}^\infty\  \sum_{\vj \in J_n } \ \left(2^d\prod_{i=1}^d\frac{|\beta_d|^{n}}{|\beta_i|^{n}}\right)|\beta_d|^{-ns}
\ \sum_{B\in\mathcal{B}_n} d(B)^s.
\end{eqnarray}
Now for any given $\ep > 0$, it follows via \eqref{needlabelB} that for $n$ sufficiently large
$$
\#  J_n \ \le   \ \prod_{i=1}^d   |\beta_i|^{n (1 + \ep)}   \, .
$$
This together with   \eqref{cov-estimate}, \eqref{rasq}   and the assumption that  $ |\beta_d|  \ge   |\beta_i| > 1 $ for  any $ 1 \le i \le d$,
 implies that for any $ s \in (d-1,d)$ and $N$ sufficiently large
\begin{eqnarray*}
\mathcal{H}^s_\rho(W^\times (T,\psi, \va))
&\le&\sum_{n=N}^\infty \ \prod_{i=1}^d|\beta_i|^{n(1+\epsilon)}
\left(2^d\prod_{i=1}^d\frac{|\beta_d|^{n}}{|\beta_i|^{n}}\right)
|\beta_d|^{-ns}\left(\psi(n)\right)^{s-d+1}\\[2ex]
&\le  &2^d\sum_{n=N}^\infty |\beta_d|^{n(d(1+\epsilon)-s)}\left(\psi(n)\right)^{s-d+1}\\[2ex]
&=&2^d\sum_{n=N}^\infty\exp\Big(n\Big(d(1+\epsilon)\log|\beta_d|-
(d-1)\frac{\log\psi(n)}{n}  \\[-1ex] & & \hspace*{30ex}
  -s\Big(\log|\beta_d|
-\frac{\log\psi(n)}{n}\Big) \, \Big) \Big)  \, . \\
\end{eqnarray*}
Now for any $$ s>d-1+\frac{(d\epsilon+1)\log|\beta_d|}{\lambda+\log|\beta_d|}$$ the above exponential sum converges   and so $\mathcal{H}^s(W^\times (T,\psi, \va))  = 0$.  Therefore, it follows from the definition of Hausdorff dimension that   $$\dim_{\rm H} W^\times (T,\psi, \va)\le   d-1+\frac{(d\epsilon+1)\log|\beta_d|}{\lambda+\log|\beta_d|}.$$
Since $ \ep > 0$ is arbitrary, on letting $\epsilon \to 0$ we obtain the desired upper bound for $ \dim_{\rm H} W^\times (T,\psi, \va)$.
	\end{proof}

We now establish the complementary lower bound statement for the Hausdorff dimension of the set $ W^\times (T,\psi, \va) $.
\begin{proposition}\label{lowerboundmult}
		Under the setting of  Theorem~\ref{multiplicative},  we have that
		$$\dim_{\rm H} W^\times (T,\Psi,\va) \ \ge \ d-1+\frac{\log|\beta_d|}{\lambda+\log|\beta_d|}  \, .$$
	\end{proposition}

\begin{proof}

 By Theorem \ref{dimensionalone}, for any $a_d\in K(\beta_d)$ we have that
		$$\dim_{\rm H}W^\times (T_{\beta_d},\psi, a_d)=\dim_{\rm H} W(T_{\beta_d},\psi, a_d)=\frac{\log|\beta_d|}{\lambda+\log|\beta_d|}.$$
Now it is easily verified that for any $x_d\in W^\times (T_{\beta_d},\psi, a_d)$
  $$\big([0, 1)^{d-1}\times W^\times (T_{\beta_d},\psi, a_d)\big)\cap L_{x_d}=[0,1)^{d-1}   \, . $$ Hence, it follows that  $$\dim_{\rm H} \Big(   \big( \, [0, 1)^{d-1}\times W^\times (T_{\beta_d},\psi, a_d) \big)\cap L_{x_d}  \Big) \ge   d-1.$$ Applying Lemma \ref{cor7.12}, we obtain
$$\dim_{\rm H}   \Big( [0, 1)^{d-1}\times W^\times (T_{\beta_d},\psi, a_d)\Big) \ge   d-1+\frac{\log|\beta_d|}{\lambda+\log|\beta_d|}.$$
This together with the fact that
	$$[0, 1)^{d-1}\times W^\times (T_{\beta_d},\psi, a_d) \ \subset  \  W^\times (T,\psi, \va),$$
	implies that  $$\dim_{\rm H} W^\times (T,\psi, \va)\ge   d-1+\frac{\log|\beta_d|}{\lambda+\log|\beta_d|}  \, .$$
\end{proof}

\vspace*{6ex}

\noindent Baowei Wang:  School of Mathematics, Huazhong University of Science and Technology,

\vspace{-2mm}

\noindent\phantom{Baowei Wang: } Wuhan
430074, China.

%\vspace{0mm}

\noindent\phantom{Baowei Wang: } e-mail: bwei\_wang@hust.edu.cn

%%%%%%%%%%%%%%%%%%%%%%%%%%%%%%%%%%%%%%%%%

\end{appendix}

\newpage

%
%\noindent {\bf Data Availability:}
%Data sharing is not applicable to this article as no datasets were generated or analyzed during the current study.
%
%\vspace*{5ex}
%
%\noindent {\bf Conflict of interest:}  On behalf of all authors, the corresponding author states that there is no conflict of interest.
%
%\vspace*{5ex}

 { }

%%%%%%%%%%%%%%%%%%%%%%%%%%%%%%%%%%%%%%%%%%%%%

\vspace*{10ex}

\noindent Bing Li: Department of Mathematics,
South China University of Technology,

\vspace{-2mm}

\noindent\phantom{Bing Li: }Wushan Road 381, Tianhe District, Guangzhou, China

%\vspace{0mm}

\noindent\phantom{Bing Li: }e-mail: scbingli@scut.edu.cn

%%%%%%%%%%%%%%%%%%%%%%%%%%%%%%%%%%%%%%%%%

\vspace{5mm}

\noindent Lingmin Liao: School of Mathematics and Statistics,
Wuhan University, %Universit\'e Paris-Est Créteil CNRS, LAMA, F-94010 Créteil, France
\vspace{-2mm}

\noindent\phantom{Lingmin Liao: }Bayi Road 299, Wuchang District, Wuhan, China
% \&
%Universit\'e Gustave Eiffel, LAMA, F-77447 Marne-la-Vallée, France
%\vspace{0mm}

\noindent\phantom{Lingmin Liao: }e-mail: lmliao@whu.edu.cn %lingmin.liao@u-pec.fr

%%%%%%%%%%%%%%%%%%%%%%%%%%%%%%%%%%%%%%%%%

\vspace{5mm}

\noindent Sanju Velani: Department of Mathematics,
University of York,

\vspace{-2mm}

\noindent\phantom{Sanju Velani: }Heslington, York, YO10
5DD, England.

%\vspace{0mm}

\noindent\phantom{Sanju Velani: }e-mail: sanju.velani@york.ac.uk

%%%%%%%%%*******************

\vspace{5mm}

\noindent Evgeniy Zorin: Department of Mathematics,
University of York,

\vspace{-2mm}

\noindent\phantom{Evgeniy Zorin: }Heslington, York, YO10
5DD, England.

%\vspace{0mm}

\noindent\phantom{Evgeniy Zorin: }e-mail: evgeniy.zorin@york.ac.uk


\begin{thebibliography}{99}


 %\slv{\bibitem{AGK} J.-P. An, L.-F. Guan and D. Kleinbock: Nondense orbits on homogeneous spaces and applications to geometry and number theory. arXiv:2001.05174.}

 \bibitem{Baladi} V. Baladi: Positive transfer operators and decay of correlations.
Advanced Series in Nonlinear Dynamics, 16. World Scientific Publishing Co., Inc., River Edge, NJ, 2000.


\bibitem{durham}
 V. Beresnevich, F. Ram\'{\i}rez and
S. Velani: {\em
Metric Diophantine approximation: aspects of recent work},
in Dynamics and Analytic Number Theory, Editors: Dmitry Badziahin, Alex Gorodnik, and Norbert Peyerimhoff. LMS Lecture Note Series 437, Cambridge University Press, (2016).   1--95.




 \bibitem{BV2006} V. Beresnevich and S. Velani: A mass transference principle and the Duffin-Schaeffer conjecture for Hausdorff measures. Ann. of Math. (2) 164 (2006), no. 3, 971--992.

%\sv{\bibitem{BM2018} J. Bochi and I. D. Morris: Equilibrium states of generalised singular value potentials and applications to affine iterated function systems. Geom. Funct. Anal. 28 (2018), no. 4, 995-1028. }

 \bibitem{BD1978}     J. D. Bovey and M. M. Dodson. The fractional dimension of sets whose simultaneous rational
approximations have errors with a small product. Bull. London Math. Soc. 10 (1978), no.2,
213-218.

 \bibitem{B1973} R. Bowen: Topological entropy for noncompact sets. Trans. Amer. Math. Soc. 184 (1973), 125--136.


\bibitem{BFK2011} R. Broderick, L. Fishman and D. Kleinbock:
Schmidt's game, fractals, and orbits of toral endomorphisms.
Ergodic Theory Dynam. Systems 31 (2011), no. 4, 1095-1107.

 \bibitem{B2005} R. C. Bradley: Basic properties of strong mixing conditions. A survey and some open questions, Probab. Surv. 2 (2005), 107--144.

 \bibitem{BK2014} H. Bruin and C. Kalle: Natural extensions for piecewise affine maps via Hofbauer towers. Monatsh. Math. 175 (2014), no. 1, 65--88.

\bibitem{BW14} Y. Bugeaud and B.-W. Wang: Distribution of full cylinders and the Diophantine properties of the orbits in $\beta$-expansions. J. Fractal Geom. 1 (2014), no. 2, 221-241.

\bibitem{B1997}  J. Buzzi: Intrinsic ergodicity of affine maps in $[0,1]^d$. Monatsh. Math. 124 (1997), no. 2, 97--118.

\bibitem{B2000} J. Buzzi: Absolutely continuous invariant probability measures for arbitrary expanding piecewise $\mathbb{R}$-analytic mappings of the plane, Ergod. Th. Dynam. Sys. 20, (2000) 697--708.

\bibitem{B2001} J. Buzzi: Thermodynamical formalism for piecewise invertible maps: absolutely continuous invariant measures as equilibrium states. Smooth ergodic theory and its applications (Seattle, WA, 1999), 749--783, Proc. Sympos. Pure Math., 69, Amer. Math. Soc., Providence, RI, 2001.

\bibitem{BM2002} J. Buzzi: V. Maume-Deschamps, Decay of correlations for piecewise invertible maps in higher dimensions. Israel J. Math. 131 (2002), 203--220




\bibitem{Dani1988} S. G. Dani:
On orbits of endomorphisms of tori and the Schmidt game.
Ergodic Theory Dynam. Systems 8 (1988), no. 4, 523-529.

%\bibitem{E1962} P. Erd\"{o}s: Representations of real numbers as sums and products of Liouville numbers. Michigan
%Math. J. 9 (1962), 59-60.

\bibitem{Fa08} B. Faller:  Contribution to the ergodic theory of piecewise monotone continuous maps. PhD Thesis, \'Ecole Polytechnique F\'ed\'erale de Lausanne, 2008.

    \bibitem{F}
K. J. Falconer:
{\em Fractal Geometry - Mathematical Foundations and Applications},
(J. Wiley, Chichester. 1990).

 \bibitem{F1999} A.-H. Fan: Decay of correlation for expanding toral endomorphisms, in Dynamical Systems, Proceedings of the International Conference in Honor of Professor Liao Shantao, % Peking University, China, 9 12 August 1998. Ed. L.Wen and Y. P. Jiang,
 World Scientic, 1999.

 \bibitem{FFW01} A.-H. Fan, D.-J. Feng and J. Wu:, Recurrence, dimension and entropy. J. London Math. Soc. (2) 64 (2001), no. 1, 229--244.


 \bibitem{FPS2010} D. F\"{a}rm, T. Persson and J. Schmeling:
Dimension of countable intersections of some sets arising in expansions in non-integer bases. Fund. Math. 209 (2010), no. 2, 157-176.

\bibitem{Fed}
H. Federer: {\em   Geometric Measure Theory.} Sringer-Verlag,
(1969).



\bibitem{FMP}
J. L. Fern\'{a}ndez, M. V. Meli\'{a}n and D. Pestana:
Quantitative mixing results and inner functions  Math. Ann. (2007) 337:233–251.




\bibitem{GallP} P. Gallagher: Metric Simultaneous Diophantine Approximation, Journal of the London Mathematical Society,  (1962), 387--390.

\bibitem{G1959} A. O. Gel'fond:
A common property of number systems,
Izv. Akad. Nauk SSSR. Ser. Mat. 23 (1959) 809--814.

\bibitem{GB1989} P. G\'ora and A. Boyarsky: Absolutely continuous invariant measures for piecewise expanding $C^2$
transformation in $\mathbb{R}^N$, Israel J. Math., 67 (1989), 272--286.



\bibitem{Harman} G.~Harman: \emph{Metric number theory}, LMS Monographs New Series,
vol.~18,
  Clarendon Press, 1998.





\bibitem{hv1}
R. Hill and S. L. Velani:
The Ergodic Theory of Shrinking Targets,
{\em Invent. Math.}, 119 (1995) 175-198.

%\bibitem{hv2}
%R.Hill and S.L.Velani,
%Diophantine Approximation in the Julia Sets of Expanding Rational Maps,
%{\it Inst. Hautes Etudes Sci. Publ. Math.},
%85 (1997) 193-216. .



\bibitem{hvMat}
R. Hill and S. L. Velani:
 The Shrinking Target Problem for Matrix Transformations of Tori,
{\it Proc. Lond. Math. Soc.},
60 (1999)  381–398.


%\sv{\bibitem{HV2001} V. Horita and M. Viana:
%Hausdorff dimension of non-hyperbolic repellers. I. Maps with holes.
%J. Statist. Phys. 105 (2001), no. 5-6, 835-862.}


\bibitem{IS2009} S. Ito and T. Sadahiro: Beta-expansions with negative bases. Integers 9 (A22) (2009), 239-259.


%\sv{ \bibitem{KL2017} A. K\"{a}enm\"{a}ki and B. Li: Genericity of dimension drop on self-affine sets. Statist. Probab. Lett. 126 (2017), 18-25.}


\bibitem{KKLL2020} C. Kalle, D.-R. Kong, N. Langeveld and W.-X. Li:
The $\beta$-transformation with a hole at 0.
Ergodic Theory Dynam. Systems 40 (2020), no. 9, 2482–2514.

 \bibitem{K1978}  G. Keller: Piecewise monotonic transformations and exactness. Seminar on Probability (Rennes, 1978).
Universit\'e de Rennes, Rennes, 1978, Exp. No. 6, p. 32.

 \bibitem{K1979} G. Keller: Ergodicit\'e et mesures invariantes pour les transformations dilatantes par morceaux
d'une r\'egion born\'ee du plan, C. R. Acad. Sci. Paris S\'er. A-B, 289 (1979), A625--A627.

 \bibitem{KPhD} G. Keller: Propri\'et\'e ergodique des endomorphismes dilatants, $C^2$ par morceaux, des r\'egions born\'ees du plan. Thesis, Universit\'e de Rennes, 1979.

\bibitem{KKP} M. Kirsebom, P. Kund and T. Persson: On shrinking targets and self-returning points, preprint, arXiv:2003.01361.




%\sv{\bibitem{KM} D. Kleinbock and S. Mirzadeh: On the dimension drop conjecture for diagonal flows on the space of lattices, arXiv:2010.14065.v3.}
%\bibitem{L1964} V. P. Leonov, Some applications of higher semi-invariants to the theory of stationary random processes, Izdat. "Nauka'', Moscow 1964, 67 pp.



\bibitem{Dorsa} B. Li, L. Liao, S. Velani, B.-W. Wang  and E. Zorin: Tentative title: Simultaneous weighted Diophantine approximation with ``rates'' of infinte lower order,
    in preparation.

\bibitem{LLVZ} B. Li, L. Liao, S. Velani and E. Zorin: The shrinking target problem for matrix transformations of tori: developing a manifold theory, in preparation.

\bibitem{LL2018} Y.-Q. Li and B. Li: Distributions of full and non-full words in beta-expansions. J. Number Theory 190 (2018), 311-332.

  \bibitem{LS12} L. Liao and W. Steiner: Dynamical properties of the negative beta-transformation. Ergodic Theory Dynam. Systems 32 (2012), no. 5, 1673--1690.

\bibitem{MU03} D. Mauldin and M. Urbanski, \emph{Graph directed Markov system: geometry and dynamics of limit sets}, Cambridge Tracts in Mathematics, 148, Cambridge University Press, 2003.

\bibitem{MAT}
P. Mattila: Geometry of sets and measures in Euclidean space,
\emph{CUP, Cambridge studies in advance mathematics} {\bf 44} (1995)

\bibitem{N09}
J. Nilsson: On numbers badly approximable by dyadic rationals. Israel J. Math., 171 (2009), 93--110.





 \bibitem{P60} W. Parry: On the $\beta$-expansions of real numbers, Acta Math. Acad. Sci. Hungar. 11 (1960) 401--416.

 \bibitem{PY1998} M. Pollicott and M. Yuri: Dynamical systems and ergodic theory. London Mathematical Society Student Texts, 40. Cambridge University Press, Cambridge, 1998.

 \bibitem{PU2010} F. Przytycki and M. Urba\'nski: Conformal fractals: ergodic theory methods. London Mathematical Society Lecture Note Series, 371. Cambridge University Press, Cambridge, 2010.

% \bibitem{P2017} I. V. Podvigin, Estimates for correlations in dynamical systems: from Hölder functions to arbitrary observables. (Russian) Mat. Tr. 20 (2017), no. 2, 90–119; translation in Siberian Adv. Math. 28 (2018), no. 3, 187–206.

\bibitem{R57} A. R\'{e}nyi: Representations for real numbers and their ergodic properties, Acta Math. Acad. Sci. Hungar. 8 (1957) 477--493.


% \bibitem{R1949} V. A. Rokhlin, On endomorphisms of compact commutative groups, Izv. Akad. Nauk SSSR Ser. Mat. 13 (1949), 329--340.

 \bibitem{R1961} V. A. Rokhlin:
Exact endomorphisms of a Lebesgue space, Izv. Akad. Nauk SSSR Ser. Mat. 25 1961 499--530.


\bibitem{Rynne} B. Rynne: Hausdorff dimension and generalized Diophantine approximation, Bull. Lond. Math.
Soc. 30 (1998), 365–376.




\bibitem{P1967} W. Philipp: Some metrical theorems in number theory, Pacific J. Math., 20 (1967), 109--127.

 \bibitem{S2000} B. Saussol: Absolutely continuous invariant measures for multidimensional expanding maps, Israel
J. Math., 116 (2000), 223--248.


%\slv{\bibitem{Schmidt1966} Schmidt, Wolfgang M.
%On badly approximable numbers and certain games.
%Trans. Amer. Math. Soc. 123 (1966), 178-199.
%}




\bibitem{SW2013} L.-M. Shen and B.-W. Wang: Shrinking target problems for beta-dynamical system. Sci. China Math. 56 (2013), no. 1, 91-104.

 \bibitem{S1973} M. Smorodinsky: $\beta$-automorphisms are Bernoulli shifts, Acta Math. Acad. Sci. Hung.,
24 (1973), 273--278.


\bibitem{Sprindzuk}
V.~G. Sprind{\v z}uk: \emph{Metric theory of {D}iophantine
approximations}.
  John Wiley, 1979, Translated by R.~A.~Silverman.

\bibitem{TV} S. Troubetzkoy and P. Varandas: The role of continuity and expansiveness on leo and periodic
specification properties. (2020),  hal-02557857v2.

\bibitem{T2000}  M. Tsujii: Absolutely continuous invariant measures for piecewise real-analytic expanding maps
on the plane, Comm. Math. Phys., 208 (2000), 605--622.

\bibitem{T2001} M. Tsujii: Absolutely continuous invariant measures for expanding piecewise linear maps, Invent.
Math., 143 (2001), 349--373.

\bibitem{U1986} M. Urba\'{n}ski: On Hausdorff dimension of invariant sets for expanding maps of a circle. Ergodic Theory Dynam. Systems 6 (1986), no. 2, 295-309.

\bibitem{W1979} G. Wagner:
The ergodic behaviour of piecewise monotonic transformations.
Z. Wahrsch. Verw. Gebiete 46 (1979), no. 3, 317--324.

\bibitem{W-GTM79} P. Walters: An introduction to ergodic theory. Graduate Texts in Mathematics, 79. Springer-Verlag, New York-Berlin, 1982.
	
\bibitem{WW2021}  B.-W. Wang and J. Wu: Mass transference principle from rectangles to rectangles in Diophantine approximation, Math. Ann. 381 (2021), no. 1-2, 243-317.

\bibitem{YW2019}   Q.-Q. Yang and S.-L. Wang:
Metric and dimensional properties of the badly approximable set for beta-transformations.
Fractals 27 (2019), no. 3, 1950025, 9pp.





%\bibitem{BWpre} \red{B.-W. Wang ??? :   Multiplicative Diophantine approximation in dynamical systems, Pre-print (draft). }

%Mass transference principle from rectangles to rectangles in Diophantine approximation, arXiv:1909.00924, 2019.
\end{thebibliography}
\end{document}